\documentclass[a4paper,11pt]{article}

\usepackage{verbatim}
\usepackage[T1]{fontenc}
\usepackage{setspace}
\usepackage{indentfirst}

\usepackage{geometry}
\usepackage[hidelinks]{hyperref}

\usepackage{amsthm}
\usepackage{amsmath}
\usepackage{amsfonts}
\usepackage{amssymb}
\usepackage{mathdots}
\usepackage{mathtools}

\usepackage{times}
\usepackage{bm}
\usepackage{bbm}
\usepackage{dutchcal}

\usepackage{pgfplots}
\pgfplotsset{compat=1.14}
\usepackage{tikz-cd}
\usepackage{tikz}
\usetikzlibrary{matrix,arrows,decorations.pathmorphing,mindmap,fit,math}
\usetikzlibrary{shapes.geometric}
\usetikzlibrary{shapes.multipart}
\usepackage[usestackEOL]{stackengine}
\usepackage{fancybox}
\usepackage{caption}
\usepackage{subcaption}
\usepackage{float}
\usepackage[allcommands]{overarrows}

\newcommand{\sbt}{\,\begin{picture}(-1,1)(-1,-3)\circle*{3}\end{picture}\ }

\usepackage[toc,page]{appendix}

\DeclareMathOperator{\id}{Id}
\DeclareMathOperator{\hw}{hw}

\DeclareMathOperator{\SL}{SL}
\DeclareMathOperator{\SU}{SU}

\DeclareMathOperator{\card}{Card}

\DeclareMathOperator{\GL}{GL}
\DeclareMathOperator{\diag}{diag}

\newcommand{\pr}{\mathrm{pr}}
\newcommand{\Mat}{\mathrm{Mat}}
\newcommand{\op}{\mathrm{op}}
\newcommand{\edge}{\mathrm{edge}}

\newcommand{\R}{{\mathbb R}}
\newcommand{\C}{{\mathbb C}}

\newcommand{\N}{{\mathbb N}}

\newcommand{\T}{{\mathbb T}}

\newcommand{\B}{\mathcal{B}}
\newcommand{\la}{\lambda}

\usepackage{theoremref}
\newtheorem{theorem}{Theorem}[section]

\newtheorem{proposition}[theorem]{Proposition}
\newtheorem{lemma}[theorem]{Lemma}
\newtheorem{corollary}[theorem]{Corollary}

\newtheorem{conj}{Conjecture}

\theoremstyle{definition}
\newtheorem{definition}[theorem]{Definition}
\newtheorem{example}[theorem]{Example}

\theoremstyle{remark}
\newtheorem{remark}[theorem]{Remark}
\newtheorem*{notation}{Notation}
\makeatletter
\renewcommand*\env@matrix[1][\arraystretch]{%
  \edef\arraystretch{#1}%
  \hskip -\arraycolsep
  \let\@ifnextchar\new@ifnextchar
  \array{*\c@MaxMatrixCols c}}
\makeatother

\newcommand\nnfootnote[1]{%
  \begin{NoHyper}
  \renewcommand\thefootnote{}\footnote{#1}%
  \addtocounter{footnote}{-1}%
  \end{NoHyper}
}

\providecommand\given{} 
\newcommand\givensymbol[1]{%
  \nonscript\;\delimsize#1\allowbreak\nonscript\;\mathopen{}%
}
\DeclarePairedDelimiterX\Set[1]\{\}{%
  \renewcommand\given{\givensymbol{\vert}}%
  #1%
}

\usepackage{tikzsymbols}

\begin{document}
\title{Multiple Horn problems for planar networks\\ and invertible matrices}
\author{ Anton Alekseev \and Arkady Berenstein \and Anfisa Gurenkova \and Yanpeng Li}

\newcommand{\Addresses}{{
  \bigskip
  \footnotesize

  \textsc{Section of Mathematics, University of Geneva, rue du Conseil-G\'en\'eral 7-9, 1205 Geneva, Switzerland}\par\nopagebreak
  \textit{E-mail address}: \texttt{Anton.Alekseev@unige.ch}

  \medskip
  \textsc{Department of Mathematics, University of Oregon, Eugene, OR 97403, USA}\par\nopagebreak
  \textit{E-mail address}: \texttt{arkadiy@uoregon.edu}

   \medskip
  \textsc{Faculty of Mathematics, HSE University, Usacheva 6, 119048 Moscow, Russia; \, 
  Igor Krichever Center for Advanced Studies, Skoltech, Bolshoy Boulevard 30, bld.1, 121205 Moscow, Russia}\par\nopagebreak
  \textit{E-mail address}: \texttt{Anfisa.Gurenkova@skoltech.ru}

  \textsc{Department of Mathematics, Sichuan University, No.24 South Section 1, Yihuan Road, Chengdu, China}\par\nopagebreak
  \textit{E-mail address}: \texttt{yanpeng.li@scu.edu.cn}

}}

\date{}
\maketitle

\nnfootnote{\today}
\nnfootnote{\emph{Keywords:} Horn problem, planar networks, octahedron recurrence, positive varieties}

\begin{abstract}
    The multiplicative multiple Horn problem is asking to determine possible singular values of the combinations $AB, BC$ and $ABC$ for a triple of invertible matrices $A,B,C$  with given singular values. There are similar  problems for eigenvalues of sums of Hermitian matrices (the additive problem), and for maximal weights of multi-paths in concatenations of planar networks (the tropical problem). 

    For the planar network multiple Horn problem, we establish necessary conditions, and we conjecture that for large enough networks they are also sufficient. These conditions are given by the trace equalities and rhombus inequalities (familiar from the hive description of the classical Horn problem), and by the new set of tetrahedron equalities. Furthermore, if one imposes Gelfand-Zeitlin conditions on weights of planar networks, tetrahedron equalities turn into the octahedron recurrence from the theory of crystals. We give a geometric interpretation of our results in terms of positive varieties with potential. In this approach, rhombus inequalities follow from the inequality $\Phi^t \leqslant 0$ for the tropicalized potential, and tetrahedron equalities are obtained as tropicalization of certain Plücker relations.

    For the multiplicative problem, we introduce a scaling parameter $s$, and
    we show that for $s$ large enough (corresponding to exponentially large/small singular values) the Duistermaat-Heckman measure associated to the multiplicative problem concentrates in a small neighborhood of the octahedron recurrence locus.

\end{abstract}

\tableofcontents

\section{Introduction}

The Horn problem is a classical problem of Linear Algebra asking to determine all possible triples $(\lambda, \mu, \rho)$ of ordered sets of real numbers which can be realized as eigenvalues of $n \times n$ Hermitian matrices $a, b$ and $c=a+b$. The answer to this question was conjectured by Horn \cite{H62}, and the Horn Conjecture was settled in the positive by Klyachko \cite{K} and by Knutson-Tao \cite{KT}. There are several ways to present the answer. One of them uses the Knutson-Tao-Woodward combinatorics of hives \cite{KTW}, and it includes the trace equality based on 
$$
{\rm Tr}(a) + {\rm Tr}(b)={\rm Tr}(a+b) = {\rm Tr}(c),
$$
and rhombus inequalities (for more details, see the body of the paper). Note that solutions of the classical Horn problem form a convex polyhedral cone. 

Surprisingly, there are two other problems which resemble the Horn problem, and which have exactly the same answer. The first one is the problem of determining possible singular values of invertible matrices $A, B$ and $C=AB$ (the multiplicative problem). Its equivalence to the classical Horn problem is the Thompson Conjecture, and it was settled in the positive by Klyachko \cite{K}. The second one is the question about maximal weights of multi-paths in the planar networks $\Pi_1$ and $\Pi_2$, and in their concatenation $\Pi_1 \circ \Pi_2$. Its equivalence to the Horn problem was established in \cite{APS-planar}.

One can introduce a scaling parameter $s \in \mathbb{R}_{\neq 0}$ in the multiplicative Horn problem such that the $s \to 0$ limit corresponds to the classical (also called additive) Horn problem, and the $s \to \infty$ limit gives the planar network problem. The equivalences described above can be summarized by saying that the solution of the multiplicative problem is independent of $s$, including the limits  $s \to 0$ and $s \to \infty$.

In this paper, we consider several versions of the multiple Horn problem. The additive multiple Horn problem is a question of finding all possible 6-tuples $(\lambda, \mu, \nu, \rho, \sigma, \tau)$ of ordered sets of real numbers which can be realized as eigenvalues of $n \times n$ Hermitian matrices $a, b, c, a+b, b+c$ and $a+b+c$. The multiple Horn problem is related (but cannot be reduced to) to the classical Horn problem. 

The multiple Horn problem also admits the multiplicative and the planar network versions. In the multiplicative version, one can again introduce a scaling parameter $s \in \mathbb{R}_{\neq 0}$, and it is still true that in the limit $s \to 0$ one gets the additive problem, and one can speculate that in the limit $s \to \infty$ one gets the planar network problem (see the discussion below). However, now solutions of the multiplicative problem depend on $s$, and three versions of the multiple Horn problem are no longer equivalent to each other. In this paper, we focus on the planar network and multiplicative problems. 

For the planar network multiple Horn problem, we prove a necessary condition for 
$(\lambda, \mu, \nu, \rho, \sigma, \tau)$ to be realized as maximal multi-paths in the concatenation of three planar networks (see Theorems \ref{theorem PN rhombi inequalities} and \ref{theorem PN tetrahedra equalities}).
We conjecture that for large enough networks these conditions are also sufficient. Our conditions include the trace equalities, the rhombus inequalities, and the new tetrahedron equalities. These tetrahedron equalities take the form
$$
\alpha \leqslant {\rm max}\{\beta, \gamma\}, \hskip 0.3cm
\beta \leqslant {\rm max}\{\gamma, \alpha\}, \hskip 0.3cm
\gamma \leqslant {\rm max}\{\alpha, \beta\}
$$
for certain combinations of weights $\alpha, \beta, \gamma$. These inequalities, trace ``tropical lines'' (see Fig.\ref{fig:tropline}) in the space of weights.
\begin{figure}[H]
    \centering
    \begin{tikzpicture}[scale=0.5]
        \draw[thick,->] (0,0)--(5,0);
        \draw[thick,->] (0,0)--(0,5);
        \draw[thick] (0,2)--(2,2);
        \draw[thick] (2,0)--(2,2);
        \draw[thick] (2,2)--(4,4);
        \node at (5.5,0) {$\alpha$};
        \node at (-0.5,4.2) {$\beta$};
        \node at (2,-0.5) {$\gamma$};
        \node at (-0.5,2) {$\gamma$}; 
    \end{tikzpicture}\caption{A tropical line: $\max\{\alpha, \beta, \gamma\}$.}\label{fig:tropline}
\end{figure}
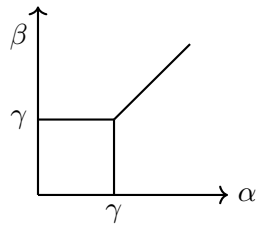
We show (see Theorems \ref{theorem PN octahedron recurrence} and \ref{pro:geoocttotropoct}) that by choosing so-called standard planar networks and by imposing the Gelfand-Zeitlin condition (see Definition \ref{definition multi-GZ condition for PN}) on their weights, the tetrahedron equalities reduce to the octahedron recurrence  from the theory of crystals:
$$
\gamma= {\rm max}\{\alpha, \beta\}.
$$
In that case, the solution of the planar network multiple Horn problem is completely determined by the trace equalities and by the rhombus inequalities, similar to the classical Horn problem.

We give a geometric interpretation of these results in terms of a certain positive variety $\mathcal{M}_3$ with potential $\Phi_3$ and with the distinguished set of multi-corner minors $M_{i,j,k}$. In terms of these data, the rhombus inequalities are equivalent to the inequality
$\Phi_3^t \leqslant 0$ for the tropicalized potential (see Theorem \ref{thm:allrhombi}), and the octahedron recurrence follows by tropicalization of Pl\"ucker relations for $M_{i,j,k}$'s (see Theorem \ref{Thm:octrec}).

For the multiplicative multiple Horn problem, we conjecture that for $s$ large enough (that is, singular values of all matrices are exponentially large/small) its solutions are contained in a small neighborhood of the set determined by trace equalities and rhombus and tetrahedron equalities, and we confirm this conjecture for $2 \times 2$ matrices (see Appendix A). The solution sets of the additive and multiplicative Horn problems are equipped with natural Duistermaat-Heckman measures induced by the identification of Hermitian matrices with the dual $\mathfrak{u}(n)^*$ of the Lie algebra $\mathfrak{u}(n)$, and of lower-triangular matrices with the Poisson-Lie dual ${\rm U}(n)^*$ of the Poisson-Lie group ${\rm U}(n)$. We show that for $s$ large enough the Duistermaat-Heckman measure of the multiplicative problem concentrates in a small neighborhood of the set of solutions of the octahedron recurrence (see Theorem \ref{theorem: symplectic convergence}).

The structure of the paper is as follows. In Section \ref{sec:background}, we give a brief review of the three versions of the Horn problem, introduce multiple Horn problems, and give a statement of our main results. In Section \ref{sec:planar}, we state and prove the results on the planar network multiple Horn problem. In that section, our main tool is combinatorics of weighted planar networks. In Section \ref{sec:geometry}, we define the variety $\mathcal{M}_3$, its potential $\Phi_3$, and multi-corner minors $M_{i,j,k}$, and we prove the results on the relation between rhombus and tetrahedron equalities and tropicalizations $\Phi_3^t$ and $M_{i,j,k}^t$. Our tools in this section are the theory of positive varieties with potential and elementary matrix calculus. In Section \ref{section: multiplicative problem}, we prove the concentration of the Duistermaat-Heckman measure of the multiplicative problem in a neighborhood of solutions of the octahedron recurrence. Our tools here include tropical calculus and geometry of Poisson-Lie groups, and they are inspired by the analysis of the Horn problem in \cite{APS-symplectic}.

\vskip 0.2cm

{\bf Acknowledgments.} We are grateful to M. Christandl for introducing us to the multiple Horn problem. We would like to thank I. Davydenkova, S. Fomin, T. C. Fraser, A. Knutson, J-H. Lu, A. Szenes, and A. Volkov for interesting discussions. 

The final version of this paper was prepared while AB and YL were visiting the University of Geneva whose hospitality and support is gratefully acknowledged.
Research of AA was supported in part by the grants 08235 and 20040, and the National Center for Competence in
Research SwissMAP of the Swiss National Science Foundation, and by the award of the Simons Foundation to the Hamilton Mathematics
Institute of the Trinity College Dublin under the program ``Targeted Grants to Institutes''. 
Research of AB was partially supported by the Simons Foundation Collaboration Grant No. 636972.
Research of AG was partially supported by the Russian Science Foundation under project 23-11-00150.
YL's research was partially supported by the National Natural Science Foundation of China (No. 12201438 and No. 12261131498). 

\section{Background and main results}      \label{sec:background}

The purpose of this section is to give some background on the Horn problem, introduce three versions of the multiple Horn problem, and to state our main results.


\subsection{Horn problems}

In this section, we recall three versions of the Horn problem, and the hive presentation of Horn inequalities.

\subsubsection{Three Horn problems}
We begin by stating three equivalent problems concerning eigenvalues of sums of Hermitian matrices, singular values of products of invertible matrices, and maximal multi-paths in weighted planar networks.

{\em Additive problem.} The classical (also called additive) Horn problem is the following question: determine the set ${\rm Horn}^{\rm add}$ of possible triples of ordered sets of eigenvalues of Hermitian matrices $a, b$ and $c=a+b$:
$$
{\rm Horn}^{\rm add} = \{ (\lambda, \mu, \nu)\mid \exists \, a, b \in \mathcal{H},
\lambda = {\rm eig}(a), \mu={\rm eig}(b), \nu ={\rm eig}(a+b)\}.
$$
Here $\lambda, \mu, \nu \in \mathbb{R}^n$ are nonincreasing sequences of eigenvalues ({\em e.g.} 
$\lambda_1 \geqslant \lambda_2 \geqslant \dots \geqslant \lambda_n$), $\mathcal{H}$ is the space of Hermitian $n\times n$ matrices, and ${\rm eig}\colon \mathcal{H} \to \mathbb{R}^n$ is the function assigning to a Hermitian matrix its ordered set of eigenvalues.

Sometimes, it is more convenient to fix eigenvalues of matrices $a$ and $b$, and to ask a question of possible eigenvalues of $c=a+b$. We will denote the corresponding set by ${\rm Horn}^{\rm add}_{\lambda, \mu} \subset \mathbb{R}^n$. One can identify the space of Hermitian matrices with the dual of the unitary Lie algebra $\mathcal{H} \cong \mathfrak{u}(n)^*$. This identification turns conjugacy classes $\mathcal{H}_\lambda\subset \mathcal{H}$ into coadjoint orbits of ${\rm U}(n)$. These conjugacy classes carry canonical symplectic forms $\omega_\lambda$ and Liouville measures $\mu^\mathcal{H}_\lambda$. 
Pushforward under the map $\pi_{\lambda, \mu}\colon (a,b) \mapsto \nu$ gives rise to Duistermaat-Heckman measures on ${\rm Horn}^{\rm add}_{\lambda, \mu}$:
$$
{\rm DH}^{\rm add}_{\lambda, \mu} = (\pi_{\lambda, \mu})_*(\mu^\mathcal{H}_\lambda \times \mu^\mathcal{H}_\mu).
$$

{\em Multiplicative problem.} Denote by $G={\GL}_n$ the general linear group over $\mathbb{C}$. For $s \in \mathbb{R}_{\neq 0}$,  one defines the multiplicative Horn problem as follows:
$$
{\rm Horn}^{\rm mult}(s) = \{ (\lambda, \mu, \nu)\mid \exists \, A, B \in G, 
e^{s\lambda}= {\rm sing}(A), e^{s \mu}={\rm sing}(B), e^{s \nu} ={\rm sing}(AB)\}.
$$
Here $e^{s \lambda} =(e^{s \lambda_1}, \dots, e^{s \lambda_n})$ and 
${\rm sing}(A) = {\rm eig}(AA^*)^{1/2}$ are singular values of $A$. We use the shorthand notation ${\rm Horn}^{\rm mult}$ for the case of $s=1$, and the notation ${\rm Horn}^{\rm mult}_{\lambda, \mu}(s)$ for fixed $\lambda$ and $\mu$.

Recall that the group $\GL_n$ admits Iwasawa decompositions
$G=K(AU_-)=(AU_-)K$, where $K={\rm U}(n)$, $A$ consists of diagonal matrices with positive entries, and $U_-$ consists of lower-triangular matrices with identity on the diagonal. Denote by $\mathcal{B}(n)=AU_- \subset B_-$. Recall that there is a group isomorphism $\mathcal{B}(n)\cong {\rm U}(n)^*$ between $\mathcal{B}(n)$ and the dual Poisson-Lie group ${\rm U}(n)^*$ of ${\rm U}(n)$.

Singular values of a matrix are invariant under left and right multiplication by elements of $K={\rm U}(n)$. Hence, the action of $K^3$ on $G^2$ given by formula
\begin{equation}       \label{intro:action_3on2}
    (u_1, u_2, u_3): (A, B) \mapsto (u_1 A u_2^{-1},
    u_2 B u_3^{-1})
\end{equation}
preserves singular values of $A, B$ and $AB$. By using this action, one can bring every pair $(A,B)$ to $\mathcal{B}(n)^2$, and to obtain an isomorphism
$$
{\rm Horn}^{\rm mult}(s) \cong \{ (\lambda, \mu, \nu)\mid \exists \, A, B \in \mathcal{B}(n), 
e^{s\lambda}= {\rm sing}(A), e^{s \mu}={\rm sing}(B), e^{s \nu} ={\rm sing}(AB)\}.
$$

The group $\mathcal{B}(n)$ carries a canonical Poisson bracket \cite{LW}, and its symplectic leaves are formed by elements with fixed singular values. Therefore, the sets ${\rm Horn}^{\rm mult}_{\lambda, \mu}(s)$ carry Duistermaat-Heckman measures
$$
{\rm DH}^{\rm mult}_{\lambda, \mu}(s) = (\pi_{\lambda, \mu}(s))_*(\mu_\lambda(s) \times \mu_\mu(s)),
$$
where $\pi_{\lambda, \mu}(s): (A, B) \mapsto \nu$.
Furthermore, one can introduce a Poisson space
$$
\mathcal{P}_2 = G^2/K^3,
$$
where the action is given by equation \eqref{intro:action_3on2}.
Functions 
$$
\lambda(A,B) =s^{-1} \log({\rm sing}(A)), \hskip 0.3cm
\mu(A,B)=s^{-1} \log({\rm sing}(B)), \hskip 0.3cm
\nu(A,B)=s^{-1} \log({\rm sing}(AB))
$$
descend to $\mathcal{P}_2$. A point $(\lambda, \mu, \nu)$ belongs to ${\rm Horn}^{\rm mult}(s)$ if and only if its pre-image in $\mathcal{P}_2$ is nonempty.

{\em Tropical problem.} Finally, for a pair of planar networks $\Pi_{1,2}$ of rank $n$ (with $n$ sources and $n$ sinks) we define the set
\[
    {\rm Horn}^{\Pi_1, \Pi_2} = \Set*{(\lambda, \mu, \nu) \given 
    \begin{aligned}
        &\exists \, w_i \in \mathbb{T}^{\Pi_i}, L(\lambda) = m(\Pi_1, w_1), L(\mu) = m(\Pi_2, w_2), \\
        &L(\nu) = m(\Pi_1 \circ \Pi_2, w_1 \circ w_2)
    \end{aligned}}.
\]
Here $\mathbb{T}=\mathbb{R} \cup \{ - \infty\}$ is the tropical semiring,
the functions $w_i$ are assigning weights to edges of the two networks,
$\Pi_1 \circ \Pi_2$ is the concatenation of $\Pi_1$ and $\Pi_2$, $L$ is the linear map defined by formula
$$
L\colon (\lambda_1, \lambda_2, \dots, \lambda_n) \mapsto (\lambda_1, \lambda_1 + \lambda_2, \dots, \lambda_1 + \dots + \lambda_n),
$$
and $m(\Pi, w)$ is an element of $\mathbb{T}^n$ whose $i$-th component is
the maximal weight of an $i$-path (a union of $i$ paths not touching each other) in $\Pi$. 

An example of large enough planar networks is given by standard networks (see Fig.\ref{fig:stplanar}). In that case, one can also impose the Gelfand-Zeitlin condition on the weights, and this doesn't affect the set of solutions of the planar network Horn problem (see \cite{APS-planar}).

\subsubsection{Comparison of Horn problems}

Surprisingly, the three Horn problems described above are intimately related to each other. This relation can be summarized in the following theorem:

\begin{theorem} \label{intro:thm_Horn}
For all $s \in \mathbb{R}_{\neq 0}$, we have
\begin{equation}     \label{eq:horn=horn}
{\rm Horn}^{\rm mult}(s) = {\rm Horn}^{\rm add},
\end{equation}
and for all $\lambda, \mu$ the corresponding Duistermaat-Heckman measures coincide:
\begin{equation}      \label{eq:DH=DH}
    {\rm DH}^{\rm mult}_{\lambda, \mu}(s) = {\rm DH}^{\rm add}_{\lambda, \mu}.
\end{equation}
For all $\Pi_1, \Pi_2$, we have
\begin{equation}    \label{eq:horn=trop}
{\rm Horn}^{\Pi_1, \Pi_2} \subset {\rm Horn}^{\rm add},
\end{equation}
and for $\Pi_1, \Pi_2$ sufficiently large networks ({\em e.g} standard networks) this inclusion is an equality.
\end{theorem}

Theorem \ref{intro:thm_Horn} is a combination of several results. 
The equality \eqref{eq:horn=horn} is the Thompson conjecture, now  Klyachko's theorem in \cite{K}. In fact, it was established for $s=1$, and it follows for all the other values of $s\neq 0$. Indeed, observe that ${\rm eig}(sa) = s\, {\rm eig}(a)$. This implies that 
the set ${\rm Horn}^{\rm add}$ is invariant under dilations
$(a, b, c=a+b) \mapsto (sa, sb, sc)$. Since ${\rm Horn}^{\rm mult}(1)={\rm Horn}^{\rm add}$, we conclude that
$$
{\rm Horn}^{\rm mult}(s)=s \, {\rm Horn}^{\rm add} = {\rm Horn}^{\rm add}.
$$
The equality of Duistermaat-Heckman measures was established in \cite{APS-symplectic}, and the
inclusion \eqref{eq:horn=trop} and its saturation for sufficiently large networks was established in \cite{APS-planar}.

\subsubsection{Solution of the Horn problem}

The complete description of the set ${\rm Horn}^{\rm add}$ was conjectured by Horn \cite{H62}, and it was established by Klyachko \cite{K}, and by Knutson-Tao \cite{KT} and Knutson-Tao-Woodward \cite{KTW}. There are several descriptions of this set, and we will focus on the description in terms of hives. It can be summarized in the following theorem:

\begin{theorem}      \label{intro:KKT}
    A triple $(\lambda, \mu, \nu)$ belongs to the set ${\rm Horn}^{\rm add}$ if and only if it satisfies the trace equality
    \begin{equation}      \label{eq:trace}
    \sum_{i=1}^n \lambda_i + \sum_{j=1}^n \mu_j = \sum_{k=1}^n \nu_k,
    \end{equation}
    and the rhombus inequalities described below.
\end{theorem}

The trace equality is equivalent to additivity of the trace:
$$
{\rm Tr}(a) + {\rm Tr}(b) = {\rm Tr}(a+b) = {\rm Tr}(c).
$$
In order to describe the rhombus inequalities, one draws a graph represented by an equilateral triangle split into smaller equilateral triangles. The components of 
$L(\lambda), \sum \lambda_i + L(\mu), L(\nu)$ are placed in the vertices on the outer boundary of the big triangle (see Fig.\ref{intro:tri}). The triple $(\lambda, \mu, \nu)$ satisfies the rhombus inequalities if there exist real numbers $m_{i,j}$ that can be placed in the interior vertices of the graph such that the inequality
$$
a_1 + a_2 \geqslant b_1 +b_2
$$
is satisfied for each small rhombus (of any orientation), see Fig.\ref{fig:intro-rhom}. Here $a_{1,2}$ and $b_{1,2}$ are the values $m_{i,j}$ corresponding to the vertices of the rhombus.
Note that the set ${\rm Horn}^{\rm add}$ is a convex polyhedral cone, and the sets
${\rm Horn}^{\rm add}_{\lambda, \mu}$ are convex polytopes. One can view these facts as a manifestation of the Kirwan's Covexity Theorem in Symplectic Geometry \cite{Ki}.
\begin{figure}[H]
    \centering
    \begin{subfigure}[t]{0.6\textwidth}
    \centering
    \begin{tikzpicture}[every text node part/.style={align=center}]
    \draw (0:0)--(0:1.5)--(300:1.5)--(0:0);
    \draw[shift=(0:1.5)] (0:0)--(0:1.5)--(300:1.5)--(0:0);
    \draw[shift=(0:3)] (0:0)--(0:1.5)--(300:1.5)--(0:0);
    \draw[shift=(300:1.5)] (0:0)--(0:1.5)--(300:1.5)--(0:0);
    \draw[shift=(330:1.5*sqrt(3))] (0:0)--(0:1.5)--(300:1.5)--(0:0);
    \draw[shift=(300:3)] (0:0)--(0:1.5)--(300:1.5)--(0:0);

    \node at (-1.2,0) {$\lambda_1+\lambda_2+\lambda_3$};
    \node at (-0.1,-1.3) {$\lambda_1+\lambda_2$};
    \node at (1.2,-2.6) {$\lambda_1$};

    \node at (5.7,0) {$\nu_1+\nu_2+\nu_3$};
    \node at (4.5,-1.3) {$\nu_1+\nu_2$};
    \node at (3.3,-2.6) {$\nu_1$};

    \node at (1.5, 0.5) {$\sum\lambda_i+$\\ $+\mu_1$};
    \node at (3, 0.5) {$\sum\lambda_i+$\\  $+\mu_1+\mu_2$};

    \node at (2.25, -4.2) {$0$};
    \end{tikzpicture}
\caption{Triangulation $T_3$.}\label{intro:tri}
    \end{subfigure}
    \begin{subfigure}[t]{0.3\textwidth}
    \centering
    \begin{tikzpicture}
        \draw (0,0) -- (0.5,0.866) --(1,0) -- (0.5, -0.866) -- cycle;
        \node[anchor=east] at (0,0) {$a_{1}$};
        \node[anchor=west] at (1,0) {$a_{2}$};
        \node[anchor=south] at (.5,.866) {$b_{1}$};
        \node[anchor=north] at (.5,-.866) {$b_{2}$};
    \end{tikzpicture}\caption{Order of vertices of a rhombus.}\label{fig:intro-rhom}
    \end{subfigure}\caption{}
\end{figure}

There is a more geometric way to view rhombus inequalities in terms of tropicalization of positive varieties with potential.  We will need the following notation: $[1, i] = \{ 1,2, \dots, i\} \subset \{ 1, \dots, n\},
[1, i]^{\rm op} = \{ n-i+1, \dots n\}$. For $I, J \subset [1,n]$ of the same cardinality $|I| = |J| =k$, we denote by $\Delta_{I,J}(g)$ the minor of the matrix $g$ with rows in the set $I$ and columns in the set $J$.

The group $G=\GL_n$ carries an action of $U\times U$ by left and right translations, where $U$ is the unipotent group of upper triangular matrices with the identity on the diagonal. The corner minors
$$
M_i(g) = \Delta_{[1, i]^{\rm op}, [1, i]}(g)
$$
are invariant under this action. Let $H$ be the subgroup of $G$ consisting of invertible diagonal matrices. The rational map ${\rm hw}\colon G \to H$ 
given by formula
$$
{\rm hw}\colon g\mapsto \diag\left(\frac{M_n(g)}{M_{n-1}(g)}, \frac{M_{n-1}(g)}{M_{n-2}(g)},\ldots, M_1(g)\right)
$$
is called the highest weight map, see Example \ref{eq:k=1andhw}.


We will consider the GIT quotient
$$
\mathcal{M}_2 =G^2/\!\!/U^3={\rm Spec}\left(\mathbb{C}[G^2]^{U^3}\right),
$$
where the action is given by \eqref{intro:action_3on2} (note that while the formula is the same, it is a different group which is acting).
The variety $\mathcal{M}_2$ is equipped with the potential \cite{bk2}
$$
\Phi_2(A,B) = \Phi_{\rm BK}(A) + \Phi_{\rm BK}(B) - \Phi_{\rm BK}(AB),
$$
where $\Phi_{\rm BK}$ is the Berenstein-Kazhdan potential.
We define the $(i,j)$-th {\em multi-corner minor} as
$$
M_{i,j}(A, B) := \sum_{L, |L| =j} \Delta_{[1, i+j]^{\rm op}, [1, i] \cup L}(A) \Delta_{L, [1, j]}(B).
$$
The functions $M_{i,j}$ descend to $\mathcal{M}_2$, and they define a coordinate chart and a positive structure on $\mathcal{M}_2$. 
We denote tropicalizations of the potential $\Phi_2$ and of the functions $M_{i,j}$ by 
$\Phi_2^t$ and $M_{i,j}^t$, respectively (for a more detailed definition of tropicalization, see \cite{ABHL}). 
The following theorem summarizes one possible geometric approach to Horn inequalities:
\begin{theorem}
\label{th:rhombus M2}
   The inequality $\Phi_2^t \leqslant 0$ is equivalent to the set of rhombus inequalities for $M^t_{i,j}$.
\end{theorem}

This result is an interpretation of some results in \cite{bl}, but expressing it in this particular form seems to be new.

\begin{remark} 
In the results stated above, we use the notion of tropicalization of \cite{ABHL} which assigns to a positive variety with potential a cone in $\mathbb{T}^d$, where $d$ is the dimension of the variety. A more common notion of tropicalization assigns to a positive variety with potential a subset of the lattice $\Lambda \cong \mathbb{Z}^d$ singled out by the set of inequalities defined by the potential. Denote by $\mathcal{M}_2^{t, \mathbb{Z}}$ the tropicalization of $\mathcal{M}_2$ in that sense, and by ${\rm hw}^t$ the tropical functions associated to the highest weight maps ${\rm hw}$ for $A,B$ and $AB$. Then, for a triple of dominant weights $(\lambda, \mu, \nu)$ of ${\rm U}(n)$, we have (\cite[Lemma 7.14]{bk2} and \cite{bl})
$$
|({\rm hw}^t)^{-1}(\lambda, \mu, \nu)| = c_{\lambda, \mu}^\nu.
$$
Here the l.h.s. is the cardinality of the finite set 
$({\rm hw}^t)^{-1}(\lambda, \mu, \nu) \subset \mathcal{M}_2^{t, \mathbb{Z}}$, and the r.h.s. is the Littlewood-Richardson coefficient in the decomposition $V_\lambda \otimes V_\mu = \oplus_\nu c_{\lambda, \mu}^\nu \, V_\nu$ of the tensor product of two simple modules of ${\rm U}(n)$.







\end{remark}

\begin{remark}
Although the varieties $\mathcal{P}_2$ (in the definition of the multiplicative Horn problem) and $\mathcal{M}_2$ (in the geometric interpretation of hives) have a very similar form, there are also important differences. In the first case, the quotient is by the action of a compact group $K^3$, and in the second case by the action of a non-compact group $U^3$ of smaller dimension. At the time of this writing, we are not aware of a direct relation between the two constructions.
    
\end{remark}

\subsection{Multiple Horn problems}

In this section, we formulate three versions of the multiple Horn problem, and then state the main results of the paper.

\subsubsection{Three versions of the problem}

Following the previous section, we state three versions of the multiple Horn problem: the additive, the multiplicative,  and the one for planar network.

{\em Additive problem.} We now define the following set
\[
    {\rm mHorn}^{\rm add}= \Set*{(\lambda, \mu, \nu, \rho, \sigma, \tau) \in \mathbb{R}^{6n} \given 
    \begin{aligned}
    &\exists \, a,b,c \in \mathcal{H}_n,
 {\rm eig}(a) =\lambda, {\rm eig}(b)=\mu, {\rm eig}(c)=\nu, \\
&{\rm eig}(a+b)=\rho, {\rm eig}(b+c)=\sigma, {\rm eig}(a+b+c)=\tau 
\end{aligned}
    } \ .
\]
The set ${\rm mHorn}^{\rm add}$ is studied in literature with motivations ranging from Linear Algebra to Quantum Information Theory (see \cite{BGS, DH, MM,PTZ}). Similar to the additive Horn problem, it is a cone since the problem is invariant under dilation by $s \in \mathbb{R}_{\neq 0}$.

Note that the sets 
\begin{align*}
    & \Set*{ (\lambda, \mu, \nu, \rho, \tau) \in \mathbb{R}^{5n} \given
    \begin{aligned}
        &\exists \, a,b,c \in \mathcal{H}_n, {\rm eig}(a) =\lambda, {\rm eig}(b)=\mu, {\rm eig}(c)=\nu,\\
        &{\rm eig}(a+b)=\rho,  {\rm eig}(a+b+c)=\tau
    \end{aligned}
    };\\
    & \Set*{(\lambda, \mu, \nu, \rho, \tau) \in \mathbb{R}^{5n} \given
    \begin{aligned}
        &\exists \, a,b,c \in \mathcal{H}_n, {\rm eig}(a) =\lambda, {\rm eig}(b)=\mu, {\rm eig}(c)=\nu,\\
        &{\rm eig}(b+c)=\sigma,  {\rm eig}(a+b+c)=\tau
    \end{aligned}
    },
\end{align*}
where we drop one of the sums ($a+b$ or $b+c$) are polyhedral cones which can be completely described in terms of trace equalities and rhombus inequalities.
It is important to stress that the set ${\rm mHorn}^{\rm add}$ is not a polyhedral cone, in general.

{\em Multiplicative problem.} The multiplicative version of the problem takes the form
\[
    {\rm mHorn}^{\rm mult}(s) = \Set*{ (\lambda, \mu, \nu, \rho, \sigma, \tau) \in \mathbb{R}^{6n} \given
    \begin{aligned}
        &\exists \, A, B, C \in {\rm GL}_n, \text{~such that ~}
        {\rm sing}(A) =e^{s\lambda},\\
        &{\rm sing}(B)=e^{s\mu},{\rm sing}(C)=e^{s\nu}, {\rm sing}(AB)=e^{s\rho},\\
        &{\rm sing}(BC)=e^{s\sigma}, {\rm sing}(ABC)=e^{s\tau}
    \end{aligned}
    }.
\]

As before, there is an action of $K^4$ on $G^3$
\begin{equation} \label{intro:action_4on3}
    (u_1, u_2, u_3, u_4): (A, B, C) \mapsto (u_1 A u_2^{-1}, 
    u_2 B u_3^{-1}, u_3 C u_4^{-1})
\end{equation}
which preserves singular values of $A, B, C, AB, BC$ and $ABC$.
Hence, one can replace ${\rm GL}_n$ by $\mathcal{B}(n)= AU_-$ in the definition above without changing the set of 6-tuples $(\lambda, \mu, \nu, \rho, \sigma, \tau)$. For this reason, similar to the multiplicative Horn problem, the sets ${\rm mHorn}_{\lambda, \mu, \nu}^{\rm mult}(s)$ carry Duistermaat-Heckman measures ${\rm DH}^{\rm mult}_{\lambda, \mu, \nu}(s)$.
Note however that the sets ${\rm mHorn}^{\rm mult}(s)$ are not polyhedral cones, and that they strongly depend on $s$.

Furthermore, one can introduce a variety
$$
\mathcal{P}_3=G^3/K^4
$$
using the action \eqref{intro:action_4on3}, and a 
singular value map $\mathcal{P}_3 \to \mathbb{R}^{6n}$. A 6-tuple
$(\lambda, \mu, \nu, \rho, \sigma, \tau)$ belongs to ${\rm mHorn}^{\rm mult}(s)$ if and only if its pre-image in $\mathcal{P}_3$ is nonempty.

{\em Tropical problem.} Similarly, for planar networks of $\Pi_1, \Pi_2, \Pi_3$ of rank $n$, we define
\[
    {\rm mHorn}^{\Pi_1, \Pi_2, \Pi_3} = \Set*{ (\lambda, \mu, \nu, \rho, \sigma, \tau) \in \mathbb{R}^{6n} \given
    \begin{aligned}
        &\exists \, w_i \in \mathbb{T}^{\Pi_i}, L(\lambda) = m(\Pi_1, w_1), L(\mu)=m(\Pi_2, w_2), \\
        &L(\nu)=m(\Pi_3, w_3), L(\rho) = m(\Pi_1 \circ \Pi_2, w_1 \circ w_2),\\
        &L(\sigma) = m(\Pi_2 \circ \Pi_3, w_2 \circ w_3), \\
        &L(\tau)=m(\Pi_1 \circ \Pi_2 \circ \Pi_3, w_1 \circ w_2\circ w_3) 
    \end{aligned}
    }.
\]
Study of this set is one of the main goals of this paper.

\subsubsection{Main results}

We are now able to state the main results of this article. 

{\em Tropical problem.} 
In this section, we address the description of the set ${\rm mHorn}^{\Pi_1, \Pi_2, \Pi_3}$. In order to present the results, we now draw a graph in the shape of a tetrahedron (see Fig.\ref{picture tetrahedron with tsv on the edges}), and write the arrays 
$$
L(\lambda), \sum_i \lambda_i +L(\mu), \sum_i \lambda_i + \sum_j \mu_j +L(\nu), L(\rho), \sum_i \lambda_i +L(\sigma), L(\tau)
$$ 
on its edges.

\begin{theorem}[Theorems \ref{theorem PN rhombi inequalities} and \ref{theorem PN tetrahedra equalities}]       \label{intro:tropical_image}
The 6-tuple $(\lambda, \mu, \nu, \rho, \sigma, \tau)$ belongs to ${\rm mHorn}^{\Pi_1, \Pi_2, \Pi_3}$ only if it satisfies the trace equalities, and there exist numbers $m_{i,j,k}$ that can be assigned to vertices of the graph such that they satisfy rhombus inequalities for all plane small rhombi, and the
tetrahedron equalities described below.
\end{theorem}

Tetrahedron equalities correspond to small tetrahedron having exactly one vertex in the middle of each edge. Denote the numbers associated to these vertices by $m_\lambda, \dots, m_\tau$, and denote
$$
\alpha=m_\lambda + m_\nu, \hskip 0.3cm
\beta=m_\mu + m_\tau, \hskip 0.3cm
\gamma = m_\rho + m_\sigma
$$
the sums of numbers corresponding to opposite edges of the tetrahedron. The tetrahedron equalities take the form
$$
\alpha \leqslant {\rm max}\{\beta, \gamma\}, \hskip 0.3cm
\beta \leqslant {\rm max}\{\gamma, \alpha\}, \hskip 0.3cm
\gamma \leqslant {\rm max}\{\alpha, \beta\}.
$$
It is natural to put forward the following conjecture:

\begin{conj}\label{conjA}
    For $\Pi_1, \Pi_2, \Pi_3$ large enough networks ({\em e.g.} standard networks), the conditions of Theorem \ref{intro:tropical_image} are necessary and sufficient. That is, the set ${\rm mHorn}^{\Pi_1, \Pi_2, \Pi_3}$ is completely described by trace and tetrahedron equalities, and by rhombus equalities.
\end{conj}

We show that this conjecture holds true for $n=2$ (see Appendix A).

In order to state the next result, we need the following notation: we label faces of the tetrahedron by the arrays written on its edges. For instance, the face $(\lambda, \mu, \rho)$ is the face with edges carrying arrays $L(\lambda), \sum_i \lambda_i + L(\mu), L(\rho)$.

\begin{theorem}[Theorem \ref{theorem PN octahedron recurrence}]      \label{intro:multiple_tropical}
    For $\Pi_1, \Pi_2, \Pi_3$ the standard networks with weights satisfying Gelfand-Zeitlin conditions, for all small tetrahedra we have
    \begin{equation}      \label{intro:octahedron}
\gamma = {\rm max}\{\alpha, \beta\}.
    \end{equation}
     Furthermore, the values of $m_{i,j,k}$ on the pair of faces $(\lambda, \mu, \rho)$ and $(\rho, \nu, \tau)$ (or the pair of faces $(\mu, \nu, \sigma)$ and $(\lambda, \sigma, \tau)$) completely determine all the other $m_{i,j,k}$'s.
\end{theorem}

Equations \eqref{intro:octahedron} coincide with the octahedron recurrence from the theory of crystals (see {\em e.g.} \cite{HK}). We denote the corresponding set by 
$$
{\rm mHorn}^{\rm oct} \subset {\rm mHorn}^{\rm trt},
$$
where ${\rm mHorn}^{\rm trt}$ is the set defined by the trace and tetrahedron equalities, and by rhombus inequalities.
Example of $n=2$ shows that this inclusion is strict, in general. The set ${\rm mHorn}^{\rm oct}$ is completely determined by trace equalities and rhombus inequalities for the pair of faces $(\lambda, \mu, \rho)$ and $(\rho, \nu, \tau)$ (or the pair of faces $(\mu, \nu, \sigma)$ and $(\lambda, \sigma, \tau)$) which makes this problem similar to the original Horn problem. 

Theorem \ref{intro:multiple_tropical} admits the following geometric interpretation. Similar to the geometric interpretation of rhombus inequalities in the Horn problem, we define the variety
$$
\mathcal{M}_3 = G^3/\!\!/U^4={\rm Spec}\left(\mathbb{C}[G^3]^{U^4}\right),
$$
where the action is given by \eqref{intro:action_4on3} (again, the formula is the same, but the acting group is different).
This variety is equipped with the potential
$$
\Phi_3(A, B, C) = \Phi_{\rm BK}(A) + \Phi_{\rm BK}(B) + \Phi_{\rm BK}(C) - \Phi_{\rm BK}(ABC),
$$
and with the set of multi-corner minors
$$
M_{i,j,k}(A, B, C) = \sum_{L_1, L_2} \Delta_{[1,i+j+k]^{\rm op}, [1, i] \cup L_1}(A) \Delta_{L_1, [1, j] \cup L_2}(B) \Delta_{L_2, [1,k]}(C).
$$
Our findings concerning the variety $\mathcal{M}_3$ are summarized in the following theorem:
\begin{theorem}[Theorems \ref{Thm:octrec} and \ref{thm:sumofmt3}]    \label{intro:M_3properties}
    The functions $M_{i,j,k}$ define a positive structure on $\mathcal{M}_3$, Pl\"ucker relations for the functions $M_{i,j,k}$ allow to express all of them in terms of $M_{i,j,k}$'s corresponding to the pair of faces $(\lambda, \mu, \rho)$ and $(\rho, \nu, \tau)$ (or the pair of faces $(\mu, \nu, \sigma)$ and $(\lambda, \sigma, \tau)$).

The inequality $\Phi_3^t \leqslant 0$ is equivalent to rhombus inequalities for $M^t_{i,j,k}$, and tropicalization of Pl\"ucker relations for $M_{i,j,k}$ is equivalent to the octahedron recurrence for $M^t_{i,j,k}$.
\end{theorem}

Using a geometric version of Gelfand-Zeitlin conditions, we obtain an alternative proof of Theorem \ref{intro:multiple_tropical} based on Theorem
\ref{intro:M_3properties}.

{\em Multiplicative problem.} We now describe some results concerning the multiplicative problem. We start with the following conjecture:

\vskip 0.2cm

\begin{conj}\label{conjB}
For all $\varepsilon >0$ there is $s_0 >0$ such that for all $s>s_0$ we have
$$
{\rm mHorn}^{\rm mult}(s) \subset U_\varepsilon({\rm mHorn}^{\rm trt}),
$$
where $U_\varepsilon({\rm mHorn}^{\rm trt})$ is the $\varepsilon$-neighborhood of ${\rm mHorn}^{\rm trt}$.
\end{conj}

The conjecture above states that for $s$ large enough solutions of the multiplicative problem can be approximated by 6-tuples satisfying the trace equalities and rhombus and tetrahedron equalities. We show that this conjecture holds true for $n=2$ (see Appendix A).

Our main result for the multiplicative problem concerns the Duistermaat-Heckman measures of its solution sets.

%
%
%
\begin{theorem}[Theorem \ref{theorem: symplectic convergence}]      \label{intro:multiple_DH}
    For all $\delta>0$ and $\varepsilon >0$ there is $s_0>0$ such that for all $s>s_0$ we have
    $$
{\rm DH}^{\rm mult}_{\lambda, \mu, \nu}(U_\varepsilon({\rm mHorn}^{\rm oct}_{\lambda, \mu, \nu}))
\geqslant (1-\delta){\rm DH}^{\rm mult}_{\lambda, \mu, \nu}({\rm mHorn}^{\rm mult}(s)).
    $$
\end{theorem}

Recall that ${\rm mHorn}^{\rm oct}_{\lambda, \mu, \nu} \subset {\rm mHorn}^{\rm trt}_{\lambda, \mu, \nu}$, and that the inclusion is strict, in general. Theorem
\ref{intro:multiple_DH} states that for $s$ large enough the Duistermaat-Heckman measure of the multiplicative problem converges towards the set ${\rm mHorn}^{\rm oct}_{\lambda, \mu, \nu}$ while Conjecture \ref{conjB} stated above claims that (again, for $s$ large enough) the whole set ${\rm mHorn}^{\rm mult}(s)_{\lambda, \mu, \nu}$ is contained in a small neighborhood of the set ${\rm mHorn}^{\rm trt}_{\lambda, \mu, \nu}$.

\section{Planar networks}\label{sec:planar}

\subsection{Preliminaries}

In this section, we recall the notion of planar networks, more details can be found in \cite{APS-planar}.

A \emph{planar network} is a planar graph embedded in a strip $1\leqslant y\leqslant n$, $0\leqslant x\leqslant r$ such that its edges are never vertical. A planar network is naturally oriented from left to right. A planar network is of \emph{rank}  $n$ if it has $n$ sources on the line $x=0$ and $n$ sinks on the line $x=r$. We will enumerate the sources and the sinks by the numbers $\{1,\ldots, n\}$ from bottom to top.

\begin{figure}[H]
    \centering
    \begin{tikzpicture}
    \draw (0,0)[red] node[black,left] {3} -- (4,0) node [black,anchor=south east]{$d$};
    \draw (0,-1) node[left] {2}  -- (2,-1);
    \draw[red] (2,-1) -- (4,-1) node [black,anchor=south east]{$e$};
    \draw[red] (0,-2) node[black, left] {1} -- (1.5,-2);
    \draw (1.5,-2) -- (4,-2) node [anchor=south east]{$f$};
    \draw (0.5,0) -- (1, -1) node [above, yshift=7]{$a$};
    \draw (2,0) -- (2.5, -1) node [above, yshift=7]{$c$};
    \draw[red] (2,-1) -- (1.5, -2) node [black,above, yshift=7]{$b$};
\end{tikzpicture}
    \caption{A planar network of rank 3.
    In red is a 2-multipath of weight $(b\cdot e)\cdot d$.}
\end{figure}
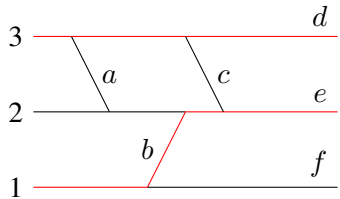

A crucial role will be played by
\begin{definition}
   For $\alpha\in\N$ an \emph{$\alpha$-multipath} in a planar network $\Pi$ is a collection of $\alpha$ non-intersecting oriented (by the planar network orientation) paths whose sources and sinks are among those of $\Pi$. We write $s(p)=I, t(p)=J$ or $p\colon I\to J$ if a multipath $p$ connects the sources labeled by $I$ to the sinks labeled by $J$. The set of all $\alpha$-multipaths in $\Pi$ is denoted by $P_\alpha(\Pi)$.
\end{definition}

Let $(R,\boxplus,\mathbf{0},\boxtimes,\mathbf{1})$ be a semiring and $\Pi$ a planar network. Denote by $\edge(\Pi)$ the set of all edges of $\Pi$. An \emph{$R$-weighting} of $\Pi$ is an assignment $w\colon \edge(\Pi)\to R$. The values $w(e)$ are called \emph{the weights}. Denote by $R^\Pi$ the set of all $R$-weightings of $\Pi$. In this paper $R$ will be $\R_{\geqslant 0}$ or $\C$ with operations $+, \times$, or the tropical numbers $\T=\R\cup\{-\infty\}$ with the operations $\boxplus=\max$ (identity $-\infty$) and $\boxtimes=+$ (identity 0). 

Given a weighted planar network $(\Pi,w)$, we define a \emph{weight} of a multipath $p$ by $w(p):=\prod_{e\in p}w(e)$. To simplify the notation, if the weight of an edge is $\mathbf{1}$, we will omit the label (as in the picture above).

For any two weighted planar networks ($\Pi_1, w_1$) and ($\Pi_2, w_2$) of rank $n$, we denote their \textit{concatenation} and the corresponding weighting by ($\Pi_1\circ \Pi_2, w_1\circ w_2$). We will also call the resulting weighted network a \textit{product} of $\Pi_1$ and $\Pi_2$.

As an analogue of a multipath in one planar network, we introduce the following notion.
\begin{definition}
    For a $k$-tuple of weighted planar networks $(\Pi_i, w_i)$ of rank $n$ and a $k$-tuple of non-negative integers $\underline{\alpha}=(\alpha_1,\ldots,\alpha_k)$, an \emph{$\underline{\alpha}$-multipath} $p=(p_1,\ldots,p_k)$ in $(\Pi_1\circ\cdots\circ\Pi_k)$ is a collection of pairwise non-intersecting multipaths $p_i\in P_{\alpha_i}(\Pi_1\circ\cdots\circ\Pi_i)$. The set of all $\underline{\alpha}$ multipaths is denoted by $P_{\underline{\alpha}}(\Pi_1\circ\cdots\circ\Pi_k)$ or  $P_{\underline{\alpha}}$ for simplicity. The \emph{weight} of $p\in P_{\underline{\alpha}}$ is given by $w(p) = \prod w_i(p_i)$. The {\em sinks} of such a multipath is the tuple of sets $(sinks(p_1),\ldots, sinks(p_k))$.
\end{definition}
Notice that $P_{\underline{\alpha}}(\Pi_1\circ\cdots\circ\Pi_k)$ is nonempty only if $0\leqslant\alpha_i\leqslant n$ and $0\leqslant\sum \alpha_i\leqslant n$.

\begin{figure}[H]
    \centering
    \begin{tikzpicture}
    \draw (0,0)[red] node[black,left] {3} -- (4,0) node [black,anchor=south east]{$d_1$};
    \draw (0,-1) node[left] {2}  -- (2,-1);
    \draw[red] (2,-1) -- (4,-1) node [black,anchor=south east]{$e_1$};
    \draw[red] (0,-2) node[black, left] {1} -- (1.5,-2);
    \draw (1.5,-2) -- (4,-2) node [anchor=south east]{$f_1$};
    \draw (0.5,0) -- (1, -1) node [above, yshift=7]{$a_1$};
    \draw (2,0) -- (2.5, -1) node [above, yshift=7]{$c_1$};
    \draw[red] (2,-1) -- (1.5, -2) node [black,above, yshift=7]{$b_1$};

    \draw (4,.3) -- (4,-2.3);
    \begin{scope}[xshift=4cm]
    \draw (0,0) -- (1.5,0); 
    \draw[red] (1.5,0) -- (4,0) node [black,anchor=south east]{$d_2$};
    \draw (0,-1)[red]  -- (1,-1);
    \draw (1,-1) -- (4,-1) node [black,anchor=south east]{$e_2$};
    \draw (0,-2) -- (1.5,-2);
    \draw (1.5,-2) -- (4,-2) node [anchor=south east]{$f_2$};
    \draw[red] (1.5,0) -- (1, -1) node [black,above, yshift=7]{$a_2$};
    \draw (2.5,0) -- (3, -1) node [above, yshift=7]{$c_2$};
    \draw (2,-1) -- (2.5, -2) node [black,above, yshift=7]{$b_2$};
    \end{scope}
\end{tikzpicture}
    \caption{A concatenation of two planar networks. In red is a (1,1)-multipath of weight $b_1\cdot e_1\cdot a_2\cdot d_2\cdot d_1$.}
\end{figure}

\subsection{Tropical multiple Horn  problem} 

Let $(\Pi, w)$ be a $\T$-weighted planar network of rank $n$.  
\begin{definition}
    The \emph{tropical singular values} $(\la_1,\ldots, \la_n)$ of $(\Pi, w)$ are defined by \[\la_1+\cdots+\la_l = \max_{p\in P_l(\Pi)} w(p),\]
     where we set $\max_{p\in P_l(\Pi)} = -\infty$ if $P_l(\Pi)$ is empty.
\end{definition}

\begin{definition}\label{definition: map m}
    Let $(\Pi_i, w_i)$ be a $k$-tuple of $\T$-weighted planar networks of rank $n$ such that $P_n(\Pi_i)\neq\emptyset$. Given a $k$-tuple of integers $\underline{\alpha}$ satisfying $0\leqslant\alpha_i\leqslant n$ and $0\leqslant\sum \alpha_i\leqslant n$ we define the following functions on $\prod_{i=1}^k\T^{\Pi_i}$:
\begin{equation}\label{equation: m_alpha tropical}
    m_{\underline{\alpha}}(w_1,\ldots , w_k) := \max_{p\in P_{\underline{\alpha}}}w(p),
\end{equation}
where we set $m_{0,\ldots, 0}=0$, and $m_{\underline{\alpha}}=-\infty$ if $P_{\underline{\alpha}}$ is empty. 
The set of all such $\alpha$'s can be identified with the set of integer points of the $k$-dimensional simplex $\{x \mid  0\leqslant x_i\leqslant n \text{ for } i=1,\ldots, k, 0\leqslant \sum_1^k x_i \leqslant n\}$. Denote this set by $\Delta^k(n)$. Denote by $\T^{\Delta^k(n)}$ the set of $\Delta^k(n)$-tuples of tropical numbers.
Let
\begin{equation}\label{equation: the map m}
    m: \prod_{i=1}^k \T^{\Pi_i} \rightarrow \T^{\Delta^k(n)}
\end{equation}
 be a map whose components are $m_{\underline{\alpha}}$. Pictorially this is represented as follows:
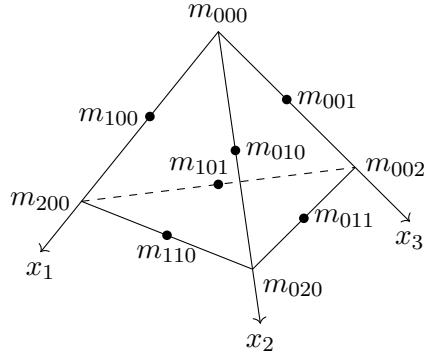
\begin{center}
\begin{tikzpicture}[scale=.9]
        \tikzmath{ 
            \l1=0; \l2=0; 
            \r1=4; \r2=.5;
            \u1=2; \u2=2.5;
            \d1=2.5; \d2=-1;}
        \draw (\u1,\u2) node[above] {$ m_{000}$} -- (\l1,\l2) node[left] {$ m_{200}$} -- (\d1,\d2) node[below right] {$ m_{020}$} -- (\r1,\r2) node[right] {$ m_{002}$}  -- cycle;
        \draw[dashed] (\l1,\l2) -- (\r1,\r2);
        \draw (\u1,\u2) -- (\d1,\d2);
        \fill(.5*\d1+.5*\u1, .5*\d2+.5*\u2) circle (2pt) node[right] {$m_{010}$};
        \fill(.5*\l1+.5*\r1, .5*\l2+.5*\r2) circle (2pt) node[above,xshift=-7] {$m_{101}$};
        \fill(.5*\l1+.5*\u1, .5*\l2+.5*\u2) circle (2pt) node[left] {$m_{100}$};
        \fill(.5*\r1+.5*\d1, .5*\r2+.5*\d2) circle (2pt) node[right] {$m_{011}$};
        \fill(.5*\r1+.5*\u1, .5*\r2+.5*\u2) circle (2pt) node[right] {$m_{001}$};
        \fill(.5*\l1+.5*\d1, .5*\l2+.5*\d2) circle (2pt) node[below] {$m_{110}$};
        \draw[->] (\l1,\l2) -- (1.3*\l1-.3*\u1,1.3*\l2-.3*\u2) node[below] {$x_1$};
        \draw[->] (\d1,\d2) -- (1.23*\d1-.23*\u1,1.23*\d2-.23*\u2) node[below] {$x_2$};
        \draw[->] (\r1,\r2) -- (1.4*\r1-.4*\u1,1.4*\r2-.4*\u2) node[below] {$x_3$};
\end{tikzpicture}
\end{center}
\end{definition}

\begin{example}\label{example: planar network and the map m}
    Here is an example of a triple of $\T$-weighted planar networks and the result of applying the function $m$.
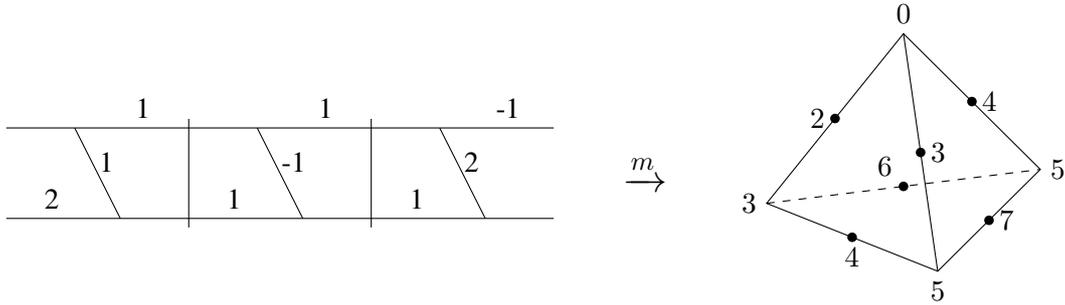
\begin{center}
\begin{tikzpicture}
\begin{scope}[yshift=1cm,scale=.6]
    \draw (0,0) -- (4,0) -- (8,0) -- (12,0);
    \draw (0,-2) -- (4,-2) -- (8,-2) -- (12,-2);
    \draw (4,.2) -- (4,-2.2);
    \draw (8,.2) -- (8,-2.2);
    \draw (1.5,0) -- (2.5,-2);
    \draw (5.5,0) -- (6.5,-2);
    \draw (9.5,0) -- (10.5,-2);
    \node[above] at (1,-2) {2};
    \node[above] at (2.2,-1.2) {1};
    \node[above] at (3,0) {1};
    
    \node[above] at (5,-2) {1};
    \node[above] at (6.3,-1.2) {-1};
    \node[above] at (7,0) {1};
    
    \node[above] at (9,-2) {1};
    \node[above] at (10.2,-1.2) {2};
    \node[above] at (11,0) {-1};

    \node at (14,-1)[font=\Large]{$\xrightarrow{m}$};
\end{scope}

\begin{scope}[xshift= 10cm,scale=.9]
        \tikzmath{ 
            \l1=0; \l2=0; 
            \r1=4; \r2=.5;
            \u1=2; \u2=2.5;
            \d1=2.5; \d2=-1;}
        \draw (\u1,\u2) node[above] {$ 0$} -- (\l1,\l2) node[left] {$ 3$} -- (\d1,\d2) node[below] {$ 5$} -- (\r1,\r2) node[right] {$ 5$}  -- cycle;
        \draw[dashed] (\l1,\l2) -- (\r1,\r2);
        \draw (\u1,\u2) -- (\d1,\d2);
        \fill(.5*\d1+.5*\u1, .5*\d2+.5*\u2) circle (2pt) node[right] {$3$};
        \fill(.5*\l1+.5*\r1, .5*\l2+.5*\r2) circle (2pt) node[above,xshift=-7] {$6$};
        \fill(.5*\l1+.5*\u1, .5*\l2+.5*\u2) circle (2pt) node[left] {$2$};
        \fill(.5*\r1+.5*\d1, .5*\r2+.5*\d2) circle (2pt) node[right] {$7$};
        \fill(.5*\r1+.5*\u1, .5*\r2+.5*\u2) circle (2pt) node[right] {$4$};
        \fill(.5*\l1+.5*\d1, .5*\l2+.5*\d2) circle (2pt) node[below] {$4$};
\end{scope}
\end{tikzpicture}
\end{center}
\end{example}

Notice that $m_{l,0,\ldots, 0}=\la_1(\Pi_1)+\cdots+\la_l(\Pi_1)$, the sum of the tropical singular values of $(\Pi_1,w_1)$. Similarly we have the following expression for the tropical singular values of $\Pi_i$ for $i\geqslant 2$:
\[m_{0^{i-2},n-l,l, 0^{k-i}} = \sum_{s=1}^{i-1}\sum_{r=1}^n\la_r(\Pi_s)+\sum_{r=1}^l\la_r(\Pi_i),\]
and for the tropical singular values of  $\Pi_i\circ\Pi_{i+1}\circ\cdots\circ\Pi_{i+j}$: 
\[m_{0^{i-2}, n-l, 0^j, l,  0^{k-i-j}} = \sum_{s=1}^{i-1}\sum_{r=1}^n\la_r(\Pi_s)+\sum_{r=1}^l\la_r(\Pi_i\circ\cdots\circ\Pi_{i+j}),\]
where the l.h.s. should be interpreted as $m_{0^j,l,0^{k-1-j}}$ if $i=1$.

\begin{figure}[H]
\centering
    \begin{subfigure}[t]{0.6\textwidth}
    \centering
        \begin{tikzpicture}[scale=1.1]
        \tikzmath{
            \l1=-2; \l2=-2.5; 
            \r1=2; \r2=-2;
            \u1=0; \u2=0;
            \d1=.25; \d2=-3.5;}
        \draw (\l1,\l2) -- (\u1,\u2) -- (\r1,\r2) -- (\d1,\d2) -- cycle;
        \draw[dashed] (\l1,\l2) -- (\r1,\r2);
        \draw (\u1,\u2) -- (\d1,\d2);
        
        \tikzmath{\k=.33333;
            \v1=\k*(\l1-\u1); \v2=\k*(\l2-\u2);
            \vv1=\k*(\d1-\u1); \vv2=\k*(\d2-\u2);
            \vvv1=\k*(\r1-\u1); \vvv2=\k*(\r2-\u2);}
            
        \foreach \z/\x in {1/0, 2/0, 1/1, 2/1, 1/2} \draw[black!30] (\z*\vvv1+\x*\v1,\z*\vvv2+\x*\v2) -- (\z*\vv1+\x*\v1,\z*\vv2+\x*\v2);
        \foreach \x/\y in {1/0, 2/0, 1/1, 2/1, 1/2} \draw[black!30] (\x*\v1+\y*\vv1,\x*\v2+\y*\vv2) -- (\x*\vvv1+\y*\vv1,\x*\vvv2+\y*\vv2);
        \foreach \y/\z in {1/0, 2/0, 1/1, 2/1, 1/2} \draw[black!30] (\y*\vv1+\z*\vvv1,\y*\vv2+\z*\vvv2) -- (\y*\v1+\z*\vvv1,\y*\v2+\z*\vvv2);
        
        \foreach \x/\y/\z in {0/1/2, 0/2/1, 1/0/2, 1/1/1, 2/0/1} \draw[black!30] (\x*\v1+\y*\vv1,\x*\v2+\y*\vv2) -- (\x*\v1+\y*\vv1+\z*\vvv1,\x*\v2+\y*\vv2+\z*\vvv2);
        \foreach \x/\y/\z in {1/2/0,2/1/0, 0/2/1, 1/1/1,0/1/2} \draw[black!30] (\x*\v1+\z*\vvv1,\x*\v2+\z*\vvv2) -- (\x*\v1+\y*\vv1+\z*\vvv1,\x*\v2+\y*\vv2+\z*\vvv2);
        \foreach \x/\y/\z in {2/0/1,1/0/2,2/1/0,1/1/1,1/2/0} \draw[black!30] (\y*\vv1+\z*\vvv1,\y*\vv2+\z*\vvv2) -- (\x*\v1+\y*\vv1+\z*\vvv1,\x*\v2+\y*\vv2+\z*\vvv2);

        \tikzmath{\x=0;\y=0;\z=0;}
            \fill (\x*\v1+\y*\vv1+\z*\vvv1,\x*\v2+\y*\vv2+\z*\vvv2) circle(1.3pt) node[left] {0};
        
        \tikzmath{\x=1;\y=0;\z=0;}
            \fill (\x*\v1+\y*\vv1+\z*\vvv1,\x*\v2+\y*\vv2+\z*\vvv2) circle(1.3pt) node[left] {$\scriptstyle\la_1$};
        \tikzmath{\x=0;\y=1;\z=0;}
            \fill (\x*\v1+\y*\vv1+\z*\vvv1,\x*\v2+\y*\vv2+\z*\vvv2) circle(1.3pt);
        \tikzmath{\x=0;\y=0;\z=1;}
            \fill (\x*\v1+\y*\vv1+\z*\vvv1,\x*\v2+\y*\vv2+\z*\vvv2) circle(1.3pt);
        
        \tikzmath{\x=2;\y=0;\z=0;}
            \fill (\x*\v1+\y*\vv1+\z*\vvv1,\x*\v2+\y*\vv2+\z*\vvv2) circle(1.3pt) node[left] {$\scriptstyle\la_1+\la_2$};
        \tikzmath{\x=1;\y=1;\z=0;}
            \fill (\x*\v1+\y*\vv1+\z*\vvv1,\x*\v2+\y*\vv2+\z*\vvv2) circle(1.3pt);
        \tikzmath{\x=0;\y=2;\z=0;}
            \fill (\x*\v1+\y*\vv1+\z*\vvv1,\x*\v2+\y*\vv2+\z*\vvv2) circle(1.3pt);
        \tikzmath{\x=1;\y=0;\z=1;}
            \fill (\x*\v1+\y*\vv1+\z*\vvv1,\x*\v2+\y*\vv2+\z*\vvv2) circle(1.3pt);
        \tikzmath{\x=0;\y=1;\z=1;}
            \fill (\x*\v1+\y*\vv1+\z*\vvv1,\x*\v2+\y*\vv2+\z*\vvv2) circle(1.3pt);
        \tikzmath{\x=0;\y=0;\z=2;}
            \fill (\x*\v1+\y*\vv1+\z*\vvv1,\x*\v2+\y*\vv2+\z*\vvv2) circle(1.3pt);
        
        \tikzmath{\x=3;\y=0;\z=0;}
            \fill (\x*\v1+\y*\vv1+\z*\vvv1,\x*\v2+\y*\vv2+\z*\vvv2) circle(1.3pt) node[left] {$\scriptstyle\sum\la_i$};
        \tikzmath{\x=2;\y=1;\z=0;}
            \fill (\x*\v1+\y*\vv1+\z*\vvv1,\x*\v2+\y*\vv2+\z*\vvv2) circle(1.3pt) node[below left] {$\scriptstyle\sum\la_i+\mu_1$};
        \tikzmath{\x=1;\y=2;\z=0;}
            \fill (\x*\v1+\y*\vv1+\z*\vvv1,\x*\v2+\y*\vv2+\z*\vvv2) circle(1.3pt) node[below left,yshift=-3,xshift=15] {$\scriptstyle\sum\la_i+\mu_1+\mu_2$};
        \tikzmath{\x=0;\y=3;\z=0;}
            \fill (\x*\v1+\y*\vv1+\z*\vvv1,\x*\v2+\y*\vv2+\z*\vvv2) circle(1.3pt) node[below, yshift=-4] {$\scriptstyle\sum\la_i+\sum\mu_i$};

        \tikzmath{\x=2;\y=0;\z=1;}
            \fill (\x*\v1+\y*\vv1+\z*\vvv1,\x*\v2+\y*\vv2+\z*\vvv2) circle(1.3pt);
        \tikzmath{\x=1;\y=1;\z=1;}
            \fill (\x*\v1+\y*\vv1+\z*\vvv1,\x*\v2+\y*\vv2+\z*\vvv2) circle(1.3pt);
        \tikzmath{\x=0;\y=2;\z=1;}
        \fill (\x*\v1+\y*\vv1+\z*\vvv1,\x*\v2+\y*\vv2+\z*\vvv2) circle(1.3pt) node[below right] {$\scriptstyle\cdots+\nu_1$};;
        \tikzmath{\x=1;\y=0;\z=2;}
        \fill (\x*\v1+\y*\vv1+\z*\vvv1,\x*\v2+\y*\vv2+\z*\vvv2) circle(1.3pt);
        \tikzmath{\x=0;\y=1;\z=2;}
        \fill (\x*\v1+\y*\vv1+\z*\vvv1,\x*\v2+\y*\vv2+\z*\vvv2) circle(1.3pt) node[below right] {$\scriptstyle\cdots+\nu_1+\nu_2$};
        \tikzmath{\x=0;\y=0;\z=3;}
        \fill (\x*\v1+\y*\vv1+\z*\vvv1,\x*\v2+\y*\vv2+\z*\vvv2) circle(1.3pt) node[right] {$\scriptstyle\cdots+\sum\nu_i$};
    \end{tikzpicture}
        \caption{$\Delta^3(3)$. $\la$, $\mu$, $\nu$ denote the tropical singular values of $\Pi_1$, $\Pi_2$, $\Pi_3$ respectively.}
        \label{picture tetrahedron with tsv on the edges}
     \end{subfigure}
    \hspace{2em}
    \begin{subfigure}[t]{0.3\textwidth}
    \centering
    \begin{tikzpicture}[scale=0.9]
        \tikzmath{ 
            \l1=0; \l2=0; 
            \r1=4; \r2=.5;
            \u1=2; \u2=2.5;
            \d1=2.5; \d2=-1;} 
        \draw[-{Stealth}________________________________________]  (\u1,\u2) -- (\l1,\l2);
        \node[left] at (.5*\l1+.5*\u1, .5*\l2+.5*\u2) {$\la(1)$};
        
        \draw[-{Stealth}____________________________________________]  (\u1,\u2) -- (\d1,\d2); 
        \node[right] at (.5*\d1+.5*\u1, .5*\d2+.5*\u2) {$\la(12)$};
        
        \draw[-{Stealth}____________________________]  (\u1,\u2) -- (\r1,\r2);
        \node[right,xshift=3] at (.5*\r1+.5*\u1, .5*\r2+.5*\u2) {$\la(123)$};
        
        \draw[-{Stealth}_______________________________]  (\l1,\l2) -- (\d1,\d2);
        \node[below] at (.5*\l1+.5*\d1, .5*\l2+.5*\d2) {$\la(2)$};
        
        \draw[-{Stealth}___________________________________]  (\d1,\d2) -- (\r1,\r2);
        \node[below,xshift=5] at (.5*\d1+.5*\r1, .5*\d2+.5*\r2) {$\la(3)$};
        
        \draw[dashed, -{Stealth}_______________________________________________________]  (\l1,\l2) -- (\r1,\r2);
        \node[below,xshift=-3] at (.5*\r1+.5*\l1, .5*\r2+.5*\l2) {$\la(23)$};    
    \end{tikzpicture}
    \caption{$\la(1)$ denotes the tropical singular values of $\Pi_1$; $\la(12)$,  of $\Pi_{12}$, etc.}\label{picture: edges of simplex and tsv of products}
\end{subfigure}\caption{}
\end{figure}

Thus we can extract the tropical singular values of a product of the form $\Pi_i\circ\Pi_{i+1}\circ\cdots\circ\Pi_{i+j}$ as consecutive differences of the numbers standing on a certain edge of the simplex (see Fig.\ref{picture: edges of simplex and tsv of products}). For example, the tropical singular values of $\Pi_2\circ\Pi_3$  from Example \ref{example: planar network and the map m} are $\la_1=6-3=3,\la_2=5-6=-1$.
Denote by $\partial\Delta^k(n)$ the set of integer points on the edges of the simplex, {\em i.e.},
\[\partial\Delta^k(n) = \Big(\bigcup_{i=1}^k\{x \mid x_j=0 \text{ for } j\neq i\}\cup \bigcup_{1\leqslant i<j\leqslant k}\{x \mid x_i+x_j=n, x_l=0 \text{ for } l\neq i, j\}\Big)\cap \Delta^k(n).\]
We define the map 
\begin{equation}\label{equation partial on R^tetrahedron}
\partial\colon \T^{\Delta^k(n)}\rightarrow \T^{\partial\Delta^k(n)},
\end{equation} extracting the values at the points of $\partial\Delta^k(n)$.

We are ready to formulate the tropical multiple Horn problem:
\begin{center}
\textit{Describe the image of the map} $\displaystyle{\partial\circ m:\prod_{i=1}^k \T^{\Pi_i} \rightarrow\T^{\partial\Delta^k(n)}}.$
\end{center}
In other words, it asks to define the possible collections of tropical singular values of the consecutive products $(\Pi_i\circ\Pi_{i+1}\circ\cdots\circ\Pi_{i+j}, w_i\circ\cdots\circ w_{i+j})$. 

\subsection{Results}

We restrict ourselves to the case of three planar networks. The proofs are given in the next section.

Consider the rhombus with vertices  $(0,0,0), (1,0,0), (1,1,0), (0,1,0)$. The set of the \emph{small rhombi} in a tetrahedron $\Delta^3(n)$ is obtained from this rhombus by the $S_4$-action on the tetrahedron and the translations. Each small rhombus has a long diagonal, which is the image of $(0,0,0), (1,1,0)$, and a short diagonal, which is the image of $(1,0,0), (0,1,0)$.

    
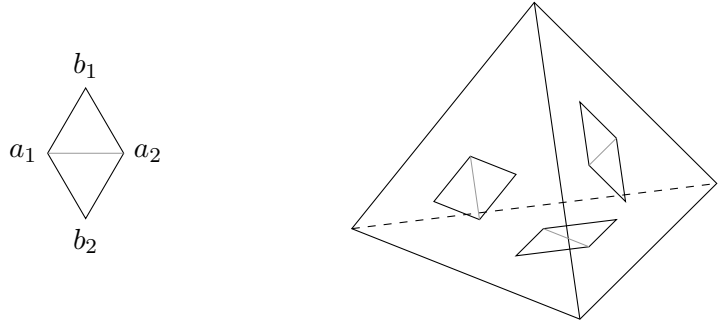
\begin{center}
    \begin{tikzpicture}
        \draw (0,0) -- (.5,.866) --(1,0) -- (.5, -.866) -- cycle;
        \draw[black!40, ultra thin] (0,0) -- (1,0);
        \node[anchor=east] at (0,0) {$a_1$};
        \node[anchor=west] at (1,0) {$a_2$};
        \node[anchor=south] at (.5,.866) {$b_1$};
        \node[anchor=north] at (.5,-.866) {$b_2$};
        
        \begin{scope}[xshift=4cm, yshift=-1cm, scale=1.2]
            \tikzmath{ 
            \l1=0; \l2=0; 
            \r1=4; \r2=.5;
            \u1=2; \u2=2.5;
            \d1=2.5; \d2=-1;}
            \draw (\l1,\l2) -- (\u1,\u2) -- (\r1,\r2) -- (\d1,\d2) -- cycle;
            \draw[dashed] (\l1,\l2) -- (\r1,\r2);
            \draw (\u1,\u2) -- (\d1,\d2);
                        
            \tikzmath{\k=.2;
            \v1=\k*(\u1-\l1); \v2=\k*(\u2-\l2);
            \w1=\k*(\d1-\l1); \w2=\k*(\d2-\l2);
            \o1=\v1+\w1; \o2=\v2+\w2;}
            \draw (\o1,\o2) -- (\o1+\v1,\o2+\v2) -- (\o1+\v1+\w1,\o2+\v2+\w2) -- (\o1+\w1,\o2+\w2) -- cycle; 
            \draw[black!40, ultra thin] (\o1+\v1,\o2+\v2) --(\o1+\w1,\o2+\w2);

             \tikzmath{\k=.2;
            \v1=\k*(\r1-\l1); \v2=\k*(\r2-\l2);
            \w1=\k*(\d1-\l1); \w2=\k*(\d2-\l2);
            \o1=2*\v1+\w1; \o2=2*\v2+\w2;}
            \draw (\o1,\o2) -- (\o1+\v1,\o2+\v2) -- (\o1+\w1,\o2+\w2) -- (\o1-\v1+\w1,\o2-\v2+\w2)-- cycle; 
            \draw[black!40, ultra thin] (\o1,\o2) --(\o1+\w1,\o2+\w2);

            \tikzmath{\k=.2;
            \v1=\k*(\u1-\d1); \v2=\k*(\u2-\d2);
            \w1=\k*(\r1-\d1); \w2=\k*(\r2-\d2);
            \o1=\d1+2*\v1+\w1; \o2=\d2+2*\v2+\w2;}
            \draw (\o1,\o2) -- (\o1+\v1,\o2+\v2) -- (\o1+\w1,\o2+\w2) -- (\o1-\v1+\w1,\o2-\v2+\w2)-- cycle; 
            \draw[black!40, ultra thin] (\o1,\o2) --(\o1+\w1,\o2+\w2);
        \end{scope}
    \end{tikzpicture}
    \end{center}

\begin{theorem}\label{theorem PN rhombi inequalities}
    The image of $m$ satisfies the rhombus inequalities: for each small rhombus in $\Delta^3(n)$ the sum of the values on the short diagonal is greater or equal to the sum of the values on the long diagonal, {\em i.e.}, $a_1+a_2\geqslant b_1+b_2$ in the notations above.
\end{theorem}

Consider the tetrahedron with vertices $(0,0,0), (2,0,0), (0,2,0), (0,0,2)$. The set of the \emph{small tetrahedra} is obtained from it by translations. Each small tetrahedron has three pairs of the opposite edges:

\begin{figure}[H]
    \centering
\begin{tikzpicture}[scale=.9]
        \tikzmath{ 
            \l1=0; \l2=0; 
            \r1=4; \r2=.5;
            \u1=2; \u2=2.5;
            \d1=2.5; \d2=-1;}
        \draw (\u1,\u2) node[above] {$\scriptstyle 000$} -- (\l1,\l2) node[left] {$\scriptstyle 200$} -- (\d1,\d2) node[below] {$\scriptstyle 020$} -- (\r1,\r2) node[right] {$\scriptstyle 002$}  -- cycle;
        \draw[dashed] (\l1,\l2) -- (\r1,\r2);
        \draw (\u1,\u2) -- (\d1,\d2);
        \fill(.5*\d1+.5*\u1, .5*\d2+.5*\u2) circle (2pt) node[right] {$a_1$};
        \fill(.5*\l1+.5*\r1, .5*\l2+.5*\r2) circle (2pt) node[above] {$a_2$};
        \fill(.5*\l1+.5*\u1, .5*\l2+.5*\u2) circle (2pt) node[left] {$b_1$};
        \fill(.5*\r1+.5*\d1, .5*\r2+.5*\d2) circle (2pt) node[right] {$b_2$};
        \fill(.5*\r1+.5*\u1, .5*\r2+.5*\u2) circle (2pt) node[right] {$c_1$};
        \fill(.5*\l1+.5*\d1, .5*\l2+.5*\d2) circle (2pt) node[below] {$c_2$};
\end{tikzpicture}
\caption{Small tetrahedron.}\label{picture small tetrahedron}
\end{figure}
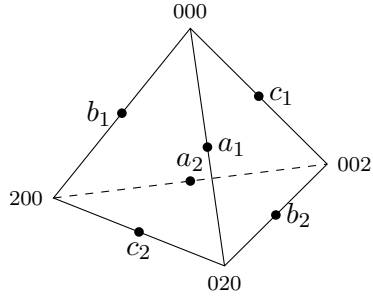

\begin{theorem}\label{theorem PN tetrahedra equalities}
    The image of $m$ satisfies the tropical tetrahedra equalities: the image of $m$ lies in a non-smooth locus of the function
    \[\max\{m_{i,j,k+1}+m_{i+1,j+1,k}, m_{i,j+1,k}+m_{i+1,j,k+1}, m_{i+1,j,k}+m_{i,j+1,k+1}\}.\]
\end{theorem}

In other words, if we denote the values on the edges of a small tetrahedron as in the Fig.\ref{picture small tetrahedron}, then the image of $m$ lies in the non-smooth locus of the function $\max\{a_1+a_2, b_1+b_2, c_1+c_2\}$.

\vspace{5pt}

For two integers $i\leqslant j$ denote by $[i,j]$ the set $\{i, i+1\ldots j-1, j\}$.
\begin{definition}\label{definition multi-GZ condition for PN}
We say that a weighting $w$ of a composite planar network $\Pi_1\circ\cdots\circ\Pi_k$ satisfies muli-Gelfand-Zeitlin condition if for any $\underline{\alpha}$ there is an $\underline{\alpha}$-multipath of maximal weight with sources $[n-\sum\alpha_i+1,n]$ and sinks $[1,\alpha_i]$. 
\end{definition}

\begin{theorem}\label{theorem PN octahedron recurrence}
    Suppose that the weightings $w_i\in \T^{\Pi_i}, i=1,2,3,$ are such that $(\Pi_i,w_i)$, $(\Pi_1\circ\Pi_2, w_1\circ w_2)$, $(\Pi_2\circ\Pi_3, w_2\circ w_3),$ $(\Pi_1\circ\Pi_2\circ\Pi_3, w_1\circ w_2\circ w_3)$ satisfy the multi-Gelfand-Zeitlin condition. Then the function $m$ of these weightings satisfies
\begin{equation}\label{equation: octahedron recurrence}
m_{i,j+1,k}+m_{i+1,j,k+1}=\max\{m_{i+1,j,k}+m_{i,j+1,k+1}, m_{i,j,k+1}+m_{i+1,j+1,k}\}.
\end{equation}
In other words, for any small tetrahedron, we have $a_1+a_2=\max\{b_1+b_2,c_1+c_2\}$ in the notations as in Fig.\ref{picture small tetrahedron}.
\end{theorem}

\begin{remark}\label{remark: octahedron recurrence and crystals}
    In the conditions of Theorem \ref{theorem PN octahedron recurrence}, the tetrahedron equalities give a recurrent functional relation between the values on a couple of faces corresponding to $\la(2),\la(3),\la(23)$ and $\la(1), \la(23),\la(123)$, and a couple corresponding to $\la(1),\la(2),\la(12)$ and $\la(12),\la(3),\la(123)$:
  
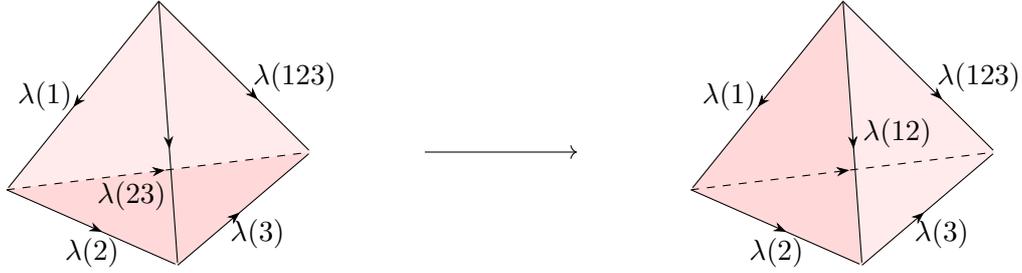
\begin{center}
    \begin{tikzpicture}
    \begin{scope}[xshift=9cm]
        \tikzmath{
            \l1=-2; \l2=-2.5; 
            \r1=2; \r2=-2;
            \u1=0; \u2=0;
            \d1=.25; \d2=-3.5;}
        \fill[red!8] (\u1,\u2) -- (\r1,\r2) -- (\d1,\d2) -- cycle;
        \fill[red!15] (\u1,\u2) -- (\d1,\d2) -- (\l1,\l2) -- cycle;
        
        \draw[-{Stealth}________________________________________]  (\u1,\u2) -- (\l1,\l2);
        \node[left] at (.5*\l1+.5*\u1, .5*\l2+.5*\u2) {$\la(1)$};
        
        \draw[-{Stealth}____________________________________________]  (\u1,\u2) -- (\d1,\d2); 
        \node[right] at (.5*\d1+.5*\u1, .5*\d2+.5*\u2) {$\la(12)$};
        
        \draw[-{Stealth}____________________________]  (\u1,\u2) -- (\r1,\r2);
        \node[right,xshift=3] at (.5*\r1+.5*\u1, .5*\r2+.5*\u2) {$\la(123)$};
        
        \draw[-{Stealth}_______________________________]  (\l1,\l2) -- (\d1,\d2);
        \node[below] at (.5*\l1+.5*\d1, .5*\l2+.5*\d2) {$\la(2)$};
        
        \draw[-{Stealth}___________________________________]  (\d1,\d2) -- (\r1,\r2);
        \node[below,xshift=5] at (.5*\d1+.5*\r1, .5*\d2+.5*\r2) {$\la(3)$};
        
        \draw[dashed, -{Stealth}_______________________________________________________]  (\l1,\l2) -- (\r1,\r2);
        
        \end{scope}

            \tikzmath{
            \l1=-2; \l2=-2.5; 
            \r1=2; \r2=-2;
            \u1=0; \u2=0;
            \d1=.25; \d2=-3.5;}
                    \draw[->] (\r1+1.5,\r2) -- (\r1+3.5,\r2);
        
        \fill[red!8] (\u1,\u2) -- (\r1,\r2) -- (\l1,\l2) -- cycle;
        \fill[red!15] (\l1,\l2) -- (\d1,\d2) -- (\r1,\r2) -- cycle;
        
        \draw[-{Stealth}________________________________________]  (\u1,\u2) -- (\l1,\l2);
        \node[left] at (.5*\l1+.5*\u1, .5*\l2+.5*\u2) {$\la(1)$};
        
        \draw[-{Stealth}____________________________________________]  (\u1,\u2) -- (\d1,\d2);

        \draw[-{Stealth}____________________________]  (\u1,\u2) -- (\r1,\r2);
        \node[right,xshift=3] at (.5*\r1+.5*\u1, .5*\r2+.5*\u2) {$\la(123)$};
        
        \draw[-{Stealth}_______________________________]  (\l1,\l2) -- (\d1,\d2);
        \node[below] at (.5*\l1+.5*\d1, .5*\l2+.5*\d2) {$\la(2)$};
        
        \draw[-{Stealth}___________________________________]  (\d1,\d2) -- (\r1,\r2);
        \node[below,xshift=5] at (.5*\d1+.5*\r1, .5*\d2+.5*\r2) {$\la(3)$};
        
        \draw[dashed, -{Stealth}_______________________________________________________]  (\l1,\l2) -- (\r1,\r2);
        \node[below,xshift=-10] at (.5*\r1+.5*\l1, .5*\r2+.5*\l2) {$\la(23)$};  
        
    \end{tikzpicture}
  \end{center}  
    This is the octahedron recurrence of \cite{HK}, which implements the action of the crystal associator on the multiplicities of $B_{\la(123)}$ in $B_{\la(1)}\otimes (B_{\la(2)} \otimes B_{\la(3)}) \xrightarrow[]{\sim} (B_{\la(1)}\otimes B_{\la(2)}) \otimes B_{\la(3)}$. To see that these recurrences are the same, one writes $a_1=\max\{b_1+b_2, c_1+c_2\}-a_2$ in the notations of Theorem \ref{theorem PN octahedron recurrence}:

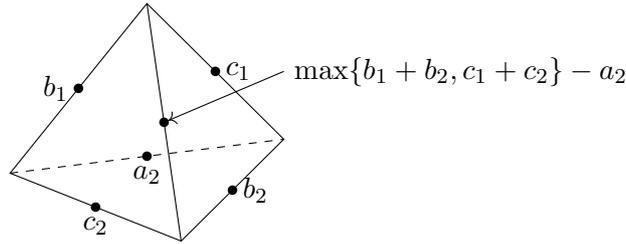
    \begin{center}
\begin{tikzpicture}[scale=.9]
        \tikzmath{ 
            \l1=0; \l2=0; 
            \r1=4; \r2=.5;
            \u1=2; \u2=2.5;
            \d1=2.5; \d2=-1;}
        \draw (\u1,\u2) -- (\l1,\l2) -- (\d1,\d2) -- (\r1,\r2) -- cycle;
        \draw[dashed] (\l1,\l2) -- (\r1,\r2);
        \draw (\u1,\u2) -- (\d1,\d2);
        \fill(.5*\d1+.5*\u1, .5*\d2+.5*\u2) circle (2pt);
        \draw[->__] (\r1,.5*\r2+.5*\u2) node[right] {$\max\{b_1+b_2, c_1+c_2\}-a_2$} -- (.5*\d1+.5*\u1, .5*\d2+.5*\u2);
        \fill(.5*\l1+.5*\r1, .5*\l2+.5*\r2) circle (2pt) node[below] {$a_2$};
        \fill(.5*\l1+.5*\u1, .5*\l2+.5*\u2) circle (2pt) node[left] {$b_1$};
        \fill(.5*\r1+.5*\d1, .5*\r2+.5*\d2) circle (2pt) node[right] {$b_2$};
        \fill(.5*\r1+.5*\u1, .5*\r2+.5*\u2) circle (2pt) node[right] {$c_1$};
        \fill(.5*\l1+.5*\d1, .5*\l2+.5*\d2) circle (2pt) node[below] {$c_2$};      
\end{tikzpicture}
\end{center}
 and compares it to the octahedron rule on the first page of \cite{HK}. See Example 4.1 in \cite{HK} for an algorithm.
\end{remark}

The next theorem is a converse to Theorems \ref{theorem PN rhombi inequalities} and \ref{theorem PN octahedron recurrence}. Before stating it let us introduce the following notion.
\begin{definition}
    The \emph{standard planar network} of rank $n$, denoted $\Pi_{\rm st}(n)$, is the following planar network:

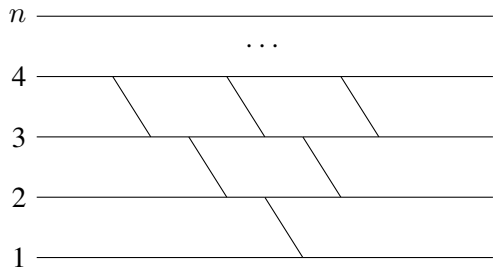
\begin{figure}[H]
    \centering
\begin{tikzpicture}[yscale=.8]
    \draw (0,1) node[left]{$n$} -- (6,1);
    \node at (3,0.5) {$\cdots$};
    \draw (0,0) node[left]{4} -- (6,0);
    \draw (0,-1) node[left]{3} -- (3.5,-1);
    \draw (3.5,-1) -- (6,-1);
    \draw (0,-2) node[left]{2} -- (3,-2);
    \draw (3,-2) -- (4,-2);
    \draw (4,-2) -- (6,-2);
    \draw (0,-3) node[left]{1} -- (3.5,-3);
    \draw (3.5,-3) -- (6,-3);
    \draw (1,0) -- (1.5,-1);
    \draw (2.5,0) -- (3,-1);
    \draw (4,0) -- (4.5,-1);
    \draw (2,-1) -- (2.5,-2);
    \draw (3.5,-1) -- (4,-2);
    \draw (3,-2) -- (3.5,-3);
\end{tikzpicture}
\caption{Standard planar network.}\label{fig:stplanar}
\end{figure}
\end{definition}

\begin{theorem}\label{theorem: m is onto the octahedron recurrence cone}
    Assume that $\Pi_i=\Pi_{\rm st}(n)$. Then $m$ maps the set of weightings of $\Pi_i$ satisfying the conditions of Theorem \ref{theorem PN octahedron recurrence} onto the cone defined by the rhombus inequalities and the octahedron recurrence equalities (\ref{equation: octahedron recurrence}).
\end{theorem}

The next result, combined with results in Section \ref{subsection: geometric preliminaries: BK potential and interlacing ineq}, explains why multi-Gelfand-Zeitlin condition is natural.

\begin{definition}\label{definition: standard network}
For a semiring $R$, let $W_R(\Pi_{\rm st})$ denote the set of $R$-weightings of $\Pi_{\rm st}$ which are nonzero only on the slanted edges or adjacent to the sinks edges. This set is isomorphic to $R^{\frac{n(n+1)}{2}}$. We call these edges {\em essential} and denote them by $essedges(\Pi_{\rm st})$.
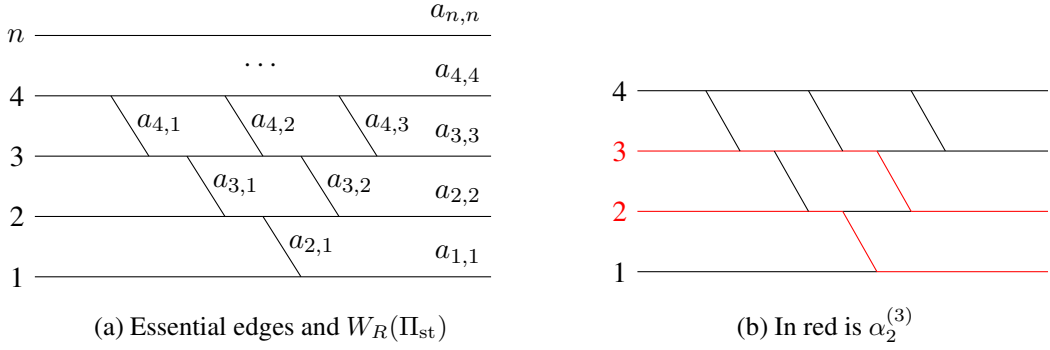
\begin{figure}[H]
    \centering
    \begin{subfigure}{.45\linewidth}
        \begin{tikzpicture}[yscale=.8]
            \draw (0,1) node[left]{$n$} -- (6,1) node[above left]{$a_{n,n}$};
            \node at (3,0.5) {$\cdots$};
            \draw (0,0) node[left]{4} -- (6,0) node[above left]{$a_{4,4}$};
            \draw (0,-1) node[left]{3} -- (3.5,-1);
            \draw (3.5,-1) -- (6,-1) node[above left]{$a_{3,3}$};
            \draw (0,-2) node[left]{2} -- (3,-2);
            \draw (3,-2) -- (4,-2);
            \draw (4,-2) -- (6,-2) node[above left]{$a_{2,2}$};
            \draw (0,-3) node[left]{1} -- (3.5,-3);
            \draw (3.5,-3) -- (6,-3) node[above left]{$a_{1,1}$};
            \draw (1,0) -- (1.5,-1) node [above, yshift=4,xshift=4]{$a_{4,1}$};
            \draw (2.5,0) -- (3,-1) node [above, yshift=4,xshift=4]{$a_{4,2}$};
            \draw (4,0) -- (4.5,-1) node [above, yshift=4,xshift=4]{$a_{4,3}$};
            \draw (2,-1) -- (2.5,-2) node [above, yshift=4,xshift=4]{$a_{3,1}$};
            \draw (3.5,-1) -- (4,-2) node [above, yshift=4,xshift=4]{$a_{3,2}$};
            \draw (3,-2) -- (3.5,-3) node [above, yshift=4,xshift=4]{$a_{2,1}$};
        \end{tikzpicture}
        \caption{Essential edges and $W_R(\Pi_{\rm st})$}
    \end{subfigure}
    \begin{subfigure}{.45\linewidth}
        \centering
        \begin{tikzpicture}[yscale=.8,xscale=.9]
            \draw (0,0) node[left]{4} -- (6,0);
            \draw[red] (0,-1) node[left]{3} -- (3.5,-1);
            \draw (3.5,-1) -- (6,-1);
            \draw[red] (0,-2) node[left]{2} -- (3,-2);
            \draw (3,-2) -- (4,-2);
            \draw[red] (4,-2) -- (6,-2);
            \draw (0,-3) node[left]{1} -- (3.5,-3);
            \draw[red] (3.5,-3) -- (6,-3);
            \draw (1,0) -- (1.5,-1);
            \draw (2.5,0) -- (3,-1);
            \draw (4,0) -- (4.5,-1);
            \draw (2,-1) -- (2.5,-2);
            \draw[red] (3.5,-1) -- (4,-2);
            \draw[red] (3,-2) -- (3.5,-3);
        \end{tikzpicture}
        \caption{In red is $\alpha_2^{(3)}$}
    \end{subfigure}
\end{figure}

For any $i\leqslant l$ there is a unique $i$-multipath in $\Pi_{\rm st}$ with sources $[l-i+1,l]$ and sinks $[1,i]$. We denote it by  $\alpha_i^{(l)}$. Denote by $\mathcal{A}\colon W_R(\Pi_{\rm st})\to R^{\frac{n(n+1)}{2}}$ the map with components $w\mapsto w(\alpha_i^{(l)})$. Since $\alpha_i^{(l)}=a_{l,i}+\cdots$, where $\cdots$ contains only $a_{l',i'}$ with $i'<i, l'<l$, the matrix of $\mathcal{A}$ is upper-triangular with 1-s on the diagonal. In particular, $\mathcal{A}$ is a linear isomorphism.
\end{definition}

\begin{definition}
    The \emph{Gelfand-Zeitlin cone} $\Delta_{GZ}\subset\T^{\frac{n(n+1)}{2}}$ is the set of tuples $(m_i^{(l)})_{0\leqslant i\leqslant l\leqslant n}$ satisfying the inequalities 
    \begin{equation}\label{equation: GZ inequalities} m_i^{(l+1)}+ m_{i-1}^{(l)}\geqslant m_{i-1}^{(l+1)}+ m_{i}^{(l)}, \qquad m_i^{(l+1)}+ m_{i}^{(l)}\geqslant m_{i+1}^{(l+1)}+ m_{i-1}^{(l)}\end{equation}
    for all $0< i\leqslant l \leqslant n$.

    Equivalently, $(\lambda_i^{(l)})=(m_i^{(l)}-m_{i-1}^{(l)})$ satisfies the interlacing inequalities \[\lambda_i^{(l)}\geqslant\lambda_i^{(l-1)}\geqslant\lambda_{i+1}^{(l)}.\]
\end{definition} 

\begin{definition}
    A weigthing of the standard planar network is \emph{Gelfand-Zeitlin} if $\mathcal{A}(w)\in\Delta_{GZ}$.
\end{definition}

\begin{lemma}\label{lemma A(w) in GZ implies multi-GZ condition on planar networks}
     Let $\Pi_1=\cdots=\Pi_k=\Pi_{\rm st}$, and suppose $w_i\in W_\T(\Pi_i)$ are Gelfand-Zeitlin for all $i$. Then for any $1\leqslant i\leqslant j\leqslant k$ the weighted networks $(\Pi_i\circ\cdots\circ\Pi_j,w_i\circ\cdots\circ w_{j})$ satisfy the multi-Gelfand-Zeitlin condition (Definition \ref{definition multi-GZ condition for PN}).
\end{lemma}

\subsection{Proofs}

In the proofs of Theorems \ref{theorem PN rhombi inequalities} and \ref{theorem PN tetrahedra equalities} we will need the following notions.

\begin{definition}
      Let ($p_1, p_2$) be a pair of multipaths in $(\Pi_1\circ\Pi_2\circ\Pi_3)$. Let $p_1\coprod p_2$ be a graph with the set of edges $edges(p_1)\coprod edges (p_2)$ (disjoint union), the set of vertices $vertices(p_1)\cup vertices(p_2)$, and the adjacency relation induced from  $(\Pi_1\circ\Pi_2\circ\Pi_3)$.  Recall that the edges of a planar network are oriented (from left to right). Introduce an equivalence relation on the set $edges(p_1)\coprod edges (p_2)$ generated by the following equivalences: $e_1\sim e_2$ if (1) they have a common source or target, or (2) they have a common vertex of valency exactly 2 in $p_1\coprod p_2$ ({\em i.e.}, they are the only edges of $p_1\sqcup p_2$ adjacent to this vertex):
  
      \begin{figure}[H]
          \centering
      \begin{tikzpicture}[scale=.55]
          \draw[->___________,green] (0,0) -- (2,-1);
          \draw[->___________,green](2,-1) -- (4,-1);
          \filldraw (2,-1) circle (2pt);
          \begin{scope}[xshift=7cm]
              \draw[->_____________,green] (0,0) -- (2,-1);
              \draw[->_____________,green] (0,-2) -- (2,-1);
              \filldraw (2,-1) circle (2pt);
        \end{scope}
        \begin{scope}[xshift=13cm]
              \draw[->__________________,green] (0,0) -- (2,-1);
              \draw[->__________________,green] (0,-2) -- (2,-1);
              \draw[->__________________,magenta] (2,-1) -- (4,0);
              \filldraw (2,-1) circle (2pt);
          \end{scope}
          \begin{scope}[xshift=20cm]
              \draw[->__________________,green] (0,0) -- (2,-1);
              \draw[->__________________,green] (0,-2) -- (2,-1);
              \draw[->__________________,magenta] (2,-1) -- (4,0);
              \draw[->__________________,magenta] (2,-1) -- (4,-2);
              \filldraw (2,-1) circle (2pt);
          \end{scope}
      \end{tikzpicture}
          \caption{Equivalent edges have the same colour.  \\ Pictures 2-4 illustrate the rule (1); picture 1, the rule (2).}
          \label{fig:enter-label}
      \end{figure}
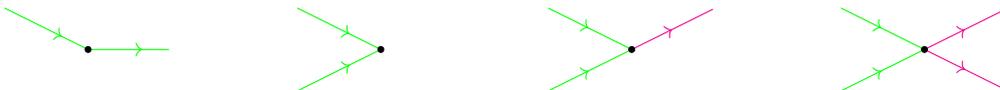
      
      A \emph{canonical path decomposition} of ($p_1, p_2$), denoted $\Theta(p_1,p_2)$, is the decomposition of the set $edges(p_1)\coprod edges (p_2)$ into the equivalence classes of this relation, which we call \emph{components}. 

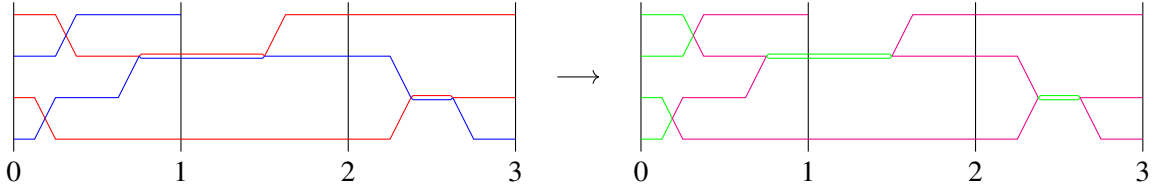
\begin{figure}[H]
    \centering 
\begin{tikzpicture}[scale=.55]
    \draw (0,.3) -- (0,-3.3) node[below]{0};
    \draw (4,.3) -- (4,-3.3) node[below]{1};
    \draw (8,.3) -- (8,-3.3) node[below]{2};
    \draw (12,.3) -- (12,-3.3) node[below]{3};
    \draw[red] (0,0) -- (1,0) -- (1.25,-.5);
    \draw[blue] (1.25,-.5) -- (1,-1) -- (0,-1);
    \draw[red] (0,-2) -- (.5,-2) -- (.75,-2.5);
    \draw[blue] (.75,-2.5) -- (.5,-3) -- (0,-3);
    \draw[blue] (4,0) -- (1.5, 0) -- (1.25,-.5);
    \draw[red] (1.25,-.5) -- (1.5, -1) -- (3,-1);
    \draw[blue] (3,-1) -- (2.5,-2) -- (1,-2) -- (.75,-2.5);
    \draw[red] (.75,-2.5) -- (1,-3) -- (9,-3) -- (9.5,-2);
    \draw[blue] (9.5,-2) -- (9,-1) -- (6,-1);
    \draw[red] (6,-1) -- (6.5,0) -- (12,0);
    \draw[red] (3,-1) -- (3.05,-.95) -- (5.95,-.95); \draw[red] (5.95,-.95) -- (6,-1);
    \draw[blue] (6,-1) -- (5.95,-1.05) -- (3.05,-1.05); \draw[blue] (3.05,-1.05) -- (3,-1);

    \draw[red] (9.5,-2) -- (9.55,-1.95) -- (10.45,-1.95);\draw[red] (10.45,-1.95) -- (10.5,-2);
    \draw[blue] (10.5,-2) -- (10.45,-2.05) -- (9.55,-2.05);\draw[blue] (9.55,-2.05) -- (9.5,-2);
    \draw[blue] (12,-3) -- (11,-3) -- (10.5,-2);
    \draw[red] (10.5,-2) -- (12,-2);  
    \draw[->] (13, -1.5) -- (14,-1.5);
\begin{scope}[xshift=15cm]
    \draw (0,.3) -- (0,-3.3) node[below]{0};
    \draw (4,.3) -- (4,-3.3) node[below]{1};
    \draw (8,.3) -- (8,-3.3) node[below]{2};
    \draw (12,.3) -- (12,-3.3) node[below]{3};
    \draw[green] (0,0) -- (1,0) -- (1.25,-.5);
    \draw[green] (1.25,-.5) -- (1,-1) -- (0,-1);
    \draw[green] (0,-2) -- (.5,-2) -- (.75,-2.5);
    \draw[green] (.75,-2.5) -- (.5,-3) -- (0,-3);
    \draw[magenta] (4,0) -- (1.5, 0) -- (1.25,-.5);
    \draw[magenta] (1.25,-.5) -- (1.5, -1) -- (3,-1);
    \draw[magenta] (3,-1) -- (2.5,-2) -- (1,-2) -- (.75,-2.5);
    \draw[magenta] (.75,-2.5) -- (1,-3) -- (9,-3) -- (9.5,-2);
    \draw[magenta] (9.5,-2) -- (9,-1) -- (6,-1);
    \draw[magenta] (6,-1) -- (6.5,0) -- (12,0);
    \draw[green] (3,-1) -- (3.05,-.95) -- (5.95,-.95); \draw[green] (5.95,-.95) -- (6,-1);
    \draw[green] (6,-1) -- (5.95,-1.05) -- (3.05,-1.05);
    \draw[green] (3.05,-1.05) -- (3,-1);

    \draw[green] (9.5,-2) -- (9.55,-1.95) -- (10.45,-1.95);\draw[green] (10.45,-1.95) -- (10.5,-2);
    \draw[green] (10.5,-2) -- (10.45,-2.05) -- (9.55,-2.05);\draw[green] (9.55,-2.05) -- (9.5,-2);
    \draw[magenta] (12,-3) -- (11,-3) -- (10.5,-2);
    \draw[magenta] (10.5,-2) -- (12,-2);    
\end{scope}
\end{tikzpicture}
    \caption{A pair of a \textcolor{red}{(0,0,2)} and a \textcolor{blue}{(1,0,1)} multipath, its canonical path decomposition (6 components).}
    \label{picture: cpd}
\end{figure}  

\end{definition}

     Any component of $\Theta$ is an unoriented path (closed or open), and the endpoints of any open path lie on the verticals separating $\Pi_i$'s. Enumerate these verticals from left to right by $0,\ldots, 3$. We call the open path components of $\Theta$ with the endpoints on different verticals \emph{essential}. For example, in Fig.\ref{picture: cpd}, there is 1 essential and 5 non-essential components.

If we reverse the orientation in the multipath $p_2$, it gives a path orientation to each component in  $\Theta(p_1, p_2)$. Vice versa, giving a path orientation to each component defines a decomposition of $\Theta$ into a couple of multipaths. Therefore we have a bijective map $2^{\# components(\Theta)}\xrightarrow{(p_+, p_-)}$ \{Couples of multipaths with c.p.d. $\Theta$\}. 

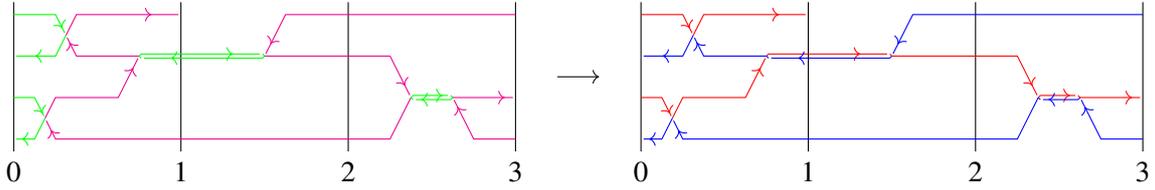
\begin{figure}[H]
    \centering 
\begin{tikzpicture}[scale=.55]
    \draw (0,.3) -- (0,-3.3) node[below]{0};
    \draw (4,.3) -- (4,-3.3) node[below]{1};
    \draw (8,.3) -- (8,-3.3) node[below]{2};
    \draw (12,.3) -- (12,-3.3) node[below]{3};
    \draw[green, ->___] (0,0) -- (1,0) -- (1.25,-.5);
    \draw[green, ->________] (1.25,-.5) -- (1,-1) -- (0,-1);
    \draw[green, ->__] (0,-2) -- (.5,-2) -- (.75,-2.5);
    \draw[green, ->___] (.75,-2.5) -- (.5,-3) -- (0,-3);
    \draw[magenta, ___________<-] (4,0) -- (1.5, 0) -- (1.25,-.5);
    \draw[magenta,___<-] (1.25,-.5) -- (1.5, -1) -- (3,-1);
    \draw[magenta,_____<-] (3,-1) -- (2.5,-2) -- (1,-2) -- (.75,-2.5);
    \draw[magenta,_____<-] (.75,-2.5) -- (1,-3) -- (9,-3) -- (9.5,-2);
    \draw[magenta,_____<-] (9.5,-2) -- (9,-1) -- (6,-1);
    \draw[magenta,_____<-] (6,-1) -- (6.5,0) -- (12,0);
    \draw[green,->___________] (3,-1) -- (3.05,-.95) -- (5.95,-.95); \draw[green] (5.95,-.95) -- (6,-1);
    \draw[green,->___________] (6,-1) -- (5.95,-1.05) -- (3.05,-1.05);
    \draw[green] (3.05,-1.05) -- (3,-1);

    \draw[green, ->___] (9.5,-2) -- (9.55,-1.95) -- (10.45,-1.95);\draw[green] (10.45,-1.95) -- (10.5,-2);
    \draw[green,->___] (10.5,-2) -- (10.45,-2.05) -- (9.55,-2.05);\draw[green] (9.55,-2.05) -- (9.5,-2);
    \draw[magenta,->____] (12,-3) -- (11,-3) -- (10.5,-2);
    \draw[magenta,->____] (10.5,-2) -- (12,-2);   
    \draw[->] (13, -1.5) -- (14,-1.5);
\begin{scope}[xshift=15cm]
    \draw (0,.3) -- (0,-3.3) node[below]{0};
    \draw (4,.3) -- (4,-3.3) node[below]{1};
    \draw (8,.3) -- (8,-3.3) node[below]{2};
    \draw (12,.3) -- (12,-3.3) node[below]{3};
    \draw[red, ->___] (0,0) -- (1,0) -- (1.25,-.5);
    \draw[blue, ->________] (1.25,-.5) -- (1,-1) -- (0,-1);
    \draw[red, ->__] (0,-2) -- (.5,-2) -- (.75,-2.5);
    \draw[blue, ->___] (.75,-2.5) -- (.5,-3) -- (0,-3);
    \draw[red, ___________<-] (4,0) -- (1.5, 0) -- (1.25,-.5);
    \draw[blue,___<-] (1.25,-.5) -- (1.5, -1) -- (3,-1);
    \draw[red,_____<-] (3,-1) -- (2.5,-2) -- (1,-2) -- (.75,-2.5);
    \draw[blue,_____<-] (.75,-2.5) -- (1,-3) -- (9,-3) -- (9.5,-2);
    \draw[red,_____<-] (9.5,-2) -- (9,-1) -- (6,-1);
    \draw[blue,_____<-] (6,-1) -- (6.5,0) -- (12,0);
    \draw[red,->___________] (3,-1) -- (3.05,-.95) -- (5.95,-.95); \draw[red] (5.95,-.95) -- (6,-1);
    \draw[blue,->___________] (6,-1) -- (5.95,-1.05) -- (3.05,-1.05); \draw[blue] (3.05,-1.05) -- (3,-1);
    \draw[red, ->___] (9.5,-2) -- (9.55,-1.95) -- (10.45,-1.95);\draw[red] (10.45,-1.95) -- (10.5,-2);
    \draw[blue,->___] (10.5,-2) -- (10.45,-2.05) -- (9.55,-2.05);\draw[blue] (9.55,-2.05) -- (9.5,-2);
    \draw[blue,->____] (12,-3) -- (11,-3) -- (10.5,-2);
    \draw[red,->____] (10.5,-2) -- (12,-2);   
\end{scope}
\end{tikzpicture}
    \caption{An orientation on the components and the corresponding pair $(\textcolor{red}{p_+},\textcolor{blue}{p_-})$.}
    \label{picture: orientation on cpd -> decomposition into (p_+, p_-)}
\end{figure}  

Given a path orientation on the essential components of $\Theta$, for each $i\neq j$ define the following numbers: $Q_{ij}=$\#\{path components going from $i$-th to $j$-th vertical\}. For example, in Fig.\ref{picture: orientation on cpd -> decomposition into (p_+, p_-)}, $Q_{31}=1$. These numbers define an oriented graph with $Q_{ij}$ arrows from $i$ to $j$. Conversely, an orientation on the corresponding unoriented graph defines an orientation on the essential components of $\Theta$.

\begin{lemma}\label{lemma alpha-beta = Q-Q}
    If $(p_+, p_-)\in P_{\alpha_1,\alpha_2,\alpha_3}\times P_{\beta_1,\beta_2,\beta_3}$, then
    $\alpha_i-\beta_i = \displaystyle\sum_j Q_{ji}-\displaystyle\sum_j Q_{ij}$.
\end{lemma}
\begin{proof}
    Let us fix a left-to-right orientation on $p_+$ and a right-to-left orientation on $p_-$. Then
    \begin{align*}
         \alpha_i&=\#\{e\in \edge(p_+) \mid t(e)\in i\text{-th vertical}\}-\#\{e\in \edge(p_+) \mid s(e)\in i\text{-th vertical}\}, \\
        \beta_i&=\#\{e\in \edge(p_-) \mid s(e)\in i\text{-th vertical}\}-\#\{e\in \edge(p_-) \mid t(e)\in i\text{-th vertical}\},
    \end{align*}
    hence $\alpha_i-\beta_i=\#\{e\in \edge(\Theta) \mid t(e)\in i\text{-th vertical}\} -\#\{e\in \edge(\Theta) \mid s(e)\in i\text{-th vertical}\}$.

    On the other hand, the last quantity equals the number of path components ending at $i$-th vertical minus the number of path components starting at $i$-th vertical, {\em i.e.}, $\sum_j Q_{ji}-\sum_j Q_{ij}$.
\end{proof}

    It will be convenient to assume that $P_n(\Pi)$ is nonempty (hence all the sets $P_k(\Pi), k\leqslant n$ are nonempty). We can  achieve this by adding edges of weight $-\infty$, this does not affect the values $m_{\underline{\alpha}}$.

\begin{proof}[Proof of Theorem \ref{theorem PN rhombi inequalities}]
    Take a small rhombus with the vertices labeled as follows:
    
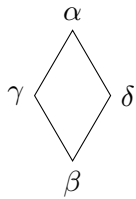
\begin{center}
    \begin{tikzpicture}
        \draw (0,0) -- (0.5,0.866) --(1,0) -- (0.5, -0.866) -- cycle;
        \node[anchor=east] at (0,0) {$\gamma$};
        \node[anchor=west] at (1,0) {$\delta$};
        \node[anchor=south] at (.5,.866) {$\alpha$};
        \node[anchor=north] at (.5,-.866) {$\beta$};
    \end{tikzpicture}
    \end{center}
    where $\alpha,\beta, \gamma, \delta$ are the triples of integers. Notice that in any small rhombus $\alpha+\beta=\gamma+\delta$, and $\delta-\gamma$ is a permutation of $\pm(1,-1,0)$ or $\pm(1,0,0)$.
    
    The rhombus inequality reads: $m_\alpha+m_\beta\leqslant m_\gamma+m_\delta$. It is enough to prove the following statement: given a pair of multipaths $(p_1, p_2)\in P_\alpha\times P_\beta$, there exists an orientation on the components of $\Theta(p_1,p_2)$ such that $(p_+,p_-)\in P_\gamma\times P_\delta$.
    
Take $(p_1, p_2)\in P_\alpha\times P_\beta $ and consider its canonical path decomposition $\Theta$. Then Lemma \ref{lemma alpha-beta = Q-Q} gives $\sum_j Q_{ji}+\sum_j Q_{ij}\equiv\alpha_i+\beta_i\equiv\gamma_i+\delta_i\text{ mod }2$. The possible values of $\gamma-\delta$ imply that among the numbers $(\sum_j Q_{j1}+\sum_j Q_{1j}, \sum_j Q_{j2}+\sum_j Q_{2j},\sum_j Q_{j3}+\sum_j Q_{3j})$ there are one or two odd. 
Construct an unoriented graph on vertices $0,\ldots, 3$ with $Q_{ij}+Q_{ji}$ edges between $i$ and $j$. Then among the  vertices $1,2,3$ there are two or one odd-valent vertices. Since the sum of valencies of all the vertices is even, there are exactly 2 vertices with odd valency among $0,1,2,3$. Therefore the graph can be decomposed into several closed paths and an open path connecting two vertices with odd valency. Give an orientation to the graph in such a way that these paths become oriented paths. This defines an orientation on the essential components of $\Theta$. Orient  the non-essential components of $\Theta$ in any way. 
In the resulting orientation $\big(
\sum_j Q'_{ji}-
\sum_j Q'_{ij}\big)_{i=1,\ldots, 3} \equiv \pm(\gamma-\delta)\text{ mod }2$ and is a permutation of $(1,0,0)$ or $(1,-1,0)$, therefore it equals $\pm(\gamma-\delta)$. Changing orientation if necessary we may assume that it equals $\gamma-\delta$. Therefore if $(p_+,p_-)\in P_{\gamma'}\times P_{\delta'}$, then $\gamma'-\delta'=\gamma-\delta$. Moreover, $\gamma'+\delta'=\alpha+\beta=\gamma+\delta$.
This implies $(p_+, p_-)\in P_\gamma\times P_\delta $.
\end{proof}

Let us illustrate the algorithm  of the proof of theorem \ref{theorem PN rhombi inequalities}.

Input: $(p_1,p_2)\in P_{(0,0,2)}\times P_{(1,1,0)}$.

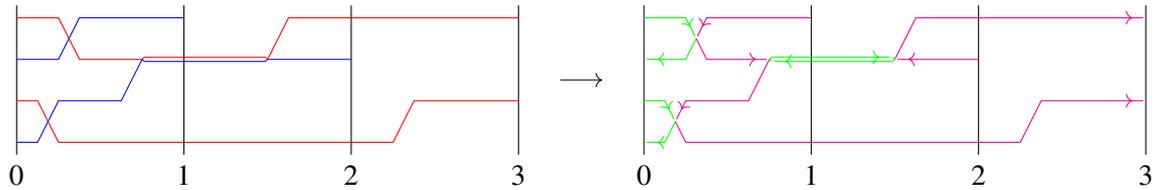
\begin{tikzpicture}[scale=.55]
    \draw (0,.3) -- (0,-3.3) node[below]{0};
    \draw (4,.3) -- (4,-3.3) node[below]{1};
    \draw (8,.3) -- (8,-3.3) node[below]{2};
    \draw (12,.3) -- (12,-3.3) node[below]{3};
    \draw[red] (0,0) -- (1,0) -- (1.25,-.5);
    \draw[blue] (1.25,-.5) -- (1,-1) -- (0,-1);
    \draw[red] (0,-2) -- (.5,-2) -- (.75,-2.5);
    \draw[blue] (.75,-2.5) -- (.5,-3) -- (0,-3);
    \draw[blue] (4,0) -- (1.5, 0) -- (1.25,-.5);
    \draw[red] (1.25,-.5) -- (1.5, -1) -- (3,-1);
    \draw[blue] (3,-1) -- (2.5,-2) -- (1,-2) -- (.75,-2.5);
    \draw[red] (.75,-2.5) -- (1,-3) -- (9,-3) -- (9.5,-2) -- (12,-2);
    \draw[blue] (8,-1) -- (6,-1);
    \draw[red] (6,-1) -- (6.5,0) -- (12,0);
    \draw[red] (3,-1) -- (3.05,-.95) -- (5.95,-.95); \draw[red] (5.95,-.95) -- (6,-1);
    \draw[blue] (6,-1) -- (5.95,-1.05) -- (3.05,-1.05); \draw[blue] (3.05,-1.05) -- (3,-1);

    \draw[->] (13, -1.5) -- (14,-1.5);
\begin{scope}[xshift=15cm]
    \draw (0,.3) -- (0,-3.3) node[below]{0};
    \draw (4,.3) -- (4,-3.3) node[below]{1};
    \draw (8,.3) -- (8,-3.3) node[below]{2};
    \draw (12,.3) -- (12,-3.3) node[below]{3};
    \draw[green,->_____] (0,0) -- (1,0) -- (1.25,-.5);
    \draw[green,->_____] (1.25,-.5) -- (1,-1) -- (0,-1);
    \draw[green,->_____] (0,-2) -- (.5,-2) -- (.75,-2.5);
    \draw[green,->_____] (.75,-2.5) -- (.5,-3) -- (0,-3);
    \draw[magenta,->_____] (4,0) -- (1.5, 0) -- (1.25,-.5);
    \draw[magenta,->_____] (1.25,-.5) -- (1.5, -1) -- (3,-1);
    \draw[magenta,->_____] (3,-1) -- (2.5,-2) -- (1,-2) -- (.75,-2.5);
    \draw[magenta,->_____] (.75,-2.5) -- (1,-3) -- (9,-3) -- (9.5,-2) -- (12,-2);
    \draw[magenta,->_____] (8,-1) -- (6,-1);
    \draw[magenta,->_____] (6,-1) -- (6.5,0) -- (12,0);
    \draw[green,->_____] (3,-1) -- (3.05,-.95) -- (5.95,-.95); \draw[green] (5.95,-.95) -- (6,-1);
    \draw[green,->_____] (6,-1) -- (5.95,-1.05) -- (3.05,-1.05);\draw[green] (3.05,-1.05) -- (3,-1);
\end{scope}
\end{tikzpicture}

Step 1: take a canonical path decomposition of $(p_1,p_2)$.

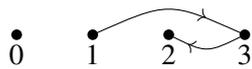
\vspace{.7cm}

\begin{tikzpicture}
    \fill (0,0) circle (2pt) node[below] {0};
    \fill (1,0) circle (2pt) node[below] {1};
    \fill (2,0) circle (2pt) node[below] {2};
    \fill (3,0) circle (2pt) node[below] {3};
    \draw[->________________] (1,0) .. controls (2,.45) .. (3,0);
    \draw[________<-] (2,0) .. controls (2.5,-.25) .. (3,0);
\end{tikzpicture}

Step 2: form a graph of essential components. Decompose it into several cycles and a path and orient them in any way.

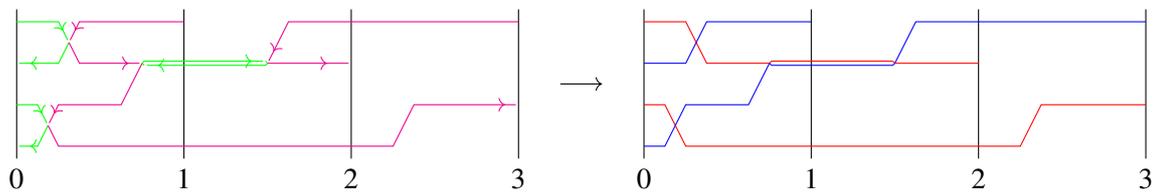
\vspace{.8cm}

\begin{tikzpicture}[scale=.55]
    \draw (0,.3) -- (0,-3.3) node[below]{0};
    \draw (4,.3) -- (4,-3.3) node[below]{1};
    \draw (8,.3) -- (8,-3.3) node[below]{2};
    \draw (12,.3) -- (12,-3.3) node[below]{3};
    \draw[green,->_____] (0,0) -- (1,0) -- (1.25,-.5);
    \draw[green,->_____] (1.25,-.5) -- (1,-1) -- (0,-1);
    \draw[green,->_____] (0,-2) -- (.5,-2) -- (.75,-2.5);
    \draw[green,->_____] (.75,-2.5) -- (.5,-3) -- (0,-3);
    \draw[magenta,->_____] (4,0) -- (1.5, 0) -- (1.25,-.5);
    \draw[magenta,->_____] (1.25,-.5) -- (1.5, -1) -- (3,-1);
    \draw[magenta,->_____] (3,-1) -- (2.5,-2) -- (1,-2) -- (.75,-2.5);
    \draw[magenta,->_____] (.75,-2.5) -- (1,-3) -- (9,-3) -- (9.5,-2) -- (12,-2);
    \draw[magenta,_________<-] (8,-1) -- (6,-1);
    \draw[magenta,_____<-] (6,-1) -- (6.5,0) -- (12,0);
    \draw[green,->_____] (3,-1) -- (3.05,-.95) -- (5.95,-.95); \draw[green] (5.95,-.95) -- (6,-1);
    \draw[green,->_____] (6,-1) -- (5.95,-1.05) -- (3.05,-1.05);\draw[green] (3.05,-1.05) -- (3,-1);
    \draw[->] (13,-1.5) -- (14,-1.5);
    
    \begin{scope}[xshift=15cm]
    \draw (0,.3) -- (0,-3.3) node[below]{0};
    \draw (4,.3) -- (4,-3.3) node[below]{1};
    \draw (8,.3) -- (8,-3.3) node[below]{2};
    \draw (12,.3) -- (12,-3.3) node[below]{3};
    \draw[red] (0,0) -- (1,0) -- (1.25,-.5);
    \draw[blue] (1.25,-.5) -- (1,-1) -- (0,-1);
    \draw[red] (0,-2) -- (.5,-2) -- (.75,-2.5);
    \draw[blue] (.75,-2.5) -- (.5,-3) -- (0,-3);
    \draw[blue] (4,0) -- (1.5, 0) -- (1.25,-.5);
    \draw[red] (1.25,-.5) -- (1.5, -1) -- (3,-1);
    \draw[blue] (3,-1) -- (2.5,-2) -- (1,-2) -- (.75,-2.5);
    \draw[red] (.75,-2.5) -- (1,-3) -- (9,-3) -- (9.5,-2) -- (12,-2);
    \draw[red] (8,-1) -- (6,-1);
    \draw[blue] (6,-1) -- (6.5,0) -- (12,0);
    \draw[red] (3,-1) -- (3.05,-.95) -- (5.95,-.95); \draw[red] (5.95,-.95) -- (6,-1);
    \draw[blue] (6,-1) -- (5.95,-1.05) -- (3.05,-1.05); \draw[blue] (3.05,-1.05) -- (3,-1);
    \end{scope}
\end{tikzpicture}

Step 3: give an orientation to the components of $\Theta(p_1,p_2)$ according to the orientation on the graph of step 2. Reconstruct $(p_+,p_-)$.

Output: $(p_+,p_-)\in P_{(0,1,1)}\times P_{(1,0,1)}$.

\vspace{7pt}

\begin{proof}[Proof of Theorem \ref{theorem PN tetrahedra equalities}]
The non-smooth locus of the function in question is the intersection of the three sets given by the inequalities  
\begin{gather*}
 m_{(i,j,k+1)}+m_{(i+1,j+1,k)}\leqslant\max\{ m_{(i,j+1,k)}+m_{(i+1,j,k+1)}, m_{(i+1,j,k)}+m_{(i,j+1,k+1)}\}, \\
 m_{(i,j+1,k)}+m_{(i+1,j,k+1)}\leqslant\max\{ m_{(i,j,k+1)}+m_{(i+1,j+1,k)}, m_{(i+1,j,k)}+m_{(i,j+1,k+1)}\},\\
 m_{(i+1,j,k)}+m_{(i,j+1,k+1)}\leqslant\max\{m_{(i,j+1,k)}+m_{(i+1,j,k+1)}, m_{(i,j,k+1)}+m_{(i+1,j+1,k)}\}.
 \end{gather*}

Take $(p_1, p_2)\in P_{(i,j,k+1)}\times P_{(i+1,j+1,k)}$ and consider its canonical path decomposition $\Theta$. To prove the statement, it is enough to show that $\Theta$ can be decomposed into $(p_+, p_-)$ belonging to either  $P_{(i+1,j,k)}\times P_{(i,j+1,k+1)}$ or $P_{(i,j+1,k)}\times P_{(i+1,j,k+1)}$. By Lemma \ref{lemma alpha-beta = Q-Q}, we have $\sum_k Q_{kl}+\sum_k Q_{lk} = 1$ mod~$2$ for $l=1,2,3$.
 This means that in the unoriented graph on vertices $0,\ldots, 3$ with $Q_{kl}+Q_{lk}$ edges between $k$ and $l$ all the vertices have odd valency. Thus it can be decomposed into several closed paths and two non-intersecting open paths. Give an orientation to the graph in such a way that these paths become oriented paths, $3$ being a sink and $0$ a source. Then the numbers $Q'_{kl}$ corresponding to this oriented graph satisfy $(\sum_k Q'_{kl}-\sum_k Q'_{lk})_{l=0,\ldots, 3} = (-1,x,y,1)$, where $(x,y)=\pm(1,-1)$. Orient the essential components of $\Theta$ in accordance with the orientation on the graph, and the non-essential components in any way. 
 By the construction, the resulting $(p_+, p_-)$ belongs to $P_\alpha\times P_\beta$ with $\alpha-\beta = (-1,1,1)$ or $(1,-1,1)$ and $\alpha+\beta = (2i+1,2j+1,2k+1)$. This implies $(\alpha, \beta) = \big((i,j+1,k+1),(i+1,j,k)\big)$ or $\big((i+1,j,k+1),(i,j+1,k)\big)$ which proves the statement. 
\end{proof}

Let us illustrate the algorithm of the proof of Theorem \ref{theorem PN tetrahedra equalities}.

Input: $(p_1,p_2)\in P_{(0,0,1)}\times P_{(1,1,0)}$.

\begin{tikzpicture}[scale=.55]
    \draw (0,.3) -- (0,-2.3) node[below]{0};
    \draw (4,.3) -- (4,-2.3) node[below]{1};
    \draw (8,.3) -- (8,-2.3) node[below]{2};
    \draw (12,.3) -- (12,-2.3) node[below]{3};
    \draw[red] (0,0) -- (1,0) -- (1.25,-.5);
    \draw[blue] (1.25,-.5) -- (1,-1) -- (0,-1);   
    \draw[blue] (4,0) -- (1.5, 0) -- (1.25,-.5);
    \draw[red] (1.25,-.5) -- (1.5, -1) -- (3,-1);
    \draw[blue] (3,-1) -- (2.5,-2) -- (0,-2);
    \draw[blue] (8,-1) -- (6,-1);
    \draw[red] (6,-1) -- (6.5,0) -- (12,0);
    \draw[red] (3,-1) -- (3.05,-.95) -- (5.95,-.95); \draw[red] (5.95,-.95) -- (6,-1);
    \draw[blue] (6,-1) -- (5.95,-1.05) -- (3.05,-1.05); \draw[blue] (3.05,-1.05) -- (3,-1);

    \draw[->] (13, -1.5) -- (14,-1.5);
\begin{scope}[xshift=15cm]
    \draw (0,.3) -- (0,-2.3) node[below]{0};
    \draw (4,.3) -- (4,-2.3) node[below]{1};
    \draw (8,.3) -- (8,-2.3) node[below]{2};
    \draw (12,.3) -- (12,-2.3) node[below]{3};
    \draw[green,->_____] (0,0) -- (1,0) -- (1.25,-.5);
    \draw[green,->_____] (1.25,-.5) -- (1,-1) -- (0,-1);
    \draw[magenta,->_____] (4,0) -- (1.5, 0) -- (1.25,-.5);
    \draw[magenta,->_____] (1.25,-.5) -- (1.5, -1) -- (3,-1);
    \draw[magenta,->_____] (3,-1) -- (2.5,-2) -- (0,-2);
    \draw[green,->_____] (3,-1) -- (3.05,-.95) -- (5.95,-.95); \draw[green] (5.95,-.95) -- (6,-1);
    \draw[green,->_____] (6,-1) -- (5.95,-1.05) -- (3.05,-1.05);\draw[green] (3.05,-1.05) -- (3,-1);
     \draw[magenta,->_____] (8,-1) -- (6,-1);
    \draw[magenta,->_____] (6,-1) -- (6.5,0) -- (12,0);
\end{scope}
\end{tikzpicture}

Step 1: take a canonical path decomposition of $(p_1,p_2)$.

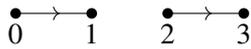
\vspace{.7cm}

\begin{tikzpicture}
    \fill (0,0) circle (2pt) node[below] {0};
    \fill (1,0) circle (2pt) node[below] {1};
    \fill (2,0) circle (2pt) node[below] {2};
    \fill (3,0) circle (2pt) node[below] {3};
    \draw[->____________] (0,0) -- (1,0);
    \draw[->____________] (2,0)-- (3,0);
\end{tikzpicture}

Step 2: form a graph of essential components. Decompose it into several cycles and two non-intersecting paths. Orient them in such a way that $3$ is a sink and $0$ is a source.

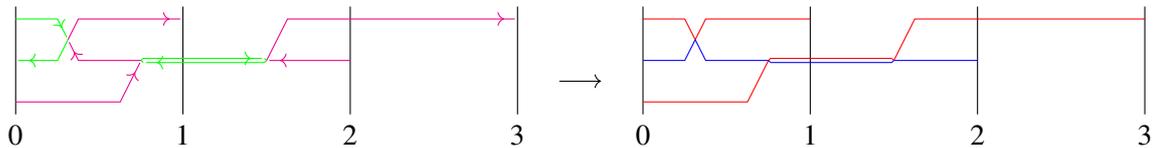
\vspace{.7cm}

\begin{tikzpicture}[scale=.55]
    \begin{scope}[xshift=15cm]
    \draw (0,.3) -- (0,-2.3) node[below]{0};
    \draw (4,.3) -- (4,-2.3) node[below]{1};
    \draw (8,.3) -- (8,-2.3) node[below]{2};
    \draw (12,.3) -- (12,-2.3) node[below]{3};
    \draw[red] (0,0) -- (1,0) -- (1.25,-.5);
    \draw[blue] (1.25,-.5) -- (1,-1) -- (0,-1);   
    \draw[red] (4,0) -- (1.5, 0) -- (1.25,-.5);
    \draw[blue] (1.25,-.5) -- (1.5, -1) -- (3,-1);
    \draw[red] (3,-1) -- (2.5,-2) -- (0,-2);
    \draw[blue] (8,-1) -- (6,-1);
    \draw[red] (6,-1) -- (6.5,0) -- (12,0);
    \draw[red] (3,-1) -- (3.05,-.95) -- (5.95,-.95); \draw[red] (5.95,-.95) -- (6,-1);
    \draw[blue] (6,-1) -- (5.95,-1.05) -- (3.05,-1.05); \draw[blue] (3.05,-1.05) -- (3,-1);
    \end{scope}
    \draw[->] (13, -1.5) -- (14,-1.5);

    \draw (0,.3) -- (0,-2.3) node[below]{0};
    \draw (4,.3) -- (4,-2.3) node[below]{1};
    \draw (8,.3) -- (8,-2.3) node[below]{2};
    \draw (12,.3) -- (12,-2.3) node[below]{3};
    \draw[green,->_____] (0,0) -- (1,0) -- (1.25,-.5);
    \draw[green,->_____] (1.25,-.5) -- (1,-1) -- (0,-1);
    \draw[magenta,_____<-] (4,0) -- (1.5, 0) -- (1.25,-.5);
    \draw[magenta,_____<-] (1.25,-.5) -- (1.5, -1) -- (3,-1);
    \draw[magenta,_____<-] (3,-1) -- (2.5,-2) -- (0,-2);
    \draw[green,->_____] (3,-1) -- (3.05,-.95) -- (5.95,-.95); \draw[green] (5.95,-.95) -- (6,-1);
    \draw[green,->_____] (6,-1) -- (5.95,-1.05) -- (3.05,-1.05);\draw[green] (3.05,-1.05) -- (3,-1);
     \draw[magenta,->_____] (8,-1) -- (6,-1);
    \draw[magenta,->_____] (6,-1) -- (6.5,0) -- (12,0);
\end{tikzpicture}

Step 3: give an orientation to the components of $\Theta(p_1,p_2)$ according to the orientation on the graph of step 2. Reconstruct $(p_+,p_-)$.

Output: $(p_+,p_-)\in P_{(1,0,1)}\times P_{(0,1,0)}$.

\vspace{7pt}


\begin{proof}[Proof of Theorem \ref{theorem PN octahedron recurrence}]
By theorem \ref{theorem PN tetrahedra equalities} we have an inequality \[m_{(i,j+1,k)}+m_{(i+1,j,k+1)}\leqslant\max\{ m_{(i,j,k+1)}+m_{(i+1,j+1,k)}, m_{(i+1,j,k)}+m_{(i,j+1,k+1)}\}.\] Thus it is enough to prove \[m_{(i,j+1,k)}+m_{(i+1,j,k+1)}\geqslant\max\{ m_{(i,j,k+1)}+m_{(i+1,j+1,k)}, m_{(i+1,j,k)}+m_{(i,j+1,k+1)}\}.\]

Take a couple of maximal multipaths $(p_1,p_2)\in P_{(i,j+1,k+1)}\times P_{(i+1,j,k)}$ which start at the top and end at the bottom. We have an equality $(\sum_j Q_{ji}-\sum_j Q_{ij})_{i=0,\ldots, 3} = (-1,-1,1,1)$. There are two cases:

(1) There exists an essential component traveling from 1 to 2.

(2) There exist essential components traveling from 0 to 2, from 1 to 3.

In the case (1) reversing the orientation of this component gives $(p_+, p_-)\in P_{(i+1,j,k+1)}\times P_{(i,j+1,k)}$.
Let us prove that the case (2) is impossible. Indeed, the $0 \rightarrow 2$ component necessarily ends at the vertex $j+1$, since all other vertices on the vertical $2$ are 2-valent, thus their two adjacent edges belong to the same component. Also, the $0 \rightarrow 2$ component cannot cross the vertical $1$ at vertices $[1, i+1]$: the crossing point must be a vertex with exactly 1 incoming and 1 outgoing edge.

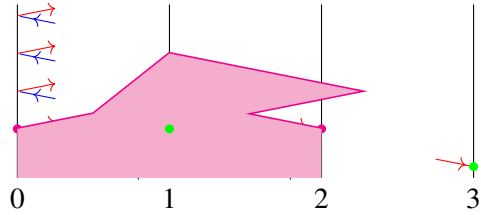
\begin{center}
\begin{tikzpicture}[scale=.5]
    \draw (0,.3) -- (0,-4.3) node[below]{0};
    \draw (4,.3) -- (4,-4.3) node[below]{1};
    \draw (8,.3) -- (8,-4.3) node[below]{2};
    \draw (12,.3) -- (12,-4.3) node[below]{3};
    \draw[red,->] (0,0) -- (1,.2);
    \draw[blue,______<-](0,0) -- (1,-.2);
    \draw[red,->] (0,-1) -- (1,-.8);
    \draw[blue,______<-](0,-1) -- (1,-1.2);
    \draw[red,->] (0,-2) -- (1,-1.8);
    \draw[blue,______<-](0,-2) -- (1,-2.2);
    \draw[red,->](0,-3) -- (1,-2.8);
    
    \draw[blue,<-](3,-4.2) -- (4,-4);
    \draw[red,->______](3,-3.8) -- (4,-4);
    \draw[blue,<-](3,-3.2) -- (4,-3);
    \draw[red, dashed] (3,-2.8) -- (5,-3.2);

    \draw[blue,<-](7,-4.2) -- (8,-4);
    \draw[red,->______](7,-3.8) -- (8,-4);
    \draw[red,->______](7,-2.8) -- (8,-3);
    
    \draw[red,->___](11,-3.8) -- (12,-4);

    \draw[magenta,nearly transparent, very thick](0,-3) -- (2,-2.6) -- (4,-1) -- (9,-2) -- (6,-2.6) -- (8,-3);
    \filldraw[magenta] (0,-3) circle(3pt);
    \filldraw[magenta] (8,-3) circle(3pt);
    \fill[magenta!40,nearly transparent](0,-4.3) -- (0,-3) -- (2,-2.6) -- (4,-1) -- (9,-2) -- (6,-2.6) -- (8,-3) -- (8,-4.3) -- cycle;

    \filldraw[green](4,-3) circle(3pt);
    \filldraw[green](12,-4) circle(3pt);
    
    \end{tikzpicture} 
\end{center}

Similarly, the $1 \rightarrow 3$ component necessarily starts at the vertex $i+1$ and cannot cross the vertical $2$ at $[1, j+1]$. This means that these components intersect. Indeed, the $0\rightarrow 2$ component together with the parts of verticals $0$ and $2$ which lie below its endpoints divide the strip into two parts (see the picture). The $1\rightarrow 3$ component has endpoints in different parts, thus it must cross the boundary: there exists a point $x$ and two edges $e_1, e_2$ adjacent to $x$ which lie in different parts. Note $x$ cannot belong to a vertical, thus $x$ lies in the interior of the $0\rightarrow 2$ component. On the other hand, a point of intersection of two components which is the interior for both of them looks as follows:

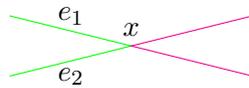
\begin{center}
\begin{tikzpicture}[scale=.4]
    \draw[green] (0,0) node[black,right,xshift=.5cm]{$e_1$} -- (4,-1) node[black,above]{$x$} -- (0,-2) node[black,right,xshift=.5cm]{$e_2$};
    \draw[magenta] (8,0) -- (4,-1) -- (8,-2);
\end{tikzpicture}
\end{center}
{\em i.e.}, $e_1$ and $e_2$ lie in the same part. This proves the impossibility of (2).

The case $(p_1,p_2)\in P_{(i+1,j+1,k)}\times P_{(i,j,k+1)}$ differs from the previous one by exchanging $1$ and $3$ and is proved in the same way.  
\end{proof}

\begin{proof}[Proof of Lemma \ref{lemma A(w) in GZ implies multi-GZ condition on planar networks}]
    It is enough to take $i=1, j=k$.
    By Lemma 9 in \cite{APS-planar}, the condition $\mathcal{A}(w)\in \Delta_{GZ}(\delta)$ is equivalent to the inequalities \[r_{j,i}^\nearrow(w)\geqslant 0, \quad 0\leqslant i<j\leqslant n; \qquad \qquad \qquad r_{j,i}^\searrow(w)\leqslant 0, \quad 0<i\leqslant j\leqslant n,\] where $r^\nearrow,r^\searrow$ are the following functions. Take a path bounding the shaded region in the picture below with orientation as indicated and take the alternating sum of the weights on its edges: the sign is "+" if an edge is oriented from left to right, and "\textminus" otherwise.
    
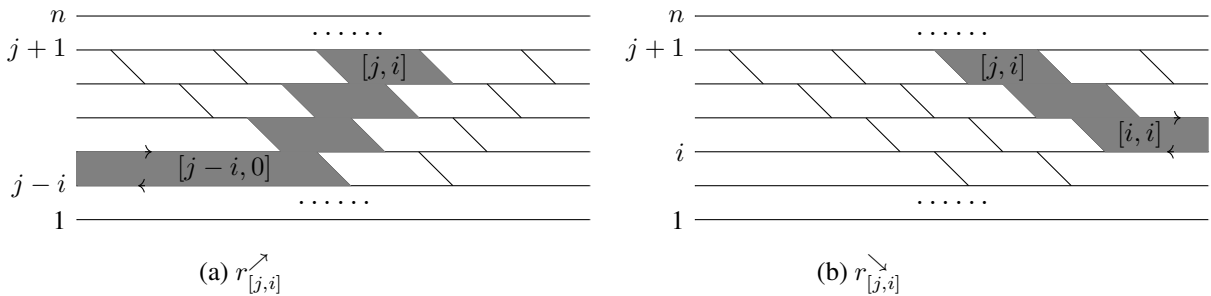
\begin{figure}[H]
        \centering
        \hspace{-4em}
        \begin{subfigure}{.4\linewidth}
        \centering
        \begin{tikzpicture}[scale=.9]
        \draw (-1,3) -- (6.5,3);
        \draw (3,2.75) node{$\dots\dots$};
        \draw (-1,0) -- (6.5,0);
        \draw (-1,0.5) -- (6.5,0.5);
        \draw (-1,1) -- (6.5,1);
        \draw (-1,1.5) -- (6.5,1.5);
        \draw (-1,2) -- (6.5,2);
        \draw (-1,2.5) -- (6.5,2.5);
        \draw (0.5,2) -- (1,1.5);
        \draw (2,2) -- (2.5,1.5);
        \draw (3.5,2) -- (4,1.5);
        \draw (5,2) -- (5.5,1.5);
        \draw (1.5,1.5) -- (2,1);
        \draw (3,1.5) -- (3.5,1);
        \draw (4.5,1.5) -- (5,1);
        \draw (2.5,1) -- (3,0.5);
        \draw (4,1) -- (4.5,0.5);
        \draw (2.8,0.25) node{$\dots\dots$};
        \draw (-0.5,2.5) -- (0,2);
        \draw (1,2.5) -- (1.5,2);
        \draw (2.5,2.5) -- (3,2);
        \draw (4,2.5) -- (4.5,2);
        \draw (5.5,2.5) -- (6,2);
        \filldraw[gray,opacity=0.15] (-1,0.5) -- (3,0.5) -- (2.5,1) -- (-1,1) -- cycle;
        \filldraw[gray, opacity=0.15] (2,1) -- (3.5,1) -- (3,1.5) -- (1.5,1.5) -- cycle;
        \filldraw[gray,opacity=0.15] (2.5,1.5) -- (4,1.5) -- (3.5,2) -- (2,2) -- cycle;
        \filldraw[gray,opacity=0.15] (3,2) -- (4.5,2) -- (4,2.5) -- (2.5,2.5) -- cycle;
        \draw[->] (0,.5)--(-.1,.5);
        \draw[->] (0,1)--(.1,1);
        \fill[black] (-1,0) circle (0pt) node[font=\small,left]{1};
        \fill[black] (-1,0.5) circle (0pt) node[font=\small,left]{$j-i$};
        \fill[black] (-1,1) circle (0pt) node[font=\small,left]{};
        \fill[black] (-1,3) circle (0pt) node[font=\small,left]{$n$};
        \fill[black] (-1,2.5) circle (0pt) node[font=\small,left]{$j+1$};
        \fill[black] (6.5,0) circle (0pt);
        \fill[black] (6.5,0.5) circle (0pt);
        \fill[black] (6.5,1) circle (0pt);
        \fill[black] (6.5,3) circle (0pt);
        \fill[black] (6.5,2.5) circle (0pt);
        \path (3.5,2.25) node [font=\small] {$[j,i]$};
        \path (2,0.75) node [left, font=\small] {$[j-i,0]$};
        \end{tikzpicture}
        \caption{$r^\nearrow_{[j,i]}$}\label{net:regions-up}
        \end{subfigure}
\hspace{4em}
        \begin{subfigure}{.4\linewidth}
        \centering
        \begin{tikzpicture}[scale=.9]
        \draw (-1,3) -- (6.5,3);
        \draw (2.8,2.75) node{$\dots\dots$};
        \draw (-1,0) -- (6.5,0);
        \draw (-1,0.5) -- (6.5,0.5);
        \draw (-1,1) -- (6.5,1);
        \draw (-1,1.5) -- (6.5,1.5);
        \draw (-1,2) -- (6.5,2);
        \draw (-1,2.5) -- (6.5,2.5);
        \draw (0.5,2) -- (1,1.5);
        \draw (2,2) -- (2.5,1.5);
        \draw (3.5,2) -- (4,1.5);
        \draw (5,2) -- (5.5,1.5);
        \draw (1.5,1.5) -- (2,1);
        \draw (3,1.5) -- (3.5,1);
        \draw (4.5,1.5) -- (5,1);
        \draw (2.5,1) -- (3,0.5);
        \draw (4,1) -- (4.5,0.5);
        \draw (2.8,0.25) node{$\dots\dots$};
        \draw (-0.5,2.5) -- (0,2);
        \draw (1,2.5) -- (1.5,2);
        \draw (2.5,2.5) -- (3,2);
        \draw (4,2.5) -- (4.5,2);
        \draw (5.5,2.5) -- (6,2);
        \filldraw[gray,opacity=0.15] (5,1) -- (6.5,1) -- (6.5,1.5) -- (4.5,1.5) -- cycle;
        \filldraw[gray,opacity=0.15] (4,1.5) -- (5.5,1.5) -- (5,2) -- (3.5,2) -- cycle;
        \filldraw[gray,opacity=0.15] (3,2) -- (4.5,2) -- (4,2.5) -- (2.5,2.5) -- cycle;
        \draw[->] (6,1.5)--(6.1,1.5);
        \draw[->] (6,1)--(5.9,1);
        \fill[black] (-1,0) circle (0pt) node[font=\small, left]{1};
        \fill[black] (-1,0.5) circle (0pt) node[font=\small, left]{};
        \fill[black] (-1,1) circle (0pt) node[font=\small, left]{$i$};
        \fill[black] (-1,3) circle (0pt) node[font=\small, left]{$n$};
        \fill[black] (-1,2.5) circle (0pt) node[font=\small, left]{$j+1$};
        \fill[black] (6.5,0) circle (0pt);
        \fill[black] (6.5,0.5) circle (0pt);
        \fill[black] (6.5,1) circle (0pt);
        \fill[black] (6.5,3) circle (0pt);
        \fill[black] (6.5,2.5) circle (0pt);
        \path (3.5,2.25) node [font=\small] {$[j,i]$};
        \path (5.5,1.25) node [font=\small] {$[i,i]$};
        \end{tikzpicture}
        \caption{$r^\searrow_{[j,i]}$}\label{net:regions-down}
        \end{subfigure}
    \end{figure}
    
    Take a path $p\in P_{\underline{\alpha}}(\Pi_i\circ\cdots\circ\Pi_j)$. Assume $p$ has a sink $a$ on $l-1$-st vertical and crosses this vertical transversely at $b$, and $a>b$. Then using $r_{s,0}(w_l)\geqslant 0$ for $s=b\ldots a-1$ we can find a new path $\Tilde{p}$ with $w(\Tilde{p})\geqslant w(p)$ which has $b$ as a sink and $a$ as a transversal intersection:
    
    
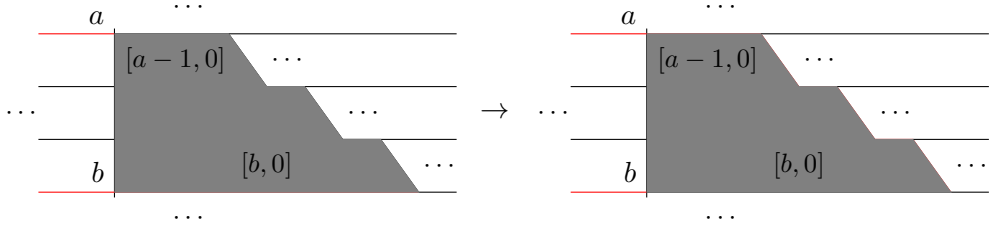
\begin{tikzpicture}[yscale=.7]
        \node at (-.2,-1.5)[font=\small]{$\cdots$};
        \node at (2,.5)[font=\small]{$\cdots$};
        \node at (2,-3.5)[font=\small]{$\cdots$};
        \draw(1,.1)--(1,-3.1);
        \draw[red] (0,0) -- (1,0) node[black, above left]{$a$};
        \draw (0,-1) -- (1,-1);
        \draw (0,-2) -- (1,-2);
        \draw[red] (0,-3) -- (1,-3) node[black, above left]{$b$};
        \begin{scope}[xshift=1.5cm]
            \draw (-.5,0) -- (4,0);
            \draw(-.5,-1) -- (4,-1);
            \draw(-.5,-2)  -- (4,-2);
            \draw[red] (-.5,-3) -- (3.5,-3);
            \draw(3.5,-3) -- (4,-3);
            \draw (1,0) -- (1.5,-1);
            \node at (1.8,-.5)[font=\small] {$\cdots$};
            \draw (2,-1) -- (2.5,-2);
            \node at (2.8,-1.5)[font=\small] {$\cdots$};
            \draw(3,-2) -- (3.5,-3);
            \node at (3.8,-2.5)[font=\small] {$\cdots$};
            \filldraw[gray,opacity=.15] (-.5,0) -- (1,0) -- (1.5,-1) -- (2,-1) -- (2.5,-2) -- (3,-2) -- (3.5,-3) -- (-.5,-3);
            \node at (.3,-.5) [font=\small] {$[a-1,0]$};
            \node at (1.5,-2.5) [font=\small] {$[b,0]$};
        \end{scope}
        
        \begin{scope}[xshift=7cm]
            \node at (-1,-1.5){$\rightarrow$};
            \node at (-.2,-1.5)[font=\small]{$\cdots$};
            \node at (2,.5)[font=\small]{$\cdots$};
            \node at (2,-3.5)[font=\small]{$\cdots$};
            \draw(1,.1)--(1,-3.1);
            \draw[red] (0,0) -- (1,0) node[black, above left]{$a$};
            \draw (0,-1) -- (1,-1);
            \draw (0,-2) -- (1,-2);
            \draw[red] (0,-3) -- (1,-3) node[black, above left]{$b$};
            \begin{scope}[xshift=1.5cm]
                \draw[red] (-.5,0) -- (1,0) -- (1.5,-1) -- (2,-1) -- (2.5,-2) -- (3,-2) -- (3.5,-3);
                \draw (1,0) -- (4,0);
                \draw(-.5,-1) -- (1.5,-1);
                \draw(2,-1) -- (4,-1);
                \draw(-.5,-2)  -- (2.5,-2);
                \draw(3,-2) -- (4,-2);
                \draw(-.5,-3) -- (4,-3);
                \node at (1.8,-.5)[font=\small] {$\cdots$};
                \node at (2.8,-1.5)[font=\small] {$\cdots$};
                \node at (3.8,-2.5)[font=\small] {$\cdots$};
                \filldraw[gray,opacity=.15] (-.5,0) -- (1,0) -- (1.5,-1) -- (2,-1) -- (2.5,-2) -- (3,-2) -- (3.5,-3) -- (-.5,-3);
                \node at (.3,-.5) [font=\small] {$[a-1,0]$};
                \node at (1.5,-2.5) [font=\small] {$[b,0]$};
            \end{scope}    
        \end{scope}
        
    \end{tikzpicture}
    
    An induction argument then shows that there is $p'$ with $w(p')\geqslant w(p)$ such that on ech vertical all the sinks lie below all the transversal intersections.



    
    Then an argument as in Lemmas 10 and 11 in \cite{APS-planar} shows that there is a multipath $p''$ with $w(p'')\geqslant w(p)$ with sources $[n-\sum\alpha_i+1,n]$ and sinks $[1,\alpha_1],\ldots,[1,\alpha_k]$.
\end{proof}

Now we turn to the proof of Theorem \ref{theorem: m is onto the octahedron recurrence cone}.

\begin{notation}
    We enumerate the faces of the tetrahedron as follows: if  its edges encode singular values $\la(1),\la(23),\la(123)$, then this face is enumerated $1;23$, if $\la(2),\la(3),\la(23)$, then $2;3$ etc.

        
        
        
        
        


    Denote by $m^{(1;23)}$ the composition of $m$ with the projection on the coordinates corresponding to a couple faces $1;23$ and $2;3$. 
 In other words, $m^{1;23}(w)=\big(m_{i,j,k}\big)_{j=0\text{ or } i+j+k=n}$. 
\end{notation}

\begin{proof}[Proof of Theorem \ref{theorem: m is onto the octahedron recurrence cone}]
In view of Theorems \ref{theorem PN rhombi inequalities}, \ref{theorem PN octahedron recurrence} and Remark \ref{remark: octahedron recurrence and crystals} it is enough to show that for any collection of numbers labeled by the integer points on the  couple of faces $1;23$ and $2;3$,
satisfying the rhombus inequalities for all rhombi contained in these faces, there exists a triple of weightings $a,b,c\in \T^{\Pi_{\rm st}}$ satisfying multi-Gelfand-Zeitlin condition, such that $m^{(1;23)}(a\circ b\circ c)$ equals this collection. The proof of this fact is a modification of the proof of Theorem 5 in \cite{APS-planar}.

Let $W\subset(\T^{\Pi_{\rm st}})^3$ be the set of triples of weightings $(a,b,c)$ such that $a,b\in W_\T(\Pi_{\rm st})$ and $c$ is non-zero on the following edges:
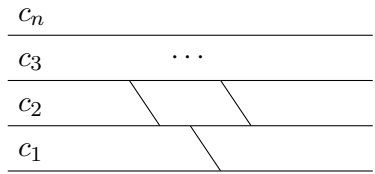
\begin{figure}[H]
    \centering
\begin{tikzpicture}[yscale=.6,xscale=.8]
    \draw (0,0) node[above right]{$c_n$} -- (6,0);
    \node at (3,-0.5) {$\cdots$};
    \draw (0,-1) node[above right]{$c_3$} -- (3.5,-1);
    \draw (3.5,-1) -- (6,-1);
    \draw (0,-2) node[above right]{$c_2$} -- (3,-2);
    \draw (3,-2) -- (4,-2);
    \draw (4,-2) -- (6,-2);
    \draw (0,-3) node[above right]{$c_1$} -- (3.5,-3);
    \draw (3.5,-3) -- (6,-3);
    \draw (2,-1) -- (2.5,-2);
    \draw (3.5,-1) -- (4,-2);
    \draw (3,-2) -- (3.5,-3);
\end{tikzpicture}
\end{figure}

This set is isomorphic to $\T^{n(n+2)}$. Denote $w=a\circ b\circ c$. Notice that for any multipath $p\in P_{i,j,k}(\Pi_{\rm st}^{\circ3})$ there is a multipath $p'\in P_{i,j,k}$ of the same weight $w(p')=w(p)$ which does not contain slanted edges of the last network.

Let $\beta_{i,0,j}$ be a unique multipath in $\Pi_{\rm st}^{\circ 3}$ with sources $[n-i-j+1,n]$ and sinks $[1,i],\emptyset,[n-j+1,n]$. Let $\beta_{n-j-k,j,k}$ be a unique multipath with sources $[1,n]$ and sinks $[1,n-j-k], [1,j], [n-k+1,n]$.
\begin{figure}[H]
\centering
\begin{subfigure}{.8\textwidth}
    \centering
    \begin{tikzpicture}[scale=.6]
        \draw[red] (0,0) node[left,black]{4} -- (6,0);
        \draw[red] (0,-1) node[left,black]{3} -- (3.5,-1);
        \draw (3.5,-1) -- (6,-1);
        \draw[red] (0,-2) node[left,black]{2} -- (3,-2);
        \draw (3,-2) -- (4,-2);
        \draw[red] (4,-2) -- (6,-2);
        \draw (0,-3) node[left]{1} -- (3.5,-3);
        \draw[red] (3.5,-3) -- (6,-3);
        \draw (1,0) -- (1.5,-1);
        \draw (2.5,0) -- (3,-1);
        \draw (4,0) -- (4.5,-1);
        \draw (2,-1) -- (2.5,-2);
        \draw[red] (3.5,-1) -- (4,-2);
        \draw[red] (3,-2) -- (3.5,-3);
        \draw (6,.1) -- (6,-3.1);
        \begin{scope}[xshift=6cm]
            \draw[red] (0,0) -- (6,0);
            \draw (0,-1) -- (3.5,-1);
            \draw (3.5,-1) -- (6,-1);
            \draw (0,-2) -- (3,-2);
            \draw (3,-2) -- (4,-2);
            \draw (4,-2) -- (6,-2);
            \draw (0,-3) -- (3.5,-3);
            \draw (3.5,-3) -- (6,-3);
            \draw (1,0) -- (1.5,-1);
            \draw (2.5,0) -- (3,-1);
            \draw (4,0) -- (4.5,-1);
            \draw (2,-1) -- (2.5,-2);
            \draw (3.5,-1) -- (4,-2);
            \draw (3,-2) -- (3.5,-3);
            \draw (6,.1) -- (6,-3.1);
        \end{scope}
        \begin{scope}[xshift=12cm]
            \draw[red] (0,0) -- (6,0);
            \draw (0,-1) -- (3.5,-1);
            \draw (3.5,-1) -- (6,-1);
            \draw (0,-2) -- (3,-2);
            \draw (3,-2) -- (4,-2);
            \draw (4,-2) -- (6,-2);
            \draw (0,-3) -- (3.5,-3);
            \draw (3.5,-3) -- (6,-3);
            \draw (1,0) -- (1.5,-1);
            \draw (2.5,0) -- (3,-1);
            \draw (4,0) -- (4.5,-1);
            \draw (2,-1) -- (2.5,-2);
            \draw (3.5,-1) -- (4,-2);
            \draw (3,-2) -- (3.5,-3);
        \end{scope}
    \end{tikzpicture}
    \caption{$\beta_{2,0,1}$}
\end{subfigure}
\begin{subfigure}{.8\linewidth}
    \centering
    \begin{tikzpicture}[scale=.6]
            \draw[red] (0,0) node[left,black]{4} -- (6,0);
            \draw[red] (0,-1) node[left,black]{3} -- (6,-1);
            \draw[red] (0,-2) node[left,black]{2} -- (6,-2);
            \draw[red] (0,-3) node[left,black]{1} -- (6,-3);
            \draw (1,0) -- (1.5,-1);
            \draw (2.5,0) -- (3,-1);
            \draw (4,0) -- (4.5,-1);
            \draw (2,-1) -- (2.5,-2);
            \draw (3.5,-1) -- (4,-2);
            \draw (3,-2) -- (3.5,-3);
            \draw (6,.1) -- (6,-3.1);

        \begin{scope}[xshift=6cm]
            \draw[red] (0,0) -- (6,0);
            \draw[red] (0,-1) -- (3.5,-1);
            \draw (3.5,-1) -- (6,-1);
            \draw[red] (0,-2) -- (3,-2);
            \draw (3,-2) -- (4,-2);
            \draw[red] (4,-2) -- (6,-2);
            \draw (0,-3) -- (3.5,-3);
            \draw[red] (3.5,-3) -- (6,-3);
            \draw (1,0) -- (1.5,-1);
            \draw (2.5,0) -- (3,-1);
            \draw (4,0) -- (4.5,-1);
            \draw (2,-1) -- (2.5,-2);
            \draw[red] (3.5,-1) -- (4,-2);
            \draw[red] (3,-2) -- (3.5,-3);
            \draw (6,.1) -- (6,-3.1);
        \end{scope}
        \begin{scope}[xshift=12cm]
            \draw[red] (0,0) -- (6,0);
            \draw (0,-1) -- (3.5,-1);
            \draw (3.5,-1) -- (6,-1);
            \draw (0,-2) -- (3,-2);
            \draw (3,-2) -- (4,-2);
            \draw (4,-2) -- (6,-2);
            \draw (0,-3) -- (3.5,-3);
            \draw (3.5,-3) -- (6,-3);
            \draw (1,0) -- (1.5,-1);
            \draw (2.5,0) -- (3,-1);
            \draw (4,0) -- (4.5,-1);
            \draw (2,-1) -- (2.5,-2);
            \draw (3.5,-1) -- (4,-2);
            \draw (3,-2) -- (3.5,-3);
        \end{scope}
    \end{tikzpicture}
    \caption{$\beta_{1,2,1}$}
\end{subfigure}
\end{figure}
Let $\beta$ be the map $W\to\T^{n(n+2)}$ with components $w\mapsto w(\beta_{i,j,k})$ where $j=0$ or $i+j+k=n$.

\textit{Claim 1.} The map $\beta$ is a linear isomorphism.

\textit{Proof.} We have: $c_j=w(\beta_{0,j-1,n-j+1})-w(\beta_{0,j,n-j})$, \; $b_{j,j}+c_j=w(\beta_{j-1,0,n-j+1})-w(\beta_{j,0,n-j})$,  $a_{j,j}+b_{j,j}+c_j=w(\beta_{0,0,n-j+1})-w(\beta_{0,0,n-j})$, therefore we can recover $a_{j,j}, b_{j,j}, c_j$ from $\beta(w)$. Moreover, $\beta_{i,0,n-l}=a_{l,i}+\ldots$ where $\ldots$ contains $a_{l',i'}$ with $i'< i$ or $l'<l$, $a_{j,j}, b_{j,j}, c_j$. Therefore we can recover $a_{l,i}$ from $\beta(w)$. Similarly, $\beta_{l-i,i,n-l}=b_{l,i}+\ldots$ where $\ldots$ contains $b_{l',i'}$ with $i'< i$ or $l'<l$, $a_{j,j}, b_{j,j}, c_j$, and we recover $b_{l,i}$.

\textit{Claim 2.} If $\beta(w)$ satisfies the rhombus inequalities for all rhombi contained in the faces $1;23$ and $2;3$, then $a,b,c$ are Gelfand-Zeitlin.

\textit{Proof.} Notice that if $(m_i^{(l)})\in\Delta_{GZ}$ and $C_l$ are some constants, then $(m_i^{(l)}+C_l)\in\Delta_{GZ}$ as well.

$a$ is Gelfand-Zeitlin: $a(\alpha_i^{(l)})=w(\beta_{i,0,n-l})-\sum_{j=l+1}^n(a_{j,j}+b_{j,j}+c_j)$. Since $\big(w(\beta_{i,0,n-l})\big)_{i\leq l}\in\Delta_{GZ}$, so is $\mathcal{A}(a)$. 

Similarly, $b(\alpha_i^{(l)})=w(\beta_{l-i,i,n-l})-\sum_1^n a_{j,j}-\sum_{l+1}^n (b_{j,j}+c_j)$, and $\big(w(\beta_{l-i,i,n-l})\big)_{i\leq l}\in \Delta_{GZ}$ implies that $\mathcal{A}(b)\in\Delta_{GZ}$.

$c$ is Gelfand-Zeitlin since $c(\alpha_i^{(l)}) = c_{l-i+1}+\cdots+c_l = w(\beta_{0,l-i,n-l+i})-\sum_1^n(a_{j,j}+b_{j,j})-\sum_{l+1}^n c_j$.

\textit{Claim 3.} If $\beta(w)$ satisfies the rhombus inequalities for all rhombi contained in the  faces $1;23$ and $2;3$, then $\beta(w)=m(w)$.

\textit{Proof.} From the proof of Theorem 5 in \cite{APS-planar} it follows that: (i) if $\beta(w)$ satisfies rhombus inequalities, then $w(\beta_{n-i-j,i,j})=m_{n-i-j,i,j}(w)$; \; (ii) the rhombus inequalities for $\beta(w)$ are equivalent to
\begin{gather*}
    r_{j,i}^\nearrow(a), r_{j,i}^\nearrow(b)\geqslant 0, \quad 0\leqslant i<j\leqslant n; \qquad 
    r_{j,i}^\searrow(a), r_{j,i}^\searrow(b)\leqslant 0, \quad 0<i\leqslant j\leqslant n; \qquad \\
    r_{j,i}^\rightarrow(b;c), r_{j,i}^\rightarrow(a;b\circ c)\geqslant0, \; 0<i\leqslant j\leqslant n.
\end{gather*}
Here $r^\rightarrow(w_1;w_2)$ is a function of two weightings: $w_1\in \T^{\Pi_{\rm st}(n)}$ and $w_2\in\T^\Pi$, where $\Pi$ is a planar network which contains a horizontal line between $i$-th source and $i$-th sink.
It is defined similarly to $r^\nearrow, r^\searrow$ (see proof of lemma \ref{lemma A(w) in GZ implies multi-GZ condition on planar networks}) via the following picture:

\begin{figure}[H]
\centering
\begin{tikzpicture}
\node at (3.5, .7){$\cdots$};
\draw (-1,1) -- (8.5,1);
\draw (-1,1.5)  node[left]{$j$} -- (8.5,1.5) node[right]{$j$};
\draw[->] (7.5,1.5) -- (7.4,1.5);
\draw[->_______________________] (-1,2) -- (8.5,2);
\draw (-1,2.5) -- (8.5,2.5);
\node at (3.5, 2.7){$\cdots$};
\draw (0.5,2) -- (1,1.5);
\draw (2,2) -- (2.5,1.5);
\draw (3.5,2) -- (4,1.5);
\draw (5,2) -- (5.5,1.5);
\draw (1.5,1.5) -- (2,1);
\draw (3,1.5) -- (3.5,1);
\draw (4.5,1.5) -- (5,1);
\draw (-0.5,2.5) -- (0,2);
\draw (1,2.5) -- (1.5,2);
\draw (2.5,2.5) -- (3,2);
\draw (4,2.5) -- (4.5,2);
\draw (5.5,2.5) -- (6,2);
\draw[opacity=0.3] (6.5,1) -- (6.5,2.5);
\filldraw[gray,opacity=0.15] (2.5,1.5) -- (6.5,1.5) -- (6.5,2) -- (2,2);
\filldraw[gray,opacity=0.15] (6.5,1.5) -- (6.5,2) -- (8.5,2) -- (8.5,1.5);
\node at (3,1.75) [font=\small] {$[j,i]$};
\node at (6,1.75) [font=\small] {$[j,j]$};
\node at (7.5,1.75) [font=\small] {$\cdots$};
\end{tikzpicture}
 \caption{$r^\rightarrow_{[j,i]}$}
\end{figure}

Let us prove that $w(\beta_{i,0,j})=m_{i,0,j}(w)$. By Claim 2 and Lemma \ref{lemma A(w) in GZ implies multi-GZ condition on planar networks}, $w$ satisfies multi-Gelfand-Zeitlin condition. Therefore there is a maximal multipath $p\in P_{i,0,j}$ with sources $[n-i-j+1,n]$ and sinks $[1,i], \emptyset, [1,j]$. 
In particular, all the single sub-paths of $p$ ending on the 3rd vertical lie above those ending at the 1st vertical. Change $p$ to a path of the same weight so that it does not contain slanted edges of the 3rd network. Take the highest non-straight sub-path $p_l$ of $p$ ending on 3rd vertical. Take a cell which $p_l$ bounds from left and below (this cell can lie in the 1st or 2nd planar network):
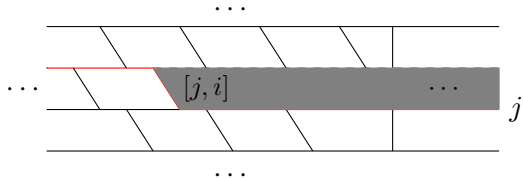
\begin{figure}[H]
    \centering
\begin{tikzpicture}[xscale=.7, yscale=1.1]
\node at (3.5, .7){$\cdots$};
\node at (-.4,1.74){$\cdots$};
\draw (0,1) -- (8.5,1);
\draw (0,1.5) -- (2.5,1.5);
\draw[red] (2.5,1.5) -- (8.5,1.5) node[right, black]{$j$};
\draw (0,2) -- (0,2);
\draw[red] (0,2) -- (2,2);
\draw[red,dashed] (2,2) -- (8.5,2);
\draw (0,2.5) -- (8.5,2.5);
\node at (3.5, 2.7){$\cdots$};
\draw (0.5,2) -- (1,1.5);
\draw[red] (2,2) -- (2.5,1.5);
\draw (3.5,2) -- (4,1.5);
\draw (5,2) -- (5.5,1.5);
\draw (1.5,1.5) -- (2,1);
\draw (3,1.5) -- (3.5,1);
\draw (4.5,1.5) -- (5,1);
\draw (1,2.5) -- (1.5,2);
\draw (2.5,2.5) -- (3,2);
\draw (4,2.5) -- (4.5,2);
\draw (5.5,2.5) -- (6,2);
\draw[opacity=0.3] (6.5,1) -- (6.5,2.5);
\filldraw[gray,opacity=0.15] (2.5,1.5) -- (6.5,1.5) -- (6.5,2) -- (2,2);
\filldraw[gray,opacity=0.15] (6.5,1.5) -- (6.5,2) -- (8.5,2) -- (8.5,1.5);
\node at (3,1.75) [font=\small] {$[j,i]$};
\node at (7.5,1.75) [font=\small] {$\cdots$};
\end{tikzpicture}
\end{figure}

Since all the sub-paths of $p$ lying above $p_l$ are straight, one can use an inequality $r^\rightarrow\geqslant 0$ corresponding to this cell to move the sink of $p_l$ higher, {\em i.e.}, there is a multipath $\Tilde{p}$ (dashed in the picture) of weight $w(\Tilde{p})\geqslant w(p)$ with the sum of labels of sinks bigger than that of $p$. An induction argument then shows that there is a maximal multipath $p'\in P_{i,0,j}$ which has sources and sinks as $\beta_{i,0,j}$, hence $p'=\beta_{i,0,j}$.

\vspace{5pt}

To finish the proof take a collection of numbers $x=\big(x_{i,j,k}\big)_{j=0\text{ or } i+j+k=n}$. Let $(a,b,c)=\beta^{-1}(x)$. By definition, $\beta(a\circ b\circ c)$ satisfies rhombus inequalities, hence the statements of claims 2 and 3 are true for $(a,b,c)$: $a,b,c$ are Gelfand-Zeitlin and $m(a\circ b\circ c) = x$. 
\end{proof}

\section{Positive functions and inequalities}\label{sec:geometry}

This section should be thought of as a geometric interpretation of the multiple Horn problem. Also, via the theory of geometric crystals, one can think of the results in this section as a geometrization of some results in \cite{HK}.

Denote by $\GL_n$ the space of $n\times n$ invertible matrices over $\mathbb{C}$. We shall introduce a bunch of functions $M_{\underline{\alpha}}$, which will be called multi-corner minors, on $\GL_n\times \cdots \times \GL_n$ and investigate their properties. Via the celebrated Loewner-Whitney Theorem, one should think of $m_{\underline{\alpha}}$'s as the ``tropicalization'' of $M_{\underline{\alpha}}$'s.

Throughout this section, we denote by $U$ the subgroup of $\GL_n$ consisting of  unipotent upper triangular matrices and denote by $B_-$ the subgroup of $\GL_n$ consisting of  invertible lower triangular matrices. For any integer $n>0$, denote by $[1,n]$ the set of integers $\{1,2,\ldots,n\}$.  For simplicity, $[1,n]$ is the emptyset if $n\leqslant 0$. For any subset $I=\{a_1,\ldots,a_k\}\subset [1,n]$ and $b\in [1,n]$ such that $a_i+b\in [1,n]$, denote by 
\[
    I^{\op}:=\{n+1-a_1,\ldots,n+1-a_k\}, \quad I+b:=\{a_1+b,\ldots,a_k+b\}
\]
the opposite subset of $I$ and $b$-shifted subset of $I$. Denote by $|I|$ the cardinality of  $I$.
Given a matrix $g$ of size $n\times m$, for any subsets $I\subset [1,n]$ and $J\subset [1,m]$ such that $|I|=|J|$, denote by $\Delta_{I,J}(g)$ the determinant of the submatrix of $g$ corresponding to $I$'s rows and $J$'s columns.


\subsection{Preliminaries: the Berenstein-Kazhdan potential and interlacing inequalities}\label{subsection: geometric preliminaries: BK potential and interlacing ineq}
In this section, we recall the notion of the Berenstein-Kazhdan potential, which is a rational function on $\GL_n$, and its relation with interlacing inequalities. It was first introduced in \cite{bk2} as a $U\times U$-linear function on $\GL_n$ in the study of  geometric crystals, see Eq \eqref{eq:propofBK}. It also appears independently in \cite{R} in the study of quantum cohomology of flag varieties.  We recall

\begin{definition}\cite[Corollary 1.25]{bk2}
    As a rational function, the potential $\Phi_{\rm BK}$ on $\GL_n$ is given by 
    \[
    \Phi_{\rm BK}(g)=\sum_{i=1}^{n-1} \frac{\Delta_{([1,i+1]\setminus \{i\})^{\op},[1,i]}(g)+\Delta_{[1,i]^{\op},[1,i+1]\setminus \{i\}}(g)}{\Delta_{[1,i]^{\op},[1,i]}(g)}.
    \]
\end{definition}

Note that for $(u_1, u_2)\in U^2$, following \cite[Corollary 1.21]{bk2}, we have
\begin{equation}\label{eq:propofBK}
    \Phi_{\rm BK}(u_1gu_2)=\chi(u_1)+\Phi_{\rm BK}(g)+\chi(u_2),
\end{equation}
where $\chi(u)$ is the additive character of $U$ defined by $\chi(u)=\sum_{i=1}^n u_{i,i+1}$ for $[u_{ij}]\in U$. 

Now let us restrict the rational function $\Phi_{\rm BK}$ to $B_-$ and consider the pair $(B_-,\Phi_{\rm BK})$, which is a {\em positive variety with potential} in the following sense. Recall that given a complex variety $P$ of dimension $d$ and a rational function $\Phi$ on $P$, there is a notion of a {\em positive structure}, which is a collection $\Theta$ of open embeddings $\theta_i\colon (\mathbb{C}^*)^d\hookrightarrow P$ such that 1) $\Phi\circ \theta$ has a positive expression in terms of natural coordinates on $(\mathbb{C}^*)^d$; 2) the composition $\theta_i\circ \theta_j^{-1}$ is positive for any $\theta_i, \theta_j\in \Theta$. Open embeddings $\theta_i$ are called positive charts on $(P, \Phi)$, see \cite[Section 2.1]{ABHL} for details. A positive structure can be recovered from the set of coordinate functions of the involved birational isomorphisms $\theta_i^{-1}:P\to (\mathbb{C}^*)^d$ (such as multi-corner minors $\{M_{i,j}\}$ before Theorem \ref{th:rhombus M2} and $\{M_{i,j,k}\}$ in Theorem \ref{intro:M_3properties}, see section \ref{subsec:potential on Bk} below).

We now introduce two equivalent positive charts on $(B_-,\Phi_{\rm BK})$. The first one comes from the {\em geometric Gelfand-Zeitlin pattern}. For each $b\in B_-$ and $(i,j)\in \delta_n:=\{(i,j)\mid 0\leqslant i\leqslant n-1,1\leqslant j\leqslant n, i+j\leqslant n\}$ define  
\begin{equation}\label{eq:geoGZpattern}
    \Delta_{i,j}:=\Delta_{[1,j]^{\op},[i+1,i+j]}=:\lambda_{1}^{(n-i)}\cdots \lambda_{j}^{(n-i)},
\end{equation}
which is equivalent to
\[
    \lambda_{j}^{(n-i)}=\frac{\Delta_{i,j}}{\Delta_{i,j-1}}, \text{~for~} j>1; \quad \lambda_{1}^{(n-i)}=\Delta_{i,1}.
\]
We have
\begin{lemma}
    The map $\Lambda\colon B_-\to \mathbb{C}^{\frac{1}{2}n(n+1)}$ given by
    \[b\mapsto \{\lambda_{n-i-j}^{(n-i)}(b)\mid (i,j)\in\delta_n\}
    \]
    is a birational isomorphism, whose inverse restricts to an open embedding $\theta_{\rm GZ}\colon (\mathbb{C}^*)^{\frac{1}{2}n(n+1)}\to B_-$.
\end{lemma}
\begin{proof}
    For $1\leqslant i\leqslant n-1$, denote by $\varphi_i\colon \SL_2\to \GL_n$ the natural embedding and 
     \[
     x_{-i}(t)=\varphi_i\left(
     \begin{bmatrix} 
     t^{-1} & 0\\
     1& t
     \end{bmatrix}
     \right).
     \]
     We show the lemma by claiming that the inverse of $\Lambda$ is given by
     \begin{equation}\label{eq:invofthetaGZ}
         \theta_{\rm GZ}(\Lambda)=\diag(\lambda_n^{(n)},\ldots,\lambda_1^{(n)})\cdot  \vec \prod x_{i-j}(t_{ij}),
     \end{equation}
     where the product is in the lexicographic order (that is $(1,2),(1,3), \ldots, (2,3) ,\ldots$) on pairs $(i,j)$ for $1\leqslant i<j\leqslant n$ and $t_{ij}=\lambda_{n+1-j}^{(n+i-j)}/\lambda_{n+1-j}^{(n)}$. The proof of the claim follows from \cite[Eq (7.1),(8.1) and (8.3)]{bz2} and by induction on $n$, where we embed $\GL_{n-1}$ to the right-bottom of $\GL_n$. The detailed proof  is purely computational, and it is omitted here.
\end{proof}

The second positive chart of $(B_-,\Phi_{\rm BK})$ comes from the standard planar  network $\Pi_{\rm{st}}$ with the weighting $w_{\rm{st}}\colon \edge(\Pi_{\rm{st}})\to \mathbb{C}$ as in Fig.\ref{fig:stplanar}.
Recall that for each planar network $\Pi$ and an $R$-weighting $w\colon \edge(\Pi)\to \mathbb{C}$, one can construct a matrix $M_{\Pi}(w)$ via

\begin{definition}\label{definition: correspondence map}
    The correspondence map is \[M_\Pi:R^\Pi\rightarrow \Mat_{n\times n}(R),\qquad w\mapsto \big[M_\Pi(w)_i^j\big]_{i,j=1\ldots n}=\Big[\sum_{p\colon i\rightarrow j}w(p)\Big]_{i,j=1,\ldots, n}.\]
\end{definition} 

The correspondence map has the following properties:
\begin{itemize}
    \item[(1)] $M_{\Pi_1\circ \Pi_2}(w_1\circ w_2) = M_{\Pi_1}(w_1)\cdot M_{\Pi_2}(w_2)$,

    \item[(2)] (Lindstr{\"o}m's Lemma) The minors of $M_\Pi(w)$ are positive rational functions of the weights. They can be expressed as
\[\Delta_{I,J}(M_\Pi(w))=\sum_{p\colon I\to J}w(p).\]
\end{itemize}

The variables $a_{i,j}$'s in the Fig.\ref{fig:stplanar} are nothing but the well-known factorization parameter for $B_-$. Thus we have a map
\[
    \theta_{\mathrm{st}}\colon (\mathbb{C}^*)^{\frac{1}{2}n(n+1)}\to B_-\ : \ \{a_{ij}\}\mapsto M_{\Pi_{\rm st}}(w_{\rm st}).
\]
\begin{lemma}\label{lemma: GZ and standard network positive structures are equivalent}
    The open embedding $\theta_{\rm GZ}$ and $\theta_{\rm{st}}$ are {\em positive equivalent}, {\em i.e.},  $\theta_{\rm GZ}\circ \theta_{\rm st}^{-1}$ and $\theta_{\rm st}\circ \theta_{\rm GZ}^{-1}$ are positive. Thus $\theta_{\rm GZ}$ and $\theta_{\rm{st}}$ define a positive structure for $(B_-,\Phi_{\rm BK})$.
\end{lemma}
\begin{proof}
    By Lindstr{\"o}m's Lemma, it is clear that $\Delta_{[1,j]^\op,[i+1,i+j]}$'s are positive polynomial of $a_{ij}$'s. By \cite{FZ}, one can write down the parameter $a_{ij}$'s using $\Delta_{[1,j]^\op,[i+1,i+j]}$'s in positive expressions. 
\end{proof}

Now let us write down the matrix $b$ and the BK potential $\Phi_{\rm BK}(b)$ using the positive chart $\theta_{\rm GZ}$. Here are two examples: 
\begin{example}
    For $n=2$, we have
    \[
    b=\begin{bmatrix}[1.7]
        \frac{\lambda_1^{(2)}\lambda_2^{(2)}}{\lambda_1^{(1)}} & 0\\
        \lambda_1^{(2)} & \lambda_1^{(1)}
    \end{bmatrix} \text{~with~} \Phi_{\rm BK}=\frac{\lambda_2^{(2)}}{\lambda_1^{(1)}}+\frac{\lambda_1^{(1)}}{\lambda_1^{(2)}}.
    \]
    For $n=3$, we have
    \[
    b=\begin{bmatrix}[1.6]
    \frac{\lambda_1^{(3)}\lambda_2^{(3)}\lambda_3^{(3)}}{\lambda_1^{(2)}\lambda_2^{(2)}}&0&0\\
    \frac{\lambda_1^{(3)}\lambda_2^{(2)}}{\lambda_1^{(1)}}+\frac{\lambda_1^{(3)}\lambda_2^{(3)}}{\lambda_1^{(2)}}    &\frac{\lambda_1^{(2)}\lambda_2^{(2)}}{\lambda_1^{(1)}} & 0\\
    \lambda_1^{(3)}    &\lambda_1^{(2)} & \lambda_1^{(1)}
    \end{bmatrix}\text{~with~} \Phi_{\rm BK}= \frac{\lambda_2^{(2)}}{\lambda_1^{(1)}}+\frac{\lambda_3^{(3)}}{\lambda_2^{(2)}}+\frac{\lambda_2^{(3)}}{\lambda_1^{(2)}}+\frac{\lambda_1^{(2)}}{\lambda_1^{(3)}}+\frac{\lambda_1^{(1)}}{\lambda_1^{(2)}}+\frac{\lambda_2^{(2)}}{\lambda_2^{(3)}}.
    \]
\end{example}
\begin{theorem}\label{Thm:poten}
    In terms of the positive chart $\theta_{GZ}$ on $(B_-,\Phi_{\rm BK})$, we have
    \[
    \Phi_{\rm BK}=\sum_{(i,j)\in \delta_n'} \frac{\lambda_j^{(n-i-1)}}{\lambda_j^{(n-i)}}+\sum_{(i,j)\in \delta_n''}\frac{\lambda_{j+1}^{(n-i+1)}}{\lambda_j^{(n-i)}}, 
    \]
    where $\delta_n'=\{(i,j)\mid 0\leqslant i\leqslant n-1,1\leqslant j\leqslant n, i+j< n\}$ and  $\delta_n''=\{(i,j)\mid 1\leqslant i\leqslant n-1,1\leqslant j\leqslant n, i+j\leqslant n\}$. 
\end{theorem}
\begin{proof}
     Recall that we have Eq \eqref{eq:invofthetaGZ}. Furthermore, abbreviate ${}_i\Phi(g):=\frac{\Delta_{([1,i+1]\setminus \{i\})^{\op},[1,i]}(g)}{\Delta_{[1,i]^{\op},[1,i]}(g)}$, $\Phi_i(g):=\frac{\Delta_{[1,i]^{\op},[1,i+1]\setminus \{i\}}(g)}{\Delta_{[1,i]^{\op},[1,i]}(g)}$ for $g\in \GL_n$. 
 Clearly, $\Phi_{\rm BK}(g)=\sum\limits_{i=1}^{n-1}{}_i\Phi(g)+\Phi_i(g)$. We need

 \begin{lemma} In the notation of \eqref{eq:invofthetaGZ}, let $b:=\diag(h_n,\ldots,h_1)\cdot \vec \prod x_{i-j}(t_{ij})$. Then
 $${}_i\Phi(b)=\frac{h_{i+1}}{h_i}\sum_{k=1}^{n-i} \frac{t_{k,n-i}}{t_{k,n+1-i}},\quad \Phi_i(b)=\sum_{k=n-i}^n \frac{t_{n-i,k}}{t_{n+1-i,k}}.$$ 
  with the convention $t_{j,j}=1$.
 \end{lemma}
The proof of  this lemma is left to the readers. Finally, taking $h_i:=\lambda_i^{(n)}$, $t_{ij}=\frac{\lambda_{n+1-j}^{(n+i-j)}}{\lambda_{n+1-j}^{(n)}}$ as in \eqref{eq:invofthetaGZ}, we obtain $b=\theta_{\rm GZ}(\Lambda)$ and:
\begin{align*}
    {}_i\Phi(b)&=\frac{h_{i+1}}{h_i}\sum_{k=1}^{n-i} \frac{t_{k,n-i}}{t_{k,n+1-i}}=\frac{\lambda_{i+1}^{(n)}}{\lambda_i^{(n)}}\sum_{k=1}^{n-i} \frac{\frac{\lambda_{i+1}^{(i+k)}}{\lambda_{i+1}^{(n)}}}{\frac{\lambda_i^{(i+k-1)}}{\lambda_i^{(n)}}}=\sum_{k=1}^{n-i} \frac{\lambda_{i+1}^{(i+k)}}{\lambda_i^{(i+k-1)}};\\
    \Phi_i(b)&=\sum_{k=n-i}^n \frac{t_{n-i,k}}{t_{n+1-i,k}}=\sum_{k=n-i}^n \frac{\frac{\lambda_{n+1-k}^{(2n-i-k)}}{\lambda_{n+1-k}^{(n)}}}{\frac{\lambda_{n+1-k}^{(2n+1-i-k)}}{\lambda_{n+1-k}^{(n)}}}=\sum_{\ell=1}^i \frac{\lambda_{\ell}^{(n+\ell-i-1)}}{\lambda_{\ell}^{(n+\ell-i)}}.\tag*{\qedhere}
\end{align*}    
\end{proof}
For a positive variety with potential $(P,\Phi)$, one can define {\em tropicalization} of a positive function on $P$.  The following brief description will be sufficient for the purposes of this paper (see \cite[Section 2.2]{ABHL} for more details). Consider a positive chart $\theta\colon (\mathbb{C}^*)^d\hookrightarrow P$. For a function $f$ on $P$, suppose that $f\circ \theta \in \mathbb{C}[x_1^\pm,\ldots, x_d^\pm]$ is a Laurent polynomial, where $x_i$'s are natural coordinates of $(\mathbb{C}^*)^d$. Furthermore, suppose that 
\[
f\circ \theta=\sum c_{i_1,\ldots,i_d} x_1^{a_{i_1}}\cdots x_d^{a_{i_d}}, \text{~for~} a_i\in \mathbb{Z},
\]
with positive coefficients $c_{i_1,\ldots,i_k}>0$. Then,
we define
\[
f^t:=\max_{(i_1,\ldots,i_k)}\{a_{i_1}x^t_1+\cdots +a_{i_d}x^t_d\},
\]
where $x_i^t$'s are variables which take values in $\T$, and we are using ordinary multiplication between $a$ and $x_k^t$ in the expression $a x_k^t$ for $a\in \mathbb{Z}$ rather than the product in $\mathbb{T}$. Theorem \ref{Thm:poten} admits the following tropicalization:

\begin{corollary}
    For the positive variety with potential $(B_-,\Phi_{\rm BK})$ equipped with the positive chart $\theta_{\rm GZ}$, the inequality $\Phi_{\rm BK}^t\leqslant 0$ is equivalent to the interlacing inequalities on functions $(\lambda^{(i)}_j)^t$, {\em i.e.},
    \[
    \Phi_{\rm BK}^t\leqslant 0\ \Leftrightarrow\ \left(\lambda^{(i)}_j\right)^t\geqslant \left(\lambda_{j}^{(i-1)}\right)^t\geqslant \left(\lambda_{j+1}^{(i)}\right)^t.
    \]
\end{corollary}

\subsection{Multi-corner minors}

In this section, we introduce the notion of {\em multi-corner minors} and explore their basic properties.

Let $(g_1,\ldots,g_k)\in\GL_n^k:=\underbrace{\GL_n\times \cdots \times \GL_n\ }_{k}$, and consider  the following $n\times ((k+1)n)$ matrix
\begin{equation}\label{eq:g}
    {\bm g}={\bm g}(g_1,\ldots,g_k):=\begin{bmatrix}
       \id_n& g_1 & g_1g_2 &\cdots & g_1g_2\cdots g_k
    \end{bmatrix}.
\end{equation}

\begin{definition}\label{def:Malpha}
    For a $k$-tuple of integers $\underline{\alpha}$ satisfying $0\leqslant \alpha_i\leqslant n$ and $0\leqslant \sum\alpha_i\leqslant n$, the $\underline{\alpha}$-th multi-corner minor is defined to be
    \[
    M_{\underline{\alpha}}(g_1,\cdots,g_k)=\Delta_{[1,n], J(\underline{\alpha})}({\bm g}),
    \]
    where $J(\underline{\alpha})=([1,n]\setminus [1,\sum \alpha_i]^\op)\cup ([1,\alpha_1]+n)\cup ([1,\alpha_2]+2n)\cup \cdots \cup ([1,\alpha_k]+kn)$.
\end{definition}

\begin{example}\label{eq:k=1andhw}
   For $k=1$,  $M_i(g)$ is the $i\times i$ corner minor of $g$ for $i\in [1,n]$. This is actually where the name ``multi-corner minor'' comes from. Consider $g\in \GL_n$ such that $M_i(g)\neq 0$, and recall the {\em highest weight} map 
    \[
        \hw: g\mapsto \diag\left( \frac{M_n(g)}{M_{n-1}(g)}, \ldots, \frac{M_2(g)}{M_1(g)},M_1(g)\right)
    \]
    introduced in \cite[Eq (2.5)]{bk2}. In more detail, let $w_0$ be the longest element in the Weyl group of $\GL_n$ and denote by $\overline{w_0}$ the lift of $w_0$ given by
    \[
    \overline{w_0}=
    \begin{bmatrix}
        &&&(-1)^{n-1}\\
        &&(-1)^{n-2}&\\
        &\iddots&&\\
        (-1)^0&&&&
    \end{bmatrix}.
    \]
    Take an element $g$ in the Bruhat cell $Bw_0B$. Note that $g\in Bw_0B$ if and only if $M_i(g)\neq 0$ for all $i$. It is not hard to see that for $g\in Bw_0B$ there is a unique decomposition  $g=uh\overline{w_0}v$ with $u,v\in U$ and $h\in H$. Then, we observe that $\hw(g)=\hw(uh\overline{w_0}v)=h$. 
\end{example}

One can write $M_{\underline{\alpha}}$ using minors of matrices $g_i$: 

\begin{proposition}\label{pro:expofM}
    We have
    \begin{equation}\label{eq:expansion}
        M_{\underline{\alpha}}({\bm g})=\sum_{L_1, \ldots,L_{k-1}}\prod_{i=1}^k\Delta_{L_{i-1}, J_i\sqcup L_i}(g_i),
    \end{equation}  
    where $J_i=[1,\alpha_i]$ for $i\in [1,k]$, $L_0=[1,\sum \alpha_i]^\op, L_k=\emptyset$, and the summation is over $L_1,\ldots,L_{k-1}$ such that  $L_i\cap J_i=\emptyset$ and $|L_i|+|J_i|=|L_{i-1}|$.
\end{proposition}
\begin{proof}
     Let $\bar{\bm g}:=[g_1 ~ g_1g_2 ~\cdots ~g_1\cdots g_k]$, and use the shorthand notation $J$ for $J(\underline{\alpha})$. Thus,
    \[
    \Delta_{[1,n],J}({\bm g})=\Delta_{L_0,J'}(\bar{\bm g}),
    \]
    where $J'=(J-n)\cap [1,kn]$. Now by applying the Cauchy-Binet formula to 
    \[
    \bar{\bm g}=g_1\begin{bmatrix}
        I_n & g_2 &\cdots & g_2\cdots g_k
    \end{bmatrix}, 
    \]
    we obtain
    \begin{equation}\label{eq:exp1}
       \Delta_{[1,n],J}({\bm g})=\sum_K \Delta_{L_0,K}(g_1)\Delta_{K,J'}(\begin{bmatrix}
        I_n & g_2 &\cdots & g_2\cdots g_k
    \end{bmatrix}),
    \end{equation}
    where $K\subset [1,n]$ and $|K|=|L_0|$. Note that $\Delta_{K,J'}(\begin{bmatrix}
        I_n & g_2 &\cdots & g_2\cdots g_k
    \end{bmatrix})$ is zero if $J_1\not\subset K$. Thus $K$ has form $J_1\sqcup L_1$ such that $J_1\cap L_1=\emptyset$, and 
    \[
    \Delta_{[1,n],J}({\bm g})=\sum_{L_1} \Delta_{L_0,J_1\sqcup L_1}(g_1)\Delta_{L_1,(J-2n)\cap [1,(k-1)n]}(\begin{bmatrix}
        g_2 & g_2g_3&\cdots & g_2\cdots g_k
    \end{bmatrix}).
    \]
     By induction, we have
    \begin{align*}
        \Delta_{[1,n],J}({\bm g})&=\sum_{L_1, \ldots,L_{k-1}} \Delta_{L_0,J_1\sqcup L_1}(g_1)\Delta_{L_1, J_2\sqcup L_2}(g_2)\cdots \Delta_{L_{k-1}, J_k}(g_k) \\
        &=\sum_{L_1, \ldots,L_{k-1}}\prod_{i=1}^k\Delta_{L_{i-1}, J_i\sqcup L_i}(g_i),
    \end{align*}
    where $L_i\cap J_i=\emptyset$ and $|L_i|+|J_i|=|L_{i-1}|$.
\end{proof}

Note that there is a $U^{k+1}$-action on $\GL_n^k$ defined as follows: for ${\bm u}=(u_0,\cdots,u_k)\in U^{k+1}$, 
\[
    (u_0,\ldots, u_k).(g_1,\ldots,g_k)=(u_0g_1u_1^{-1}, u_1g_2u_2^{-1},\ldots, u_{k-1}g_ku_k^{-1}).
\]
\begin{theorem}
    The function $M_{\underline{\alpha}}$ is $U^{k+1}$-invariant.
\end{theorem}
\begin{proof}
    In terms of ${\bm g}$, the action can be rewritten as
\begin{equation}\label{eq:u-action}
    {\bm u}.{\bm g}=u_0\begin{bmatrix}
        \id_n & g_1 & g_1g_2 &\cdots & g_1g_2\cdots g_k
    \end{bmatrix}\diag(u_0^{-1}, u_1^{-1},\ldots, u_k^{-1}),
\end{equation}
where $\diag(u_0^{-1}, u_1^{-1},\ldots, u_k^{-1})$ is the block diagonal matrix, whose $j$'s block is $u_{j-1}^{-1}$. Thus, by applying the Cauchy-Binet formula to the r.h.s. of \eqref{eq:u-action} , we have
\[
   \Delta_{[1,n],J(\underline{\alpha})}({\bm u}.{\bm g})=\sum_{L} \det(u_0)\Delta_{[1,n],L}({\bm g})\Delta_{L, J(\underline{\alpha})}(\diag(u_0^{-1}, u_1^{-1},\ldots, u_k^{-1})).
\]
Note that $\Delta_{L, J(\underline{\alpha})}(\diag(u_0^{-1}, u_1^{-1},\ldots, u_k^{-1}))\neq 0$ if and only if $L=J(\underline{\alpha})$. Thus by the fact $\det(u_0)=1$ and $\Delta_{J(\underline{\alpha}), J(\underline{\alpha})}(\diag(u_0^{-1}, u_1^{-1},\ldots, u_k^{-1}))=1$, we get $\Delta_{[1,n],J(\underline{\alpha})}({\bm u}.{\bm g})=\Delta_{[1,n],J(\underline{\alpha})}({\bm g})$.
\end{proof}

Denote by $\mathcal{M}_k:=\GL_n^k/\!\!/U^{k+1}$ the GIT quotient of $\GL_n$ by $U^{k+1}$, which is one of the main objects studied in \cite{bl}.  
Recall that the $U^2$-action on $\GL_n$ by left and right multiplication restricts to a free action on the Bruhat cell $Bw_0B$.

\begin{proposition}\label{Pro:anykcomplete}
    For any $k>0$, the  the cardinality of the set
    \begin{equation}\label{eq:invingen}
        {\bm M}_k:=\{M_{\underline{\alpha}}\mid \alpha=(0^l,i,j,0^{k-l-2})\text{~for some~} 0\leqslant l\leqslant k-2, 0\leqslant i,j\leqslant n, 0< i+j\leqslant n\},
    \end{equation}
     is equal to $\dim \mathcal{M}_k$. Furthermore, the functions $M_{\underline{\alpha}}$ in this set descend to algebraically independent functions ${\overline M}_{\underline{\alpha}}$ on $\mathcal{M}_k$.
\end{proposition}
\begin{proof}
    Firstly, consider the birational isomorphism
    \[
        \psi\colon \GL_n^k\to \GL_n^k\ :\ (g_1,\ldots,g_k)\mapsto (g_1,g_1g_2,\ldots,g_1\cdots g_k)
    \]
   and the map
    \begin{align*}
        \eta\colon H\times (U \times H)^{k-1} &\to \GL_n^k\\
        (h_1,u_2,h_2,u_3,h_3,\ldots,u_{k},h_{k})&\mapsto (h_1\overline{w_0}, u_2h_2\overline{w_0},\ldots,u_{k}h_{k}\overline{w_0}).
    \end{align*}
    Denote by $\pr\colon \GL_n^k\to \mathcal{M}_k$ the natural projection. Thus it is not hard to see that
    \[
    \delta:=\pr\circ \psi^{-1} \circ \eta\colon  H\times (U \times H)^{k-1}\to \mathcal{M}_k
    \]
    is an open embedding.  Put $g_1\cdots g_i:=u_ih_i\overline{w_0}$, where $u_1=\id_n$. Thus $g_1\in Hw_0$ and $g_i\in B_-w_0$ for $i>1$. By definition,
   \begin{align*}
        \overline{M}_{0^{l-1},i,j,0^{k-l-1}}(g_1,...,g_{k})&=\overline{M}_{0^{l-1},i,j}(g_1,...,g_{l+1})=\overline{M}_{i,j}(g_1\cdots g_l, g_{l+1})\\
        &=\Delta_{[1,i+j]^\op,[1,i+j]^\op}(h_l)\Delta_{[i+1,i+j],[1,j]}(g_{l+1})\\
         &=\overline{M}_{0^{l-1},i,0^{k-l}}(g_1,\ldots ,g_l)\Delta_{[i+1,i+j],[1,j]}(g_{l+1}).
   \end{align*}
    
    Note that $\psi^{-1}\circ \eta$ is a birational isomorphism onto its image. 
    
    We next prove the proposition by induction on $k$. For $k=2$, the above calculation shows
    \[
    \overline{M}_{i,j}(g_1,g_2)=\overline{M}_i(g_1)\Delta_{[i+1,i+j],[1,j]}(g_2).
    \]
    Note that $\{\Delta_{[i+1,i+j],[1,j]}\mid 0\leqslant i \leqslant n, 1\leqslant j \leqslant n, 1\leqslant i+j \leqslant n\}$ is a set of independent functions on $B_-w_0$. Recall that $\{\overline{M}_i\mid 1\leqslant i \leqslant n\}$ is a set of  independent functions on $Hw_0$. Besides,  $\{\overline{M}_{i,0}\mid 1\leqslant i \leqslant n\}$ only depends on  $g_1$ and $\{\overline{M}_{i,j}/\overline{M}_{i,0}\mid j\neq 0\}$ only depends on $g_2$. Thus $\overline{\bm M}_2$ is a set of  independent functions on $\mathcal{M}_2$.
    
    Suppose the proposition is true for $k=p-1$. For $k=p$, we see that $\{\overline{M}_{0^{p-2},i,j}/ \overline{M}_{0^{p-2},i,0}\mid j\neq 0\}$ is independent and only depends on $g_p$. By induction ${\bm M}_p\setminus \{\overline{M}_{0^{p-2},i,j}\mid j\neq 0\}$ is independent and only depends on $g_1,\ldots,g_{p-1}$. Note that $\overline{M}_{0^{p-2},i,0}\in {\bm M}_p\setminus \{\overline{M}_{0^{p-2},i,j}\mid j\neq 0\}$. Thus ${\bm M}_p$ is a set of independent functions on $\mathcal{M}_{p}$.
\end{proof}

\begin{example}\label{ex:k=3}
    For $k=3$, consider
   \begin{align*}
    {\bm M}_2'&:=\{M_{i,0,j}\mid 0\leqslant i,j\leqslant n, 0<i+j\leqslant n\}\cup \{M_{i,j,k}\mid 0\leqslant i,j,k\leqslant n,i+j+k=n\}.
   \end{align*}
   Similar to Proposition \ref{Pro:anykcomplete} show that ${\bm M}_2'$ descend to an algebraically independent set of functions on $\mathcal{M}_3$ as well (see further interesting properties in Example \ref{ex:pluasmutation}). For $k\geqslant 3$, functions $M_{\underline{\alpha}}$'s are not algebraically independent, in general (see Section \ref{sec:geooct} for more details).
\end{example}

\subsection{Potential on \texorpdfstring{$B_-^k$}{Mk} and rhombus inequalities}
\label{subsec:potential on Bk}

In this section, we introduce a $U^{k+1}$-invariant function $\Phi_k$ on $\GL_n^k$ (note that the function $\Phi_k$ is called the central charge in \cite{bk2, bl}). 

Let $\theta$ be a positive chart on $B_-$, and extend it to $B_-^k$
\begin{equation}\label{eq:posforB}
    \theta\times \cdots \times \theta\colon (\mathbb{C}^*)^{\frac{1}{2}n(n+1)}\times\cdots \times (\mathbb{C}^*)^{\frac{1}{2}n(n+1)}\to B_- \times \cdots \times B_-
\end{equation}
in the natural way. Thus $B_-^k$ becomes a positive variety. We will show that $\Phi_k$ restricts to a positive function on $B_-^k$, and that the condition $\Phi_k^t\leqslant 0$ is equivalent to rhombus inequalities. Furthermore, we will show that functions  $\overline{M}_{\underline{\alpha}}$ define a positive structure for $(\mathcal{M}_k, \overline{\Phi_k})$. 

\begin{definition}
    For $k\geqslant 2$, let $\Phi_k$ be the following rational function on $\GL_n^k$ 
    \[
        \Phi_k(g_1,\ldots,g_k):=\Phi_{\rm BK}(g_1)+\cdots+\Phi_{\rm BK}(g_k)-\Phi_{\rm BK}(g_1\cdots g_k).
    \]
\end{definition}

By Eq \eqref{eq:propofBK}, $\Phi_k$ is $U^{k+1}$-invariant and  descends to a function $\overline{\Phi_k}$ on $\mathcal{M}_k$. 

First, we focus on the case of $k=2$. 
\begin{proposition}\label{prop:ceninM}
    In terms of $M_{i,j}$, the potential $\Phi_2$ on $\GL_n^2$ has form
    \[
        \sum _{i,j} \frac{M_{i,j-1}}{M_{i,j}}\cdot \frac{M_{i+1,j}}{M_{i+1,j-1}}+ \frac{M_{i,j-1}}{M_{i,j}}\cdot \frac{M_{i-1,j+1}}{M_{i-1,j}}+\frac{M_{i-1,j+1}}{M_{i,j+1}}\frac{M_{i+1,j}}{M_{i,j}},
    \]
    where in the summation, we discard the Laurent monomial terms that contain  $M_{i,j}$ for $(i,j)\notin \{(i,j)\mid 0\leqslant i,j\leqslant n, 0\leqslant i+j\leqslant n\}$.
\end{proposition}
\begin{proof}
      Recall that we view $H\times U\times H$ as an open subset of $\mathcal{M}_2$ via $(h_1,u,h_2)\mapsto [h_1\overline{w_0},uh_2\overline{w_0}]$. In terms of elements $(h_1,u,h_2)\in H\times U\times H$, we have
    \[
        \Phi(h_1,u, h_2)=\Phi_{\rm BK}(h_1\overline{w_0})+\Phi_{\rm BK}(uh_2\overline{w_0})-\Phi_{\rm BK}(h_1\overline{w_0}uh_2\overline{w_0})=\chi(u)-\Phi_{\rm BK}(h_1\overline{w_0}uh_2\overline{w_0}).
    \]
    Since $\Phi_{\rm BK}(\overline{w_0}g^T\overline{w_0})=-\Phi_{\rm BK}(g)$, we get
    \[
        \Phi(h_1,u, h_2)=\chi(u)+\Phi_{\rm BK}(h_2u^T\overline{w_0}^{-1}h_1\overline{w_0}).
    \]
    Note that by \eqref{eq:expansion}, we have
    \[
    \overline{M}_{i,j}(h_1\overline{w_0},uh_2\overline{w_0} )= \Delta_{[1,i+j]^\op,[1,i+j]^\op}(h_1)\Delta_{[i+1,i+j],[1,j]^\op}(uh_2).
    \]
   Denote by $b=h_2u^T\overline{w_0}^{-1}h_1\overline{w_0}$. Then
   \begin{equation}\label{eq:deltatoM}
       \Delta_{i,j}(b):=\Delta_{[1,j]^\op,[i+1,i+j]}(b)=\overline{M}_{i,j}/ \overline{M}_{i,0}.
   \end{equation}
   Recall that in terms of $\Delta_{i,j}$ as in Eq \eqref{eq:geoGZpattern}, we have the expression of BK potential as
   \[
    \sum_{(i,j)\in \delta_n'} \frac{\Delta_{i,j-1}}{\Delta_{i,j}}\cdot \frac{\Delta_{i+1,j}}{\Delta_{i+1,j-1}}+ \sum_{(i,j)\in \delta_n''} \frac{\Delta_{i,j-1}}{\Delta_{i,j}}\cdot \frac{\Delta_{i-1,j+1}}{\Delta_{i-1,j}}.
   \]
   where $\delta_n'$ and $\delta_n''$ are as in Theorem \ref{Thm:poten}. Thus we get
    \[
        \Phi_{\rm BK}(b)=\sum _{i,j} \frac{\overline{M}_{i,j-1}}{\overline{M}_{i,j}}\cdot \frac{\overline{M}_{i+1,j}}{\overline{M}_{i+1,j-1}}+ \frac{\overline{M}_{i,j-1}}{\overline{M}_{i,j}}\cdot \frac{\overline{M}_{i-1,j+1}}{\overline{M}_{i-1,j}}.
    \]

 Next, we would like to give a formula of $\chi(u)=\sum u_{i,i+1}$ for $u\in U$ in terms of minors. We first write down factorization parameters for $u^T$ as in the following graph. Then it is clear that $\chi(u)=\sum a_{j}^{(n-i)}$. Thus we need to write down $a_j^{(n-i)}$ using minors $\Delta_{i,j}$. Note that
\[
    a_j^{(n)} a_j^{(n-1)}\cdots a_j^{(n-i)}= \frac{\Delta_{[1,i+1]^\op,[1,i+1]^\op-j+1}}{\Delta_{[1,i+1]^\op,[1,i+1]^\op-j+2}}=\frac{\Delta_{n-i-j,i+1}}{\Delta_{n-i-j+1,i+1}},
\]
 which gives 
\begin{equation}\label{eq:facpara}
    a_j^{(n-i)}=\frac{\Delta_{[1,i+1]^\op,[1,i+1]^\op-j+1}}{\Delta_{[1,i+1]^\op,[1,i+1]^\op-j+2}}\cdot \frac{\Delta_{[1,i]^\op,[1,i]^\op-j+2}}{\Delta_{[1,i]^\op,[1,i]^\op-j+1}}=\frac{\Delta_{n-i-j,i+1}}{\Delta_{n-i-j+1,i+1}}\cdot \frac{\Delta_{n-i-j+2,i}}{\Delta_{n-i-j+1,i}}.
\end{equation}

\begin{figure}[H]
    \centering
\begin{tikzpicture}[scale=0.4]
    \foreach \n in {0,...,5}{
    \draw (-1,2*\n)--(23,2*\n);
    }
    \foreach \n in {0,1,3,4}{
    \draw (3+4*\n,2)--(5+4*\n,0);
    }
    \foreach \n in {0,...,2}{
    \draw (3+4*\n,6)--(5+4*\n,4);
    }
    \foreach \n in {0,...,1}{
    \draw (3+4*\n,8)--(5+4*\n,6);
    }
    \draw (3,10)--(5,8);
    \node at (12,1){$\cdots$};
    \foreach \n in {0,...,3}{
    \node at (4+4*\n,3) {$\cdots$};
    }
    \node at (5.5,1) {$a^{(2)}_2$};
    \node at (9.5,1) {$a^{(3)}_3$};
    \node at (17.8,1) {$a^{(n-1)}_{n-1}$};
    \node at (21.5,1) {$a^{(n)}_n$};

    \node at (5.8,5) {$a^{(n-2)}_2$};
    \node at (9.8,5) {$a^{(n-1)}_3$};
    \node at (13.8,5) {$a^{(n)}_{4}$};

    \node at (5.8,7) {$a^{(n-1)}_2$};
    \node at (9.5,7) {$a^{(n)}_3$};
    
    \node at (5.5,9) {$a^{(n)}_{2}$};

    \node at (-1,0)[left] {$1$};
    \node at (-1,2)[left] {$2$};
    \node at (-1,4)[left] {$n-3$};
    \node at (-1,6)[left] {$n-2$};
    \node at (-1,8)[left] {$n-1$};
    \node at (-1,10)[left] {$n$};
\end{tikzpicture}
\end{figure}
   
   By Eq \eqref{eq:deltatoM} and \eqref{eq:facpara}, 
    \[
        \chi(u)=\sum \frac{\overline{M}_{i-1,j+1}}{\overline{M}_{i,j+1}}\frac{\overline{M}_{i+1,j}}{\overline{M}_{i,j}}.
    \]
    Thus the statement follows since  $M_{i,j}$'s and $\Phi_2$ are all $U^{k+1}$-invariant.
\end{proof}

\begin{corollary}\label{cor:deltak}
    For $k\geqslant 2$ and $0\leqslant l\leqslant k-2$, denote by $M^{(l)}_{i,j}=M_{0^l,i,j,0^{k-l-2}}$. Then on $\GL_n^k$
    \[
        \Phi_k=\sum_{l=0}^{n-2}\sum _{i,j} \frac{M^{(l)}_{i,j-1}}{M^{(l)}_{i,j}}\cdot \frac{M^{(l)}_{i+1,j}}{M^{(l)}_{i+1,j-1}}+ \frac{M^{(l)}_{i,j-1}}{M^{(l)}_{i,j}}\cdot \frac{M^{(l)}_{i-1,j+1}}{M^{(l)}_{i-1,j}}+\frac{M^{(l)}_{i-1,j+1}}{M^{(l)}_{i,j+1}}\frac{M^{(l)}_{i+1,j}}{M^{(l)}_{i,j}}.
    \]
\end{corollary}
\begin{proof}
    For $k=2$, the statement is given by Proposition \ref{prop:ceninM}. For $k=3$, note that
    \begin{align*}
        \Phi_3(g_1,g_2,g_3)&=\Phi_{\rm BK}(g_1)+\Phi_{\rm BK}(g_2)-\Phi_{\rm BK}(g_1g_2)+\Phi_{\rm BK}(g_1g_2)+\Phi_{\rm BK}(g_3)-\Phi_{\rm BK}(g_1g_2g_3)\\
        &=\Phi_2(g_1,g_2)+\Phi_2(g_1g_2,g_3).
    \end{align*}
       Observe  that
        \[
            M_{i,j}(g_1g_2,g_3)=M_{0,i,j}(g_1,g_2,g_3).
        \]
        Hence, one can represent $\Phi_2(g_1g_2,g_3)$ in terms of $M_{0,i,j}$. A similar argument applies for any $k\geqslant 3$. 
\end{proof}

Now we first restrict the potential $\Phi_k$ to $B_-^k$ and consider the positive variety $(B_-^k,\Phi_k)$. What follows is an analogs of Theorem \ref{theorem PN rhombi inequalities}. For $0\leqslant l \leqslant k-2$, denote by
\[
    \mathcal{F}_l:=\{\underline{\alpha}\mid \underline{\alpha}=(0^l, i, j, 0^{k-2-l}), 0\leqslant i,j\leqslant n, 0\leqslant i+j\leqslant n \}
\]
a $2$-face of simplex $\Delta^k(n)$. We label a small rhombus on $\mathcal{F}_l$ as follows:
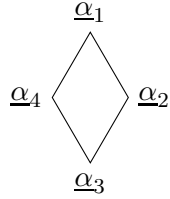
\begin{figure}[H]
    \centering
    \begin{tikzpicture}
        \draw (0,0) -- (0.5,0.866) --(1,0) -- (0.5, -0.866) -- cycle;
        \node[anchor=east] at (0,0) {$\underline{\alpha}_{4}$};
        \node[anchor=west] at (1,0) {$\underline{\alpha}_{2}$};
        \node[anchor=south] at (.5,.866) {$\underline{\alpha}_{1}$};
        \node[anchor=north] at (.5,-.866) {$\underline{\alpha}_{3}$};
    \end{tikzpicture}\caption{Order of vertices of a rhombus.}\label{fig:orderofrhom}
\end{figure}

\begin{proposition}\label{pro:tropcen=rhom}
    Restrict $\Phi_k$ to $B_-^k$ and consider the positive structure on $B_-^k$ defined by Eq \eqref{eq:posforB}. Then, the functions $M_{\underline{\alpha}}$'s are positive and
    $$
    \Phi_k^t\leqslant 0 \Leftrightarrow
    M^t_{\underline{\alpha}_1}+M^t_{\underline{\alpha}_3}\leqslant M^t_{\underline{\alpha}_2}+M^t_{\underline{\alpha}_4}
    $$ 
    for all small rhombi  on faces $\mathcal{F}_l$ with vertices $\underline{\alpha}_{1},\ldots,\underline{\alpha}_{4}$.  
\end{proposition}
\begin{proof}
    The Lindstr{\"o}m's Lemma and Proposition \ref{pro:expofM} imply that $M_{\underline{\alpha}}$'s are positive with respect to $\theta_{\mathrm{st}}\times \cdots \times \theta_{\mathrm{st}}$ on $B_-^k$. What remains in the proposition follows directly from the expression in Corollary \ref{cor:deltak}.
\end{proof}

The results stated above descend to the positive variety $(\mathcal{M}_k, \overline{\Phi_k})$:

\begin{proposition}\label{Pro:l=2complete}
    The set of functions ${\bm M}$ defined by Eq \eqref{eq:invingen} defines a positive chart on $(\mathcal{M}_k,\overline{\Phi_k})$, and we have 
     $$
     \overline{\Phi_k}^t\leqslant 0 \Leftrightarrow
     \overline{M}^t_{\underline{\alpha}_1}+\overline{M}^t_{\underline{\alpha}_3}\leqslant \overline{M}^t_{\underline{\alpha}_2}+\overline{M}^t_{\underline{\alpha}_4}
     $$ for all small rhombi  on faces $\mathcal{F}_l$ for $0\leqslant l\leqslant k-2$ with vertices $\underline{\alpha}_{1},\ldots,\underline{\alpha}_{4}$.  
\end{proposition}
\begin{proof}
    The statements follow from the fact that $\card({\bm M}_k)=\dim \mathcal{M}_k$ and $\overline{\Phi_k}$ has a positive expression in terms of functions in ${\bm M}_k$ by Corollary \ref{cor:deltak}.
\end{proof}

\subsection{Geometric octahedron recurrence}\label{sec:geooct}
In this section, we focus on the case of $k=3$. In that case, functions $M_{i,j,k}$ satisfy the following recurrence property:
\begin{theorem}\label{Thm:octrec}
    The functions $M_{i,j,k}$'s satisfy the geometric octahedron recurrence:
    \begin{equation}\label{eq:geooct}
        M_{i+1,j,k+1}M_{i,j+1,k}=M_{i,j+1,k+1}M_{i+1,j,k}+M_{i,j,k+1}M_{i+1,j+1,k},
    \end{equation}
    where $M_{i,j,k}=0$ if $(i,j,k)\notin \Delta^3(n)$. Moreover, each $M_{i,j,k}$ is expressible both as a rational function of  $M_{i,j,k}$'s  corresponding to the faces $\mathcal{F}_0$ and $\mathcal{F}_1$ and a rational function of  $M_{i,j,k}$'s  corresponding to the faces of $\Delta^3(n)$ other than $\mathcal{F}_0$ and $\mathcal{F}_1$.
\end{theorem}

\begin{proof}
    Recall that $M_{i,j,k}(g_1,g_2,g_3)=\Delta_{[1,n], J_{i,j,k}}(\bm g)$, where
    \[
    J_{i,j,k}=([1,n]\setminus [1,i+j+k]^\op)\cup ([1,i]+n)\cup ([1,j]+2n)\cup ([1,k]+3n).
     \]
    Denote
    \[
        i_1=n-(i+j+k), i_2=i+1+n, i_3=j+1+2n, i_4=k+1+3n.
    \]
    A direct calculation gives
    \begin{align*}
        J_{i+1,j,k+1}&=J'\cup \{i_2,i_4\},& &&  J_{i,j+1,k}&=J'\cup \{i_1,i_3\}, &&&
        J_{i,j+1,k+1}&=J'\cup \{i_3,i_4\}; \\ J_{i+1,j,k}&= J'\cup \{i_1,i_2\}, &&& 
        J_{i,j,k+1}&= J'\cup \{i_1,i_4\}, &&&  J_{i+1,j+1,k}&= J'\cup \{i_2,i_3\},
    \end{align*}
    where
    \[
        J'=[1,n-(i+j+k)-1]\cup ([1,i]+n)\cup ([1,j]+2n)\cup ([1,k]+3n).
    \]
    Finally, since $i_1<i_2<i_3<i_4$, the Pl\"ucker relation of ${\bm g}$ gives
    \[\Delta_{[1,n], J'\cup\{i_2,i_4\}}\Delta_{[1,n], J'\cup\{i_1,i_3\}}=\Delta_{[1,n], J'\cup\{i_3,i_4\}}\Delta_{[1,n], J'\cup\{i_1,i_2\}}+\Delta_{[1,n], J'\cup\{i_1,i_4\}}\Delta_{[1,n], J'\cup\{i_2,i_3\}}.\]
    The recurrence \eqref{eq:geooct} now follows directly.

    We now introduce a partial order on the set of triples $(i,j,k)$: we have $(i,j,k) < (i',j',k')$ if $i+k < i'+k'$, or (1) if $i=0$ or $k=0$ and $i'\neq 0, k'\neq 0$, (2) if $j'=0$ or $i'+k'=n$ and $j\neq 0$ and $i+k\neq n$. Thus Eq \eqref{eq:geooct} can be rewritten as
    \[
        M_{i+1,j,k+1}=M_{i,j+1,k}^{-1}\left(M_{i,j+1,k+1}M_{i+1,j,k}+M_{i,j,k+1}M_{i+1,j+1,k}\right).
    \]
    Each indices appeared on the r.h.s of the equation has a smaller partial order than $(i+1,j,k+1)$. Thus by  recurrence, one can express $M_{i+1,j,k+1}$ using $M_{i,j,k}$'s corresponding to the pair the faces $\mathcal{F}_0$ and $\mathcal{F}_1$. Similarly, we have
    \[
    M_{i,j+1,k}=M_{i+1,j,k+1}^{-1}\left(M_{i,j+1,k+1}M_{i+1,j,k}+M_{i,j,k+1}M_{i+1,j+1,k}\right).
    \]
    Each indices appeared on the r.h.s of the equation has a bigger partial order than $(i,j+1,k)$. Thus by  recurrence, one can express $M_{i,j+1,k}$ using $M_{i,j,k}$'s corresponding to the pair of the faces of $\Delta^3(n)$ other than $\mathcal{F}_0$ and $\mathcal{F}_1$.
\end{proof}
\begin{remark}
    Theorem \ref{Thm:octrec} descends to the property of $\mathcal{M}_3$ as stated in the first part of Theorem \ref{intro:M_3properties}. 
\end{remark}

\begin{example}\label{ex:pluasmutation}
    Retain the notation in Example \ref{ex:k=3}. For $n=2$, we have
    \begin{align*}
        {\bm M}_2=\{M_{1,0,0},M_{0,1,0}, M_{0,0,1},M_{1,1,0},M_{0,1,1},M_{2,0,0},M_{0,2,0},M_{0,0,2}\};\\
        {\bm M}_2'=\{M_{1,0,0},M_{1,0,1}, M_{0,0,1},M_{1,1,0},M_{0,1,1},M_{2,0,0},M_{0,2,0},M_{0,0,2}\}.
    \end{align*}
    The sets ${\bm M}_2$ and ${\bm M}_2'$ are related to each other by the recurrence:
    \[
    M_{1,0,1}M_{0,1,0}=M_{0,1,1}M_{1,0,0}+M_{0,0,1}M_{1,1,0},
    \]
    which can be thought of as a cluster mutation. For $n=3$, the set ${\bm M}_3$ can be transformed into ${\bm M}_3'$ by $4$ subsequent moves: Denote  ${\bm F}={\bm M}_3\cap {\bm M}_3'$, and consider
    \begin{align*}
        {\bm M}_3&= {\bm F} \cup \{ a=M_{0,1,0}, b=M_{0,2,0}, c=M_{1,1,0}, d=M_{0,1,1} \};\\
        {\bm N}_1&= {\bm F} \cup \{ a'=M_{1,0,1}, b, c, d \};\\
        {\bm N}_2&= {\bm F} \cup \{ a', b'=M_{1,1,1}, c, d \};\\
        {\bm N}_3&= {\bm F} \cup \{ a', b', c'=M_{2,0,1}, d \};\\
        {\bm M}_3'&= {\bm F} \cup \{ a', b', c', d'=M_{1,0,2} \}.
    \end{align*}
    The cluster nature of the variety $\mathcal{M}_k$ will be studied elsewhere.
\end{example}

We can now prove the remaining rhombus inequalities for functions $M^t_{i,j,k}$:

\begin{theorem}\label{thm:allrhombi}
    Consider the positive variety with potential $(B_-^3,\Phi_3)$ and restrict functions $M_{\underline{\alpha}}$ to $B_-^3$. If $\Phi_3^t\leqslant 0$, then for all small plane rhombi  in $\Delta^3(n)$ with vertices $\underline{\alpha}_{1},\ldots,\underline{\alpha}_{4}$ ordered as in Fig.\ref{fig:orderofrhom},  we have $M^t_{\underline{\alpha}_1}+M^t_{\underline{\alpha}_3}\leqslant M^t_{\underline{\alpha}_2}+M^t_{\underline{\alpha}_4}$.
\end{theorem}

\begin{proof}

Consider a plane rhombus $\lozenge$ with corners $\underline{\alpha}_1,\ldots,\underline{\alpha}_4$ as in Fig.\ref{fig:orderofrhom}, the long diagonal $\underline{\alpha}_1 - \underline{\alpha}_3$ and the short diagonal $\underline{\alpha}_2 - \underline{\alpha}_4$. Note that 
$\underline{\alpha}_1 + \underline{\alpha}_3=\underline{\alpha}_2 + \underline{\alpha}_4$. 

We introduce a partial order on the set of triples $(i,j,k)$: we have $(i,j,k) < (i',j',k')$ if $i+k < i'+k'$, or if $i=0$ or $k=0$ and $i'\neq 0, k'\neq 0$. This partial order induces a partial order on the set of plane rhombi: we say that $\lozenge \prec \lozenge'$ if $\underline{\alpha}_1 + \underline{\alpha}_3 < \underline{\alpha}_1' + \underline{\alpha}_3'$.
To a rhombus $\lozenge$ we associate the rational function
$$
\Phi_\lozenge = \frac{M_{\underline{\alpha}_1} M_{\underline{\alpha}_3}}{M_{\underline{\alpha}_2} M_{\underline{\alpha}_4}}. 
$$
The octahedron recurrence implies relations between functions $\Phi_\lozenge$ associated to different plane rhombi. In more detail, let 
$$
M_a M_d = M_b M_c + M_e M_f
$$
be one of the octahedron recurrence equations. Assuming that we are away from the walls of $\Delta^3(n)$, we can re-write this equation as
\begin{equation}     \label{eq:alpha_v}
M_{\underline{\alpha}+v_3} M_{\underline{\alpha}+v_1+v_2} = M_{\underline{\alpha}+v_1} M_{\underline{\alpha}+v_2+v_3} +
M_{\underline{\alpha}+v_2} M_{\underline{\alpha}+v_1+v_3}
\end{equation}
in 8 different ways corresponding to different choices of $\underline{\alpha}, v_1, v_2, v_3$.
These 8 ways correspond to different identifications in the two equations above. For instance, one can set
$$
a=\underline{\alpha}+v_3, d=\underline{\alpha}+v_1+v_2, b=\underline{\alpha}+v_1, c=\underline{\alpha}+v_2+v_3, 
e=\underline{\alpha}+v_2, f=\underline{\alpha}+v_1+v_3,
$$
which implies (note that $a+d=b+c=e+f$) :
\[
v_1=f-a;\quad v_2=c-a; \quad v_3=c-e;\quad \underline{\alpha}=e-c+a. 
\]
Permutation of elements in pairs $(a,d), (b,c), (e,f)$ leads to 
other values of $\alpha, v_1, v_2, v_3$.
Four possible choices of the vectors $v_1, v_2, v_3$ are given in the table below:
\begin{center}
\begin{tabular}{|c|c|c|c|c|}
$v_1$ & $(0,0,-1)$ & $(-1,0,0)$ & $(-1,1,0)$ & $(0,1,-1)$ \\
$v_2$ & $(-1,0,0)$ & $(-1,1,0)$ & $(0,1,-1)$ & $(0,0,-1)$ \\
$v_3$ & $(0,-1,0)$ & $(-1,0,1)$ & $(0,1,0)$ & $(1,0,-1)$ \\
\end{tabular}   
\end{center}
The other four choices are obtained by simultaneously reversing signs of the three vectors. The following figure illustrates the position of the vertex $\underline{\alpha}$ and of the vectors $v_1, v_2, v_3$ with respect to the vertices involved in the octahedron recurrence:
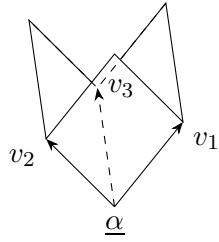
\begin{center}
    \begin{tikzpicture}[scale=.9]
        \tikzmath{ 
            \l1=0; \l2=0; 
            \r1=-4; \r2=-.5;
            \u1=-2; \u2=-2.5;
            \d1=-2.5; \d2=1;}

        \draw[-{Stealth}] (\u1,\u2) node[below]{$\underline{\alpha}$} -- (.5*\l1+.5*\u1, .5*\l2+.5*\u2) node[below right]{$v_1$} ;
        \draw[-{Stealth}] (\u1,\u2) -- (.5*\r1+.5*\u1, .5*\r2+.5*\u2) node[below left]{$v_2$} ;
        \draw[opacity=.5,-{Stealth},dashed] (\u1,\u2) -- (.5*\d1+.5*\u1, .5*\d2+.5*\u2) node[right,opacity=1]{$v_3$} ;

        \draw[opacity=.5,dashed]  (.5*\l1+.5*\d1, .5*\l2+.5*\d2) -- (.5*\d1+.5*\u1, .5*\d2+.5*\u2) -- (.5*\r1+.5*\d1, .5*\r2+.5*\d2); 
        
        \draw (.045*\r1+.5*\d1+.455*\u1,.045*\r2+ .5*\d2+.455*\u2) -- (.5*\r1+.5*\d1, .5*\r2+.5*\d2) -- (.5*\r1+.5*\u1, .5*\r2+.5*\u2) -- (.5*\l1+.5*\r1, .5*\l2+.5*\r2) -- (.5*\l1+.5*\u1, .5*\l2+.5*\u2) -- (.5*\l1+.5*\d1, .5*\l2+.5*\d2) -- (.17*\l1+.5*\d1+.33*\u1, .17*\l2+.5*\d2+.33*\u2);
\end{tikzpicture}
\end{center}

Equation \eqref{eq:alpha_v} implies
\begin{equation} \label{eq:3rhombi}
    \frac{M_{\underline{\alpha}} M_{\underline{\alpha} +v_1+v_2}}{M_{\underline{\alpha}+v_1}M_{\underline{\alpha}+v_2}} =
    \frac{M_{\underline{\alpha}} M_{\underline{\alpha} +v_1+v_3}}{M_{\underline{\alpha}+v_1}M_{\underline{\alpha}+v_3}}
    +\frac{M_{\underline{\alpha}} M_{\underline{\alpha} +v_2+v_3}}{M_{\underline{\alpha}+v_2}M_{\underline{\alpha}+v_3}}.
\end{equation}
Note that each fraction in this equation is a rational function corresponding to a plane rhombus. Hence, each octahedron recurrence equation gives rise to 8 linear relations between the functions $\Phi_\lozenge$. Denote the rhombi corresponding to the terms of the equation \eqref{eq:3rhombi} by $\lozenge_{12}, \lozenge_{13}$ and $\lozenge_{23}$. Note that for the vectors $v_{1,2,3}$ in the table above, we have $\lozenge_{12} \prec\lozenge_{13}, \lozenge_{23}$. And for the vectors with opposite signs, we have $\lozenge_{12} \succ \lozenge_{13}, \lozenge_{23}$.

For each plane rhombus, the rational function $\Phi_\lozenge$ enters in exactly two relations 
\eqref{eq:3rhombi} corresponding to the opposite signs of vectors $v_{1,2,3}$. Hence, in one of these equations the side where $\Phi_\lozenge$ appears is smaller in the partial order than the other side of the equation. In the second equation, it is larger in the partial order. We will be interested in the latter case. For a plane rhombus $\lozenge$, there are two possibilities: either we have
\begin{equation}     \label{eq:Phi=Phi+Phi}
\Phi_\lozenge = \Phi_{\lozenge'} + \Phi_{\lozenge''}
\end{equation}
with $\lozenge \succ \lozenge', \lozenge''$, or
\begin{equation}     \label{eq:Phi+Phi=Phi}
\Phi_\lozenge + \Phi_{\lozenge'} = \Phi_{\lozenge''}
\end{equation}
with $\lozenge, \lozenge' \succ \lozenge''$.

Recall that until now we were assuming that we were away from the walls of $\Delta^3(n)$. The case-by-case analysis shows that $\Phi_\lozenge$ satisfies one of the equations \eqref{eq:Phi=Phi+Phi} and \eqref{eq:Phi+Phi=Phi}, unless $\lozenge$ belongs to one of the faces $\mathcal{F}_0$ or $\mathcal{F}_1$. In the following, we show an example of this analysis.

Denote $\underline{\alpha}=(i,j,k)$ and choose $v_1=(1,0,0), v_2 =(1, -1, 0), v_3 =(1, 0, -1)$. The rhombus $r_{12}$ has vertices
$$
\underline{\alpha}=(i, j, k), \hskip 0.3cm
\underline{\alpha}+v_1=(i+1, j, k), \hskip 0.3cm
\underline{\alpha}+v_2=(i+1, j-1, k), \hskip 0.3cm
\underline{\alpha}+v_1+v_2=(i+2, j-1, k).
$$
By assumptions, these vertices belong to $\Delta^3(n)$ which implies that
$i \geqslant 0, j \geqslant 1, k \geqslant 0$ and $i+j+k+1 \leqslant n$. The remaining three vertices of the rhombi $\lozenge_{13}$ and $\lozenge_{23}$ are
$$
\underline{\alpha}+v_3=(i+1, j, k-1), \hskip 0.3cm
\underline{\alpha}+v_1+v_3=(i+2, j, k-1), \hskip 0.2cm
\underline{\alpha}+v_2+v_3=(i+2, j-1, k-1).
$$
There vertices belong to $\Delta^3(n)$ unless $k=0$, but in that case the rhombus $\lozenge_{12}$ belongs to the face $\mathcal{F}_0$, as required. Other cases can be treated in a similar way.

Using the procedure described above, we arrive at the following equation for each plane rhombus $\lozenge$:
$$
\Phi_\lozenge + \sum_{k=1}^K \Phi_{\lozenge_k} = \sum_{l=1}^L \Phi_{\lozenge_l},
$$
where the rhombi $\lozenge_l$ belong to the faces $\mathcal{F}_0$ and $\mathcal{F}_1$. By Corollary \ref{cor:deltak},  $\Phi_{\lozenge_l} + P_l = \Phi_3$, where $P_l$ is a positive function. Hence, $\Phi_\lozenge + P = L \Phi_3$,
where $P$ is a positive function. In turn, this yields $\Phi_\lozenge^t \leqslant \Phi_3^t$.
Then, $\Phi_3^t \leqslant 0$ implies 
$$
\Phi_\lozenge^t =M_{\underline{\alpha}_1}^t + M_{\underline{\alpha}_3}^t - M_{\underline{\alpha}_2}^t -M_{\underline{\alpha}_4}^t \leqslant 0,
$$ 
and this concludes the proof.
\end{proof}


\begin{remark}
    In the algorithm described in the proof of Theorem \ref{thm:allrhombi}, all rhombi $\lozenge$ (with the exception of the very first rhombus and rhombi $\lozenge_l$ on the faces $\mathcal{F}_0$ and $\mathcal{F}_1$) enter twice, first as the smaller (in partial order) rhombi in equations \eqref{eq:Phi=Phi+Phi} and \eqref{eq:Phi+Phi=Phi}, and then as the larger rhombi (again, in partial order). If the first occurrence of $\lozenge$ corresponds to the vertex with the acute angle $\alpha$ and the vectors $v_i$ and $v_j$ emanating from that vertex, its second occurrence will correspond to the vertex $\alpha + v_i +v_j$ and the vectors $(-v_i)$ and $(-v_j)$. In this way, in each step of the algorithm one always changes the vertex with the acute angle in the rhombus under consideration.
\end{remark}

\begin{example} 

    We illustrate the process described above for $n=3$. In each step, we mark the Laurent monomials of interest in red. Consider the rhombus $\lozenge_1$ with vertices $201, 012, 111, 102$ (the first two are acute angles). It is $\lozenge_{23}$ for $\alpha=(201)$ and $v_i$'s as in the 2nd column. The corresponding monomial is $\Phi_{\lozenge_1}=\frac{M_{201}M_{012}}{M_{111}M_{102}}$. 
    Therefore we have the relation $\frac{M_{201}M_{001}}{M_{101}M_{102}}+\textcolor{red}{\Phi_{\lozenge_1}}=\frac{M_{201}M_{011}}{M_{101}M_{111}}=:\Phi_{\lozenge_2}$, where $\lozenge_2$ is a rhombus with vertices $201, 011, 101,  111$.

   For the second step we consider $\lozenge_2$. It is $\lozenge_{12}$ for $\alpha=(011)$ and $v_i$'s negative of those in the 2nd column. (Here the choice of $\alpha$ is clear: it is an acute angle of $\lozenge_2$ which differs from the one used on the previous step).
   Therefore we can use the relation for the rhombi monomials: $\textcolor{red}{\Phi_{\lozenge_2}}=\frac{M_{011}M_{210}}{M_{111}M_{110}}+\frac{M_{011}M_{200}}{M_{101}M_{110}}=:\Phi_{\lozenge_3}+\Phi_{\lozenge_4}$, where $\lozenge_3$ has vertices  $011, 210, 111, 110$, and $\lozenge_4$ has vertices $011, 200, 101, 110$.

    For the third step we should consider two rhombi: $\lozenge_3$ and $\lozenge_4$.  
    
     $\lozenge_3$ is $\lozenge_{13}$ for $\alpha=(210)$ and $v_i$'s as in column 2.
     The relation for rhombi monomials is  $\frac{M_{2,1,0}M_{0,2,1}}{M_{1,1,1}M_{1,2,0}}+\textcolor{red}{\Phi_{\lozenge_3}}=\frac{M_{2,1,0}M_{0,2,0}}{M_{1,1,0}M_{1,2,0}}$.

     $\lozenge_4$ is $\lozenge_{23}$ for $\alpha=(200)$ and $v_i$'s as in column 2. The corresponding relation is
    is $\textcolor{red}{\Phi_{\lozenge_4}}+\frac{M_{2,0,0}M_{0,0,1}}{M_{1,0,1}M_{1,0,0}}=\frac{M_{2,0,0}M_{0,1,0}}{M_{1,0,0}M_{1,1,0}}$.

   After summing up the four equations described above, we have
    \[
    \frac{M_{012}M_{201}}{M_{102}M_{111}}+\cdots=\frac{M_{2,1,0}M_{0,2,0}}{M_{1,1,0}M_{1,2,0}}+\frac{M_{2,0,0}M_{0,1,0}}{M_{1,0,0}M_{1,1,0}},
    \]
    where $\cdots$ is a sum of some monomials (in fact, each corresponding to some rhombus). The monomials on the RHS of the above equality are summands of $\Phi_3$. Thus we have $\Phi_3-\Phi_{\lozenge}$ is positive. 
\end{example}

Summarizing the results described above, we consider the positive variety $(\mathcal{M}_3,\overline{\Phi}_3)$ as in the previous section, we get the second part of Theorem \ref{intro:M_3properties}:
\begin{theorem}\label{thm:sumofmt3}
    For the positive variety $(\mathcal{M}_3,\overline{\Phi}_3)$, we have
    \begin{equation}\label{eq:recforMbar}
        \overline{M}_{i+1,j,k+1}^t+\overline{M}_{i,j+1,k}^t=\max\left\{\overline{M}_{i,j+1,k+1}^t+\overline{M}_{i+1,j,k}^t,\overline{M}_{i,j,k+1}^t+\overline{M}_{i+1,j+1,k}^t\right\}.
    \end{equation}
    We have $\overline{\Phi_3}^t\leqslant 0$ if and only if for all small plan rhombi  in $\Delta^3(n)$ with vertices $\underline{\alpha}_{1},\ldots,\underline{\alpha}_{4}$ ordered as in Fig.\ref{fig:orderofrhom},  there is 
    \begin{equation}\label{eq:rhombiforMbar}
    \overline{M}^t_{\underline{\alpha}_1}+\overline{M}^t_{\underline{\alpha}_3}\leqslant \overline{M}^t_{\underline{\alpha}_2}+\overline{M}^t_{\underline{\alpha}_4}.
    \end{equation}
\end{theorem}
\begin{proof}
    The Eq \eqref{eq:recforMbar} follows from Eq \eqref{eq:geooct} by $U^4$-invariance of $M_{i,j,k}$. For each rhombus, $\Phi_{\lozenge}$ is $U^4$-invariant. Thus following the proof of Theorem \ref{thm:allrhombi}, $\overline{\Phi_3}^t\leqslant 0$ is equivalent to Eq \eqref{eq:rhombiforMbar}. 
\end{proof}

\subsection{From \texorpdfstring{$M_{\underline{\alpha}}$}{Ma} to \texorpdfstring{$m_{\underline{\alpha}}$}{ma} via GZ conditions}

Consider the restriction of $M_{\underline{\alpha}}$ to $B_-^k\subset \GL_n^k$ and extend the positive chart $\theta_{\rm st}$ on $B_-$ to $B_-^k$. The functions $M_{\underline{\alpha}}$'s are positive with respect to $\theta_{\rm st}\times \cdots \times \theta_{\rm st}$ by the Lindstr{\"o}m's Lemma. Our aim in this section is to show that $M_{\underline{\alpha}}^t=m_{\underline{\alpha}}$ under GZ conditions.

We first construct the ``detropicalization'' of $m_{\underline{\alpha}}$. 

\begin{definition}
    Take $(g_1,\ldots,g_k)\in\GL_n^k$ and consider the coefficients of
    \begin{equation}\label{eq:ascoeff}
        \det(x_0\id_n+ x_1g_1+x_2g_1g_2+\cdots+ x_kg_1\cdots g_k).
    \end{equation}
    For a $k$-tuple of integers $\underline{\alpha}$ satisfying $0\leqslant \alpha_i\leqslant n$ and $0\leqslant \sum\alpha_i\leqslant n$, define  $\widetilde{M}_{\underline{\alpha}}$ to be the coefficient of $x_0^{n-\sum \alpha_i} x_1^{\alpha_1}\cdots x_k^{\alpha_k}$.
\end{definition}

\begin{remark}
The determinant \eqref{eq:ascoeff} is similar to the one introduced by Spyer (see \cite[Lemma 4]{Speyer}) in his proof of Horn inequalities using Vinnikov curves.
\end{remark}

Similar to Proposition \ref{pro:expofM}, one can write down $\widetilde{M}_{\underline{\alpha}}$ in terms of minors of $g_i$'s. Let ${\bm g}$ be the $n\times ((k+1)n)$ matrix in Eq \eqref{eq:g}. We have
\begin{proposition}\label{pro:expofMtilta}
For a $k$-tuple of integers $\underline{\alpha}$ satisfying $0\leqslant \alpha_i\leqslant n$ and $0\leqslant \sum\alpha_i\leqslant n$, we have
\[
\widetilde{M}_{\underline{\alpha}}(g_1,\ldots,g_k)=\sum_{J}\Delta_{[1,n],J}({\bm g}),
\]
where the summation is over all possible $J\subset [1,(k+1)n]$ such that $|J|=n$ and 
\[
 |J\cap \left[in+1,(i+1)n\right]|=\alpha_i.
\]
Moreover, for such $J$, put $J_i=(J\cap \left[in+1,(i+1)n\right])-in$ and $L_0=[1,n]\setminus J_0$. We then have
\begin{align}
        \Delta_{[1,n],J}({\bm g})&=\sum_{L_1, \ldots,L_{k-1}} \Delta_{L_0,J_1\cup L_1}(g_1)\Delta_{L_1, J_2\cup L_2}(g_2)\cdots \Delta_{L_{k-1}, J_k}(g_k)\nonumber \\
        &=\sum_{L_1, \ldots,L_{k-1}}\prod_{i=1}^k\Delta_{L_{i-1}, J_i\cup L_i}(g_i),\label{eq:expansion2}
    \end{align}
where $L_i\cap J_i=\emptyset$ and $|L_i|+|J_i|=|L_{i-1}|$.
\end{proposition}
\begin{proof}
    Take partial derivative of \eqref{eq:ascoeff} with respect to $x_0$ of order ${n-\sum \alpha_i}$, $ x_1$ of order ${\alpha_1}$, $\cdots$, and $x_k$ of order ${\alpha_k}$. Then set $x_i=0$ for $i\in [0,k]$. We see that $\widetilde{M}_{\underline{\alpha}}$ is just the sum of determinants of submatrices of ${\bm g}$ of the form given by choosing any ${n-\sum \alpha_i}$ columns of $\id_n$, any $\alpha_1$ columns of $g_1$,  any $\alpha_2$ columns of $g_1g_2$, $\ldots$, and any $\alpha_k$ columns of $g_1\cdots g_k$. The proof of Eq \eqref{eq:expansion2} is similar to the proof of Proposition \ref{pro:expofM}, which is omitted here.
\end{proof}

Together with Lindstr{\"o}m's Lemma, it is clear

\begin{proposition}\label{pro:detropofma}
    The functions$\widetilde{M}_{\underline{\alpha}}$ are positive with respect to the positive structure $\theta_{\mathrm{st}}\times \cdots \times \theta_{\mathrm{st}}$ on $B_-^k$ and $\widetilde{M}_{\underline{\alpha}}^t=m_{\underline{\alpha}}$.
\end{proposition}

Note that $M_{\underline{\alpha}}$ is a summand in $\widetilde{M}_{\underline{\alpha}}$, thus $\widetilde{M}_{\underline{\alpha}}-M_{\underline{\alpha}}$ is positive. In the case $n=1$, $M_{i}^t=\widetilde{M}_{i}^t$ under GZ condition since
\begin{lemma}\cite[Theorem 4.13, 4.14]{ABHL}\label{lem:bk=GZ}
   Consider the positive variety with potential $(B_-,\Phi_{\rm BK})$ and the positive chart $\theta_{\rm st}$. The condition $\Phi_{\rm BK}^t( w)\leqslant 0$ is equivalent to the GZ conditions on $(\Pi_{\rm st}, w)$, where $w$ is a $\T$-weighting of $\Pi_{\rm st}$.
\end{lemma}
\begin{proof}
    Recall that Theorems 4.13 and 4.14 in \cite{ABHL} state that $\Delta^t\leqslant \Delta_{[1,k]^\op, [1,k]}^t$ for any minor $\Delta$ of size $k\times k$ if and only if $\Phi_{\rm BK}^t\leqslant 0$.  By the Lindstr{\"o}m's Lemma, these are exactly the GZ conditions for $(\Pi_{\rm st}, w)$.
\end{proof}

Similar to the statement that $M_{i}^t=\widetilde{M}_{i}^t$ under GZ condition, we have
\begin{theorem}\label{thm:bk=gz}
   Consider the positive chart $\theta_{\mathrm{st}}\times \cdots \times \theta_{\mathrm{st}}$ on $B_-^k$, we have
    \[
        M_{\underline{\alpha}}^t=\widetilde{M}_{\underline{\alpha}}^t(=m_{\underline{\alpha}}) \Leftrightarrow (\Phi_{\rm BK}\circ \pr_i)^t\leqslant 0 \text{~for all~} i\in [1,k],
    \]
    where $\pr_i\colon B_-^k\to B_-$ is the projection from $B_-^k$ to the $i$-th component of $B_-^k$.
\end{theorem}
\begin{proof}
    By Lemmas \ref{lemma A(w) in GZ implies multi-GZ condition on planar networks} and \ref{lem:bk=GZ}, we know
    \[
        M_{\underline{\alpha}}^t=\widetilde{M}_{\underline{\alpha}}^t \Leftrightarrow (\Phi_{\rm BK}\circ \pr_i)^t\leqslant 0 \text{~for all~} i\in [1,k].
    \]
    Thus we get the statement by Proposition \ref{pro:detropofma}.
\end{proof}

Finally, by restricting the statement of Theorem \ref{Thm:octrec} to $B_-^3$ and by applying tropicalization, we get an alternative proof of Theorem \ref{theorem PN octahedron recurrence}. Indeed, by restricting equation \eqref{eq:geooct} to $B_-^3$, we obtain
    \begin{equation}       \label{eq:M^t_recurrence}
        M_{i+1,j,k+1}^t+M_{i,j+1,k}^t=\max\{M_{i,j+1,k+1}^t+M_{i+1,j,k}^t,M_{i,j,k+1}^t+M_{i+1,j+1,k}^t\}.
    \end{equation}
Furthermore, for a triple of standard planar networks satisfying GZ conditions, we have $M^t_{i,j,k}=m_{i,j,k}$. Then Eq \eqref{eq:M^t_recurrence} implies
\begin{proposition}\label{pro:geoocttotropoct}
    For $(\Pi_{\rm st},w_i)$, $i=1,2,3$ satisfying GZ condition, we have
    \[
        m_{i+1,j,k+1}+m_{i,j+1,k}=\max\{m_{i,j+1,k+1}+m_{i+1,j,k},m_{i,j,k+1}+m_{i+1,j+1,k}\}.
    \]
\end{proposition}

\section{Multiplicative problem}\label{section: multiplicative problem}
\begin{definition}
    The \emph{singular values} of a matrix $A\in \GL_n$ are the positive square roots of the eigenvalues of $AA^*$. We denote them by $\sigma_1\geqslant\cdots\geqslant\sigma_n$ and set $\mathrm{sing}$ to be the map $A\mapsto (\sigma_1,\ldots,\sigma_n)$.
\end{definition}

\textbf{The multiplicative multiple Horn problem:} given $k$ matrices $A_i\in \GL_n$, $i\in 1\ldots k$, what are the relations between ${\rm sing}(A_i\cdots A_j)$, $1\leqslant i\leqslant j\leqslant k$?


This problem possesses a ${\rm U}(n)^{k+1}$-symmetry, where ${\rm U}(n)$ is a group of unitary matrices.
Namely, the action of $(u_0,\ldots,u_k)$ on $(A_1,\ldots,A_k)$ sends it to $(u_0 A_1 u_1^{-1}, u_1 A_2 u_2^{-1},\ldots, u_{n-1} A_n u_n^{-1})$ and this action preserves ${\rm sing}(A_i\cdots A_j)$. 
Recall that by the Iwasawa decomposition we can identify $\GL_n/{\rm U}(n)$ with the set of lower-triangular matrices with real positive values on the diagonal, which we denote by $\B(n)$. 
Therefore, we can replace each $A_i$ by its $\B(n)$-part and obtain an equivalent problem.
The residual action of ${\rm U}(n)$ is called {\em dressing action}. 
Explicitly, it is given by $u.(b_1,\ldots,b_k)=(u b_1 u_1^{-1}, u_1 b_1 u_2^{-1},\ldots\\ u_k b_k u_{k+1}^{-1})$,
where $u_{i+1}$ is uniquely determined by $u_i b_iu_{i+1}^{-1}\in\mathcal{B}(n)$.

\subsection{The correspondence map with a parameter}\label{subsection: correspondence with parameter}

To connect tropical and classical functions (such as singular values), we need the correspondence map with a parameter. Notice that for any $s\neq 0$ the map \[\mu_s: \T\times {\rm U}(1)\to\C, \quad (x,\phi)\mapsto e^{sx}\phi\] is an epimorphism of monoids $(\T,+)\times ({\rm U}(1),\cdot)$ and $(\C,\cdot)$. Given a planar network $\Pi$ the {\em angles} of $\Pi$ is the set $\Phi_\Pi$ of maps edges($\Pi$)$\to {\rm U}(1)$. Applying $\mu_s$ to each edge, we get a map $\mu_s:\T^\Pi\times\Phi_\Pi\to\C^\Pi$.

\begin{definition}
    The correspondence map with a parameter is
    \[M_s\colon \T^\Pi\times \Phi_\Pi\to \Mat_{n\times n}(\C), \qquad (w,\phi)\mapsto M(\mu_s(w,\phi)),\]
    where $M$ is a usual correspondence map (\ref{definition: correspondence map}).
\end{definition}

The tropical singular values of $w$ and the singular values of $M_s(w,\phi)$ are related as follows:

\textit{Claim.} Denote  $A_s=M_s(w,\phi)$. There exists a constant $C=C(\Pi)$ such that \[\left|\frac{1}{s}\log\sigma_i(A_s)-\lambda_i^\T(w)\right|<\frac{C}{s}.\]
The proof is given in Appendix \ref{appendix: proof of proposition gz_s M_s close to gz^T}.



Let $\Pi_{\rm st}$ be a standard planar network (Definition \ref{definition: standard network}).    Let $\Phi$ be the subset of $\Phi_{\Pi_{\rm st}}$ of angles which differ from 1 only on the slanted edges. This set is isomorphic to ${\rm U}(1)^\frac{n(n-1)}{2}$. It is clear that if $(w,\phi)\in W_\T(\Pi_{\rm st})\times\Phi$ then $M_s(w,\phi)\in\B$.

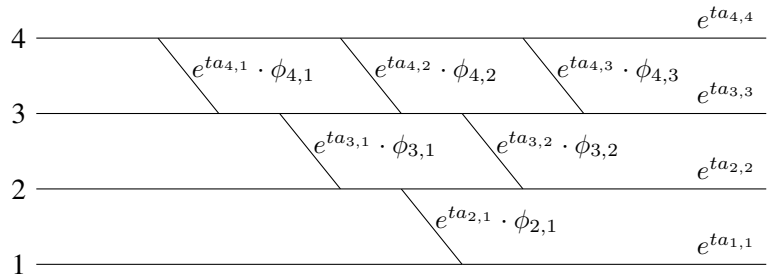
    \begin{figure}[H]
    \centering
\begin{tikzpicture}[xscale=1.6]
    \draw (0,0) node[left]{4} -- (6,0) node[above left]{\small$ e^{ta_{4,4}}$};
    \draw (0,-1) node[left]{3} -- (3.5,-1);
    \draw (3.5,-1) -- (6,-1) node[above left ]{\small$ e^{ta_{3,3}}$};
    \draw (0,-2) node[left]{2} -- (3,-2);
    \draw (3,-2) -- (4,-2);
    \draw (4,-2) -- (6,-2) node[above left]{\small$ e^{ta_{2,2}}$};
    \draw (0,-3) node[left]{1} -- (3.5,-3);
    \draw (3.5,-3) -- (6,-3) node[above left]{\small$ e^{ta_{1,1}}$};
    \draw (1,0) -- (1.5,-1) node [above, yshift=7,xshift=13]{\small$ e^{ta_{4,1}}\cdot\phi_{4,1}$};
    \draw (2.5,0) -- (3,-1) node [above, yshift=7,xshift=13]{\small$e^{ta_{4,2}}\cdot\phi_{4,2}$};
    \draw (4,0) -- (4.5,-1) node [above, yshift=7,xshift=13]{\small$ e^{ta_{4,3}}\cdot\phi_{4,3}$};
    \draw (2,-1) -- (2.5,-2) node [above, yshift=7,xshift=13]{\small$ e^{ta_{3,1}}\cdot\phi_{3,1}$};
    \draw (3.5,-1) -- (4,-2) node [above, yshift=7,xshift=13]{\small$ e^{ta_{3,2}}\cdot\phi_{3,2}$};
    \draw (3,-2) -- (3.5,-3) node [above, yshift=7,xshift=13]{\small$ e^{ta_{2,1}}\cdot\phi_{2,1}$};
\end{tikzpicture}
\end{figure}

\begin{lemma}\label{lemma: M_s is birational}
    For any $s\neq 0$ the map $M_s\colon W_\R(\Pi_{\rm st})\times\Phi\to\B$ is a birational isomorphism onto a subset given by  $\Delta_{[l-i+1,l],[1,i]}\neq 0$ for all $1\leqslant i\leqslant l\leqslant n$.
\end{lemma}
The proof is similar to Lemma \ref{lemma: GZ and standard network positive structures are equivalent}.



\subsection{Symplectic analysis}\label{subsection: symplectic analysis}

This section uses ideas of \cite{APS-symplectic}.

Let $\Pi$ be any planar network and let $\Pi^{(l)}$ denote a subnetwork of $\Pi$ where we delete all the edges containing sources or sinks with labels bigger than $l$:

\begin{figure}[H]
    \centering
\begin{tikzpicture}[scale=.8]
    \draw (1,0) -- (4,0);
    \draw (1.5,-1) -- (4.5,-1);
    \draw (0,-2) node[left]{2} -- (6,-2);
    \draw (0,-3) node[left]{1} -- (6,-3);
    \draw (3.5,-3) -- (6,-3);
    \draw (1,0) -- (1.5,-1);
    \draw (2.5,0) -- (3,-1);
    \draw (4,0) -- (4.5,-1);
    \draw (2,-1) -- (2.5,-2);
    \draw (3.5,-1) -- (4,-2);
    \draw (3,-2) -- (3.5,-3);
\end{tikzpicture}
    \caption{$\Pi_{\rm st}(4)^{(2)}$.}
\end{figure}

Notice that $M_{\Pi^{(l)}}(w)$ is a submatrix of $M_{\Pi}(w)$ concentrated in the first $l$ rows and columns. Notice that $M_{\Pi_{\rm st}(n)^{(l)}}(w)=M_{\Pi_{\rm st}(l)}(w)$, where $\Pi_{\rm st}(l)$ is viewed as a subnetwork of $\Pi_{\rm st}(n)$ in an obvious way.
\begin{definition}
    The {\em tropical Gelfand-Zeitlin map } $gz^\T\colon\T^\Pi\to\T^{\frac{n(n+1)}{2}}$ is a map with components $m_i(\Pi^{(l)})$, $0\leqslant i\leqslant l\leqslant n.$
\end{definition}
\begin{lemma}\cite[Theorem 2]{APS-planar}
    $gz^\T(w)\in\Delta_{GZ}$.
\end{lemma}

Let $\mathcal{H}$ denote the set of Hermitian $n\times n$ matrices. For $A\in \mathcal{H}$ let $A^{(l)}$ denote its submatrix concentrated in the first $l$ rows and columns. Let $\lambda_1^{(l)}\geqslant\cdots\geqslant\lambda_l^{(l)}$ be the eigenvalues of $A^{(l)}$.
\begin{definition}
    The \emph{Gelfand-Zeitlin map} $gz^\mathcal{H}\colon \mathcal{H}\to\R^{\frac{n(n+1)}{2}}$ is a map with components $A\mapsto \lambda_1^{(l)}+\cdots+\lambda_i^{(l)}$, $0\leqslant i\leqslant l\leqslant n.$
\end{definition}
\begin{lemma}(Cauchy’s Interlace Theorem) \quad 
    $gz^\mathcal{H}(A)\in\Delta_{GZ}$. 
\end{lemma}

Let $\B$, as before, denote the set of lower-triangular $n\times n$ matrices with real positive values on the diagonal. For $A\in\B$ let $\sigma_1^{(l)}\geqslant\cdots\geqslant\sigma_l^{(l)}$ denote the singular values of $A^{(l)}$.
\begin{definition}
    The map $gz_s\colon \B\to \R^{\frac{n(n+1)}{2}}$ is a map with components $A\mapsto\frac{1}{s}\log\sigma_1^{(l)}+\cdots+\frac{1}{s}\log\sigma_i^{(l)}$, $0\leqslant i\leqslant l\leqslant n.$
\end{definition}

\begin{lemma}\cite{FR} 
    $gz_s(A)\in\Delta_{GZ}$. 
\end{lemma}

    Recall that $\mathcal{H}=\mathfrak{u}(n)^*$ and $\B={\rm U}(n)^*$ are Poisson manifolds. The Poisson structure on $\mathcal{H}$ is the \emph{Kirillov-Kostant-Souriau} structure, its symplectic leaves are  $\mathcal{H}_\lambda=\{A\in \mathcal{H}\mid {\rm eig}(A)=\lambda\}$. We denote its Poisson bivector by $\pi_{\mathcal{H}}$; the corresponding symplectic measure on $\mathcal{H}_\lambda$, by $\mu^\mathcal{H}_\lambda$. The Poisson structure on $\B$ is the \emph{standard structure}, its symplectic leaves are $\B_\sigma=\{A\in\B\mid {\rm sing}(A)=\sigma\}$. We denote its Poisson bivector by $\pi_{\B}$.

Let us collect known facts about these objects which we will need.

\begin{proposition}\label{proposition: symplectic measure and gz} \cite{GS83}
    The pushforward of $\mu^\mathcal{H}_\lambda$ under the map $gz^\mathcal{H}$ is a scalar multiple of the Lebesgue measure on $\Delta_{GZ}\cap\{l_i^{(n)}-l_{i-1}^{(n)}=\lambda_i\}$.
\end{proposition}

\begin{proposition}\label{proposition: existense of GW_s} \cite{AM}
    There exists a family of Poisson isomorphisms $GW_s\colon (\mathcal{H},s\pi_\mathcal{H})\to(\B,\pi_\B)$, intertwining the maps $gz^\mathcal{H}$ and $gz_s$.
\end{proposition}
In particular, $GW_s$ maps $\mathcal{H}_\la$ to $\B_{e^{s\la}}$.

Denote by $\zeta_s=(w_s,\phi_s)\colon\B\dashrightarrow W_\T(\Pi_{\rm st})\times\Phi$ the inverse of the birational isomorphism $M_s$ from Lemma \ref{lemma: M_s is birational}.
\begin{center}
\begin{tikzcd}
    \mathcal{H} \arrow[rrr, bend left=25, "\zeta_s"] \arrow[dr, swap, "gz^\mathcal{H}"] \arrow[rr, "GW_s"]  & & \B \arrow[dl, "gz_s"] \arrow[r, dashrightarrow, shift left, "\zeta_s"]  & \R^{\frac{n(n-1)}{2}}\times\Phi \arrow[d,"gz^\T"] \arrow[l, shift left, "M_s"] \\ 
    & \Delta_{GZ} & & \Delta_{GZ} 
\end{tikzcd}
\end{center}

For $\delta>0$ denote by $\Delta_{GZ}(\delta)$ a subset in $\Delta_{GZ}$ where all the inequalities (\ref{equation: GZ inequalities}) are $\delta$-strict, {\em i.e.},\[m_i^{(l+1)}+ m_{i-1}^{(l)}\geqslant m_{i-1}^{(l+1)}+ m_{i}^{(l)}+\delta, \qquad m_i^{(l+1)}+ m_{i}^{(l)}\geqslant m_{i+1}^{(l+1)}+ m_{i-1}^{(l)}+\delta.\]

Recall the map $\mathcal{A}$ from Definition \ref{definition: standard network}. Since its matrix is upper-triangular with $1$'s on the diagonal, we have $||\mathcal{A}||=||\mathcal{A}^{-1}||=1$.

\begin{proposition}\label{proposition: existense of zeta_infty} \cite[Proposition 5.1]{ALL}.
On a dense subset of $\mathcal{H}$, given by $gz^\mathcal{H}(A)\in\cup_{\delta>0}\Delta_{GZ}(\delta)$, there exists a limit $\lim_{s\to\infty}\zeta_s=:\zeta_\infty$ which intertwines $gz^\mathcal{H}$ and $\mathcal{A}$. More precisely, for any $\delta$ there exists $C$ and $s_0$ such that $gz^\mathcal{H}(A)\in \Delta_{GZ}(\delta)$ implies $|\mathcal{A}\circ w_s-gz^\mathcal{H}|(A)\leqslant Ce^{-s\delta/2}$.
\end{proposition}
\begin{remark}\label{remark: how close w and w_s zeta_infty^-1}
    Setting $A=\zeta_\infty^{-1}(w,\phi)$ in the inequality gives: $|\mathcal{A} w_s\zeta_\infty^{-1}(w,\phi)-\mathcal{A}(w)|\leqslant Ce^{-s\delta/2}$ whenever $\mathcal{A}(w)\in\Delta_{GZ}(\delta)$. Since $\mathcal{A}^{-1}$ is a linear isomorphism of norm $1$, $|w_s\zeta_\infty^{-1}(w,\phi)-w|\leqslant Ce^{-s\delta/2}$ whenever $\mathcal{A}(w)\in\Delta_{GZ}(\delta)$.
\end{remark}

Notice that by Lemma \ref{lemma A(w) in GZ implies multi-GZ condition on planar networks} and Proposition \ref{proposition: existense of zeta_infty}, $\mathcal{A}=gz^\T$ on the image of $\zeta_\infty$.

\vspace{5pt}

Let $W_{\Pi,\delta}\subset \R^\Pi$ be given by  (a) $gz^\T(w)\in\Delta_{GZ}(\delta)$, and (b) $|w(p)-w(p')|>\delta$ for any distinct $p, p'\in P_k(\Pi^{(l)})$.  

 \begin{proposition}\label{proposition: gz_s M_s close to gz^T} \cite[Proposition 2]{APS-symplectic}.
 For any $\delta$ there exists a constant $C$ and $s_0$ such that $(w,\phi)\in W_{\Pi,\delta}\times\Phi_\Pi$ and $s\geqslant s_0$ implies $|gz_s\circ M_s-gz^\T|(w,\phi)\leqslant Ce^{-s\delta}$.
 \end{proposition}
Since the definition of the set $W_{\Pi,\delta}$ does not coincide with one used in \cite{APS-symplectic}, we prove Proposition \ref{proposition: gz_s M_s close to gz^T} in Appendix \ref{appendix: proof of proposition gz_s M_s close to gz^T}.


Introduce two maps:
\[Horn_s\colon\B^3\to \R^{6n},\qquad (A,B,C)\mapsto \big(\frac{1}{s}\sum_{i=1}^k\log{\rm sing}_i(X)\big)_{ k\leqslant n, X\in\{A,B,C,AB,BC,ABC\}}.\] 
In other words, $Horn_s$ takes $A,B,C$ to the upper rows of $gz_s(X)$ for $X\in\{A,B,C,AB,BC,ABC\}$. 

Similarly, \[Horn^\T\colon\prod_{i=1}^3\T^{\Pi_i}\times\Phi_{\Pi_1\circ\Pi_2\circ\Pi_3}\to \T^{6n},\quad (w_1,w_2,w_3,\phi)\mapsto\big(\sum_{i=1}^k \la^\T_i(w)\big)_{k\leqslant n, w\in \{w_1\ldots w_1\circ w_2\circ w_3\}}.\] 
 In other words it sends $(w_1,w_2,w_3,\phi)$ to the upper rows of $gz^\T(w)$ for $w\in \{w_1,\ldots, w_1\circ w_2\circ w_3\}$. 

Let $\Delta_{octah}\subset\R^{\Delta^k(n)}$ denote the set of $\Delta^k(n)$-tuples of reals satisfying rhombus inequalities and the octahedron recurrence (\ref{theorem PN octahedron recurrence}). Let $\partial\Delta_{octah}$ be the image of $\Delta_{octah}$ under the map (\ref{equation partial on R^tetrahedron}). 

\begin{theorem}\label{theorem: symplectic convergence}
    For any $\epsilon_1$ and $\epsilon_2$ there exists a set $U\subset \mathcal{H}_{\la(1)}\times \mathcal{H}_{\la(2)}\times \mathcal{H}_{\la(3)}$ of normalized measure at least $1-\epsilon_1$ and $s_0$ such that $\forall s\geqslant s_0$:
    \[Horn_s(GW_s^{\times 3}(U))\subset \mathcal{U}_{\epsilon_2}(\partial\Delta_{octah}).\]
\end{theorem}

\begin{proof}
    Define $W^{(3)}_\delta\subset \R^{3\cdot\frac{n(n-1)}{2}}=W(\Pi_{\rm st})^3$ as the set of those weightings $w=(w_1,w_2,w_3)\in W(\Pi_{\rm st})^3$ for which:

    (a) for $i=1,2,3$: \quad $\mathcal{A}(w_i)\in\Delta_{GZ}(\delta)$,

    (b) $gz^\T(w_1\circ w_2),\, gz^\T(w_2\circ w_3), \, gz^\T(w_1\circ w_2\circ w_3) \in\Delta_{GZ}(\delta)$,

    (c) for any distinct subsets $\alpha,\beta\subset essedges(\Pi_{\rm st})^3$: \quad $|w(\alpha)-w(\beta)|>\delta$.
    
    This set is the complement of $\delta$-neighbourhood of the union of hyperplanes given by $w(\alpha)=w(\beta)$, hyperplanes given by GZ equalities for $\mathcal{A}(w_i)$, and subsets given by GZ equalities for $gz^\T(w_1\circ w_2)$ etc., each being an intersection of a hyperplane with a cone. Denote this union by $X$.

    Let $U=(\zeta_\infty^{-1})^{\times 3} (W^{(3)}_\delta\times\Phi^3)$.

\textit{Claim 1.
    $\mu^\mathcal{H}(U)>1-\delta \cdot {\rm Vol}(X)$, where $\mu^\mathcal{H}$ and ${\rm Vol}$ are normalized so that $\mu^\mathcal{H}(\mathcal{H}_\lambda)={\rm Vol}(\Delta_{GZ,\lambda})=1$.
}

    Indeed, by Proposition \ref{proposition: symplectic measure and gz}: $\mu^\mathcal{H}(U)={\rm Vol}\big((gz^\mathcal{H})^{\times 3}(U)\big)$. Since $\zeta_\infty$ intertwines the maps $gz^\mathcal{H}$ and $gz^\T$, $(gz^\mathcal{H})^{\times 3}(U)=(gz^\T)^{\times 3}(W^{(3)}_\delta)$, so $\mu^\mathcal{H}(U) = {\rm Vol}\big((gz^\T)^{\times 3}(W^{(3)}_\delta)\big)\geqslant 1-\delta\cdot {\rm Vol}(X)$.

\textit{Claim 2. 
There exists $s_0$ such that $\forall s\geqslant s_0$:    $Horn_s((GW_s^{\times 3}(U))\subset \mathcal{U}_{C^{-s\cdot 2\delta}}(\partial\Delta_{octah})$.
}

Lemma \ref{lemma A(w) in GZ implies multi-GZ condition on planar networks} implies that the weightings  $(w_1,w_2,w_3)\in W_\delta^{(3)}$ satisfy the conditions of Theorem \ref{theorem PN octahedron recurrence}, hence $Horn^\T(W^{(3)}_\delta)\subset \partial\Delta_{octah}$.
\begin{align*}
Horn_s(GW_s^{\times 3}(U))
 &= Horn_s( (GW_s\zeta_\infty^{-1})^{\times 3}(W^{(3)}_\delta\times\Phi^3))     
& &\text{by definition of $U$}\\
 &= Horn_s( (GW_s\zeta_s^{-1}\zeta_s\zeta_\infty^{-1})^{\times 3}(W^{(3)}_\delta\times\Phi^3)) 
& &\text{trivially}\\
 &= Horn_s( (M_s\zeta_s\zeta_\infty^{-1})^{\times 3}(W^{(3)}_\delta\times\Phi^3)) 
& &\text{since }\zeta_s^{-1}=GW_s^{-1}M_s\\
 &\subset Horn_s( M_s^{\times 3}(\mathcal{U}_{Ce^{-s\delta}}(W^{(3)}_\delta)\times\Phi^3)) 
& &\text{by Remark \ref{remark: how close w and w_s zeta_infty^-1} }\\
 &\subset Horn_s( M_s^{\times 3}(W^{(3)}_{2\delta})\times\Phi^3) 
& &\text{for }s\geqslant s_1:=-\frac{2}{\delta}\log\frac{\delta}{C}\\
 &\subset \mathcal{U}_{C^{-s\cdot 2\delta}}(Horn^\T (W^{(3)}_{2\delta}\times\Phi^3)) 
& &\text{by Proposition \ref{proposition: gz_s M_s close to gz^T} }\\
 &\subset\mathcal{U}_{C^{-s\cdot 2\delta}}(\partial\Delta_{octah}). 
& &\text{by Lemma \ref{lemma A(w) in GZ implies multi-GZ condition on planar networks} and Theorem \ref{theorem PN octahedron recurrence}.}\\
\end{align*}
 Now the proposition follows by choosing $s_0\geqslant s_1$ such that $C^{-s_0\cdot 2\delta}<\epsilon_2$.
\end{proof}

\begin{appendices}

\section{The multiplicative problem for \texorpdfstring{$n=2$}{n=2}}

In this Appendix, we consider the multiplicative multiple Horn problem for $n=2$. First, recall that it is sufficient to consider lower-triangular matrices with positive reals on the diagonal:
$$
g = 
\begin{bmatrix}
u & 0 \\
v & w
\end{bmatrix},
$$
where $u, w \in \mathbb{R}_{>0}, v \in \mathbb{C}$. Furthermore, one can multiply each of the matrices $A, B$ and $C$ by a positive real number to make their determinants equal to $1$ (this corresponds to $w=u^{-1}$ in the formula above), and then the same will apply to determinants of $AB, BC$ and $ABC$. In this setup, all trace equalities are automatically satisfied. Furthermore, the ordered arrays of reals in the formulation of the multiple Horn problem are of the form $(\lambda, -\lambda)$ with $\lambda \geqslant 0$. Hence, we can think of 6-tuples $(\lambda, \mu, \nu, \rho, \sigma, \tau)$ taking values in $\mathbb{R}^6$, where we keep only the non-negative number in each array.

Next, we observe that for $n=2$ all plane rhombi lie in the faces of the tetrahedron. Since each face corresponds to an (ordinary) multiplicative  Horn problem, the rhombus inequalities are automatically satisfied for every solution of the multiplicative multiple Horn problem.

In order to deal with the tetrahedron equalities, we define a function
$f: \SU(2)^* \to \mathbb{R}_{>0}$ given by formula
$$
f(g) = {\rm Tr}(gg^*) = u^2 + u^{-2}+|v|^2.
$$
Note that $f(g)=f(g^{-1})$, and that
$$
f(g) = \Lambda^2 + \Lambda^{-2},
$$
where $\Lambda $ and $\Lambda^{-1}$ are the singular values of $g$ (that is, the eigenvalues of the positive definite Hermitian matrix $(gg^*)^{1/2}$).

\begin{proposition} \label{appendix:main_inequality}
    For all $A, B, C \in \SU(2)^*$, we have
    $$
2(f(AB)f(BC)+f(A)f(C)) \geqslant f(B)f(ABC).
    $$
\end{proposition}

\begin{proof}
    All functions in the inequality are invariant under the dressing action (defined in \ref{section: multiplicative problem}) of ${\rm U}(2)$ on triples $(A,B,C)$. Hence, without loss of generality we can choose one of these matrices to be diagonal. We use the following notation
$$
A=
\begin{bmatrix}
u & 0 \\
v & u^{-1}
\end{bmatrix},
\hskip 0.3cm
B=
\begin{bmatrix}
w & 0\\
0 & w^{-1}
\end{bmatrix},
\hskip 0.3cm
C=
\begin{bmatrix}
x & 0 \\
y & x^{-1}
\end{bmatrix}.
$$
We have,
$$
\begin{array}{lll}
f(A) & = & u^2+u^{-2}+|v|^2, \\
f(B) & = & w^2 + w^{-2}, \\
f(C) & =& x^2+x^{-2} +|y|^2 \\
f(AB) & = & u^2w^2+u^{-2}w^{-2}+w^{2}|v|^2, \\
f(BC) & = & w^2x^2+w^{-2}x^{-2}+w^{-2}|y|^2, \\
f(ABC) & = & u^2w^2x^2 +u^{-2}w^{-2}x^{-2} +|vwx+u^{-1}w^{-1}y|^2.
\end{array}
$$
First, we give an estimate from above for the right hand side of the inequality:
$$
\begin{array}{lll}
f(B)f(ABC) & = & (w^2 + w^{-2})(u^2w^2x^2 +u^{-2}w^{-2}x^{-2} +|vwx+u^{-1}w^{-1}y|^2) \\
& \leqslant & u^2w^4x^2 +u^2x^2 + u^{-2}x^{-2} +u^{-2}w^{-4}x^{-2}  \\
& + &  2|v|^2(1+w^{4})x^2 + 2u^{-2}(1+w^{-4})|y|^2,
\end{array}
$$
where we have used that  $|a+b|^2 \leqslant 2(|a|^2 + |b|^2)$. Next, we re-write the left hand side as follows:
$$
\begin{array}{lll}
f(AB)f(BC)+f(A)f(C) & = & (u^2w^2+u^{-2}w^{-2}+w^{2}|v|^2)(w^2x^2+w^{-2}x^{-2}+w^{-2}|y|^2) \\
& + & (u^2+u^{-2}+|v|^2)(x^2+x^{-2} +|y|^2) \\
& = & u^2w^4x^2 +u^2x^2 + u^{-2}x^{-2} +u^{-2}w^{-4}x^{-2} + 2(u^2x^{-2} + u^{-2}x^2) \\
& + & |v|^2(x^{2}+w^{4}x^{2}+2x^{-2}) + (u^{-2}w^{-4}+u^{-2}+2u^{2})|y|^2\\
& + & 2|v|^2|y|^2.
\end{array}
$$
We conclude the proof by a direct comparison of coefficients in front of monomials.
\end{proof}

\begin{proposition}       \label{appendix:2}
   For all $A, B, C \in \SU(2)^*$, we have
$$
\begin{array}{lll}
2(f(A)f(C) + f(B)f(ABC)) &\geqslant & f(AB)f(BC), \\
2(f(B)f(ABC)+f(AB)f(AB)) & \geqslant & f(A)f(C).
\end{array}
$$
\end{proposition}

\begin{proof}
Both inequalities follow from the inequality of Proposition \ref{appendix:main_inequality} by a change of variables. For the first, we map $A \mapsto A^{-1}, B\mapsto AB, C \mapsto C$. For the second, we map $A \mapsto (AB)^{-1}, B \mapsto A, C\mapsto BC$. In both cases, we use the fact that $f(g^{-1})=f(g)$.
\end{proof}

Let $(\lambda, \mu, \nu, \rho, \sigma, \tau) \in \mathbb{R}_{\geqslant 0}^6$ be a 6-tuple of nonnegative real numbers solving the multiplicative multiple Horn problem for the parameter $s$. We denote
$$
\alpha = \lambda + \nu, \hskip 0.3cm
\beta = \mu + \tau, \hskip 0.3cm
\gamma = \rho + \sigma.
$$

\begin{proposition}
    We have,
    $$
\alpha < {\rm max}\{\beta, \gamma\} + s^{-1}\log(16), \hskip 0.3cm
\beta < {\rm max}\{\alpha, \gamma\} + s^{-1}\log(16), \hskip 0.3cm
\gamma < {\rm max}\{\alpha, \beta\} + s^{-1}\log(16).
    $$
\end{proposition}

\begin{proof}
    We consider the inequality of Proposition \ref{appendix:main_inequality}. The right hand side can be estimated from below:
$$
f(B)f(ABC)=(e^{s\mu}+e^{-s\mu})(e^{s\tau}+e^{-s\tau}) > e^{s(\mu+\tau)}
= e^{s\beta},
$$
and the left hand side can be estimated from above:
\begin{align*}
    2(f(AB)f(AC)+f(A)f(C))&=2((e^{s\rho}+e^{-s\rho})(e^{s\sigma}+e^{-s\sigma})
+(e^{s\lambda}+e^{-s\lambda})(e^{s\nu}+e^{-s\nu}))\\
&\leqslant 16 \, e^{s\, {\rm max}\{\alpha, \gamma\}} \, .
\end{align*}
We conclude that
$$
16 \, e^{s\, {\rm max}\{\alpha, \gamma\}} > e^{s\beta},
$$
and 
$$
\beta < {\rm max}\{\alpha, \gamma\} + s^{-1} \log(16),
$$
as required. The other two inequalities are proven in the same way using the inequalities from Proposition \ref{appendix:2}.
\end{proof}

We conclude that for $n=2$ solutions of the multiplicative multiple Horn problem satisfy the trace equalities and the rhombus inequalities, and they may violate the tetrahedron equalities by $\varepsilon=s^{-1} \log(16)$. Clearly, for $s$ large enough $\varepsilon$ is arbitrarily small which proves the $n=2$ case of Conjecture \ref{conjB} stated in Section \ref{sec:background}.

\begin{proposition}
    The image of the map $m \colon \T^{(\Pi_{\rm st}(2))}\to\T^{\Delta^3(2)}$ coincides with the cone defined by rhombus inequalities and tetrahedron equality.
\end{proposition}
\begin{proof}
    Notice that adding $x_i$ to the weights of the edges of the $i$-th network which are adjacent to sinks affects the values $m_{ijk}$ as follows:
    
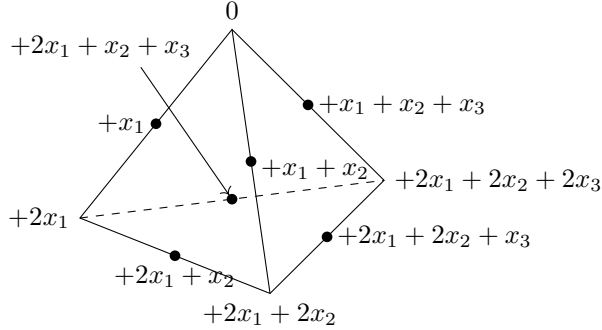
\begin{center}
        \begin{tikzpicture}
            \tikzmath{ 
            \l1=0; \l2=0; 
            \r1=4; \r2=.5;
            \u1=2; \u2=2.5;
            \d1=2.5; \d2=-1;}
        \draw (\u1,\u2) node[font=\small, above] {$ 0$} -- (\l1,\l2) node[font=\small, left] {$ +2x_1$} -- (\d1,\d2) node[font=\small, below] {$+2x_1+2x_2$} -- (\r1,\r2) node[font=\small, right] {$+2x_1+2x_2+2x_3$}  -- cycle;
        \draw[dashed] (\l1,\l2) -- (\r1,\r2);
        \draw (\u1,\u2) -- (\d1,\d2);
        \fill(.5*\d1+.5*\u1, .5*\d2+.5*\u2) circle (2pt) node[font=\small, right,yshift=-3] {$+x_1+x_2$};
        \fill(.5*\l1+.5*\r1, .5*\l2+.5*\r2) circle (2pt);
        \draw[->__] (.6*\l1+.4*\u1,.2*\l2+.8*\u2) node[font=\small, above,xshift=-15] {$+2x_1+x_2+x_3$} -- (.5*\l1+.5*\r1, .5*\l2+.5*\r2);
        \fill(.5*\l1+.5*\u1, .5*\l2+.5*\u2) circle (2pt) node[font=\small, left] {$+x_1$};
        \fill(.5*\r1+.5*\d1, .5*\r2+.5*\d2) circle (2pt) node[font=\small, right] {$+2x_1+2x_2+x_3$};
        \fill(.5*\r1+.5*\u1, .5*\r2+.5*\u2) circle (2pt) node[font=\small, right] {$+x_1+x_2+x_3$};
        \fill(.5*\l1+.5*\d1, .5*\l2+.5*\d2) circle (2pt) node[font=\small, below] {$+2x_1+x_2$};
    \end{tikzpicture}
    \end{center}
and it is easy to see that such an operation does not affect rhombus inequalities or tetrahedron equality. Therefore we can assume that $m_{200}=m_{020}=m_{002}=0$.
Notice that in this case the rhombus inequaltites are equivalent to triangle inequalities for triples of numbers  standing on a face (e.g. $|\la(1)-\la(2)|\leqslant\la(12)\leqslant\la(1)+\la(2)$). 

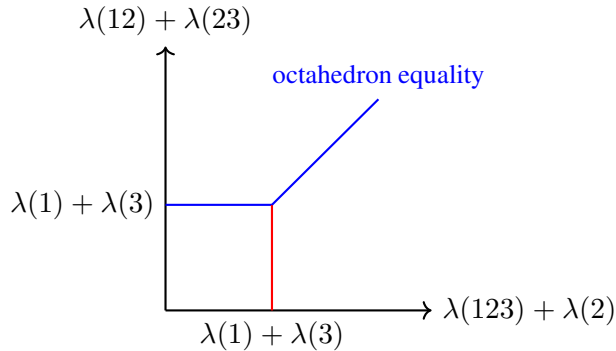
\begin{center}
    \begin{tikzpicture}[scale=0.7]
        \draw[thick,->] (0,0)--(5,0) node[right] {$\la(123)+\la(2)$};
        \draw[thick,->] (0,0)--(0,5) node[above] {$\la(12)+\la(23)$};
        \draw[thick,blue] (0,2) node[left,black] {$\la(1)+\la(3)$} --(2,2);
        \draw[thick,red] (2,0) node[below,black] {$\la(1)+\la(3)$} --(2,2);
        \draw[thick,blue] (2,2)--(4,4) node[above,font=\small]{octahedron equality};
    \end{tikzpicture}
\end{center}

In view of Theorem \ref{theorem: m is onto the octahedron recurrence cone}, it suffices to prove that the map $m$ covers the ``$\la(123)+\la(2)=\la(1)+\la(3)$-part'' of the tetrahedron equality (red). That is,  for any 5-tuple of non-negative numbers $\la(1), \la(2), \la(3),\\ \la(12), \la(23)$ which together with $\la(123):=\la(1)-\la(2)+\la(3)$ satisfy triangle inequalities and
    $\la(12)+\la(23)\leqslant \la(1)+\la(3)$,
there exists a triple of weightings with prescribed tropical singular values.

Set $x=\la(12)-\la(2)$, $z=\la(23)-\la(2)$, and consider the following weighting:

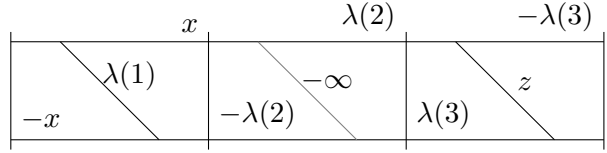
\begin{center}
    \begin{tikzpicture}[scale=1.3]
    \draw (0,0) -- (6,0);
    \draw (0,1) -- (6,1);
    \draw (0,-.1) -- (0,1.1);
    \draw (2,-.1) -- (2,1.1);
    \draw (4,-.1) -- (4,1.1);
    \draw (6,-.1) -- (6,1.1);
    \draw (.5,1) -- (1.5,0);
    \node[above] at (1.2,.4){$\la(1)$};
    \node[above left] at (2,1){$x$};
    \node[above right] at (0,0){$-x$};
    \draw[gray](2.5,1) -- (3.5,0);
    \node[above] at (3.2,.4){$-\infty$};
    \node[above left] at (4,1){$\la(2)$};
    \node[above right] at (2,0){$-\la(2)$};
    \draw (4.5,1) -- (5.5,0);
    \node[above] at (5.2,.4){$z$};
    \node[above left] at (6,1){$-\la(3)$};
    \node[above right] at (4,0){$\la(3)$};
    \end{tikzpicture}
\end{center}    

From triangle inequalities it follows that $|x|\leqslant\la(1), |z|\leqslant\la(3)$, therefore $m_{100}=\la(1),m_{110}=\la(2),m_{011}=\la(3)$.
Also $x+\la(2)=\la(12)\geqslant\la(1)-\la(2)$, therefore $m_{010}=\la(12)$; $z+\la(2)=\la(23)\geqslant\la(3)-\la(2)$, therefore $m_{101}=\la(23)$. Finally, $m_{001}=\max\{x+\la(2)+z,\la(1)-\la(2)+\la(3)\}$, and the last condition implies that $x+\la(2)+z=\la(12)+\la(23)-\la(2)\leqslant\la(1)-\la(2)+\la(3)$, hence $m_{001}=\la(1)-\la(2)+\la(3)$.
\end{proof}


\section{Proof of Proposition \ref{proposition: gz_s M_s close to gz^T}.}\label{appendix: proof of proposition gz_s M_s close to gz^T}

The proof is essentially as in \cite{APS-symplectic}, Proposition 2.

\textit{Step 1. Denote  $A_s=M_s(w,\phi)$. There exists a constant $C=C(\Pi)$ such that \[\left|\frac{1}{s}\log\sigma_i(A_s)-\lambda_i^\T(w)\right|<\frac{C}{s}.\]}

Note that $\frac{1}{s}\log\sigma_i(A_s)=\frac{1}{2s}\log\lambda_i(A_sA_s^*)$. Write a characteristic polynomial of $A_sA_s^*$: \[\chi_{A_sA_s^*}(x)=\prod_1^n(x-\lambda_i(A_sA_s^*)),\]
so the coefficient of $x^{n-k}$ is equal to $e_k$, the $k$-th elementary symmetric polynomial in $\lambda_i$'s. On the other hand, this coefficient equals
\[ \sum_{|K|=k} \Delta_{K,K}(A_sA_s^*)=\sum_{p_1,p_2:K\to K'}e^{s(w(p_1)+w(p_2))}e^{i(\phi(p_1)-\phi(p_2)}.\]

We can estimate both expressions:
\[\lambda_1\cdots\lambda_k(A_sA^*_s)\leqslant e_k\leqslant \lambda_1\cdots\lambda_k(A_sA_s^*)(1+C), \]
and 
\[e^{2s m_k(w)}\leqslant\sum_{p_1,p_2:K\to K'}e^{s (w(p_1)+w(p_2))}e^{i(\phi(p_1)-\phi(p_2)}\leqslant e^{2s m_k(w)}(1+C),\]
where  $C>0$ is some constant (say $C=\max(2^n, |P_k|^2)$). Therefore
\[\lambda_1\cdots\lambda_k(A_sA^*_s)\leqslant
e^{2sm_k(w)}(1+C), \qquad 
e^{2sm_k(w)}\leqslant \lambda_1\cdots\lambda_k(A_sA_s^*)(1+C).\]

Taking $\frac{1}{2s}\log$ in both inequalities and using $\log(1+C)<C$, we get
\[\frac{1}{2s}\log\lambda_1\cdots\lambda_k(A_sA^*_s)-\frac{1}{2s}C< m_k(w)< \frac{1}{2s}\log\lambda_1\cdots\lambda_k(A_sA_s^*)+\frac{1}{2s}C,\]
thus $\left|\frac{1}{2s}\log\lambda_i(A_sA_s^*)-\lambda^\T_i(w)\right|<\frac{C}{s}$.

\vspace{5pt}

\textit{Step 2. Suppose $w\in W_\delta$. Then there exists a constant $C$ and $s_0$ such that $|\frac{1}{s}\log\sigma_i(A_s)-\lambda_i^\T(w)|<Ce^{-s\delta}$ for any $s\geqslant s_0$.}

Since $gz^\T(w)\in\Delta_{GZ}(\delta)$, Step 1 applied to $\Pi^{(l)}$'s implies that $gz_s\circ M_s\in \Delta(\frac{\delta}{2})$ for any $s\geqslant\max_l \frac{4C(\Pi^{(l)})}{\delta}$. This means in particular that $\left|\frac{1}{2s}\log\frac{\lambda_i}{\lambda_j}\right|>\frac{\delta}{2}$, {\em i.e.}, $\frac{\lambda_j}{\lambda_i}<e^{-s\delta}$ for $j>i$.
Therefore we can refine inequalities from Step 1 as follows:
\[\lambda_1\cdots\lambda_k(A_sA^*_s)\leqslant e_k\leqslant \lambda_1\cdots\lambda_k(A_sA_s^*)(1+Ce^{-s\delta}).\]

Condition (b) in the definition of $W_\delta$ gives: 
\[e^{2sm_k(w)}\leqslant\sum_{p_1,p_2:K\to K'}e^{s(w(p_1)+w(p_2))}e^{i(\phi(p_1)-\phi(p_2)}\leqslant e^{2sm_k(w)}(1+Ce^{-s\delta}).\]

Similarly to the proof of Step 1 we get $\left|\frac{1}{2s}\log\lambda_i(A_sA_s^*)-\lambda_i^\T(w)\right|<\frac{Ce^{-s\delta}}{s}<Ce^{-s\delta}$. 

\end{appendices}

\Addresses
\end{document}